\documentclass{amsart}

\usepackage{latexsym,amsfonts,amssymb,exscale,enumerate,comment}
\usepackage{amsmath,amsthm,amscd}

\usepackage{url}
\usepackage[bookmarks=true,%
    colorlinks=true,%
    linkcolor=blue,%
    citecolor=blue,%
    filecolor=blue,%
    menucolor=blue,%
    urlcolor=blue,%
    breaklinks=true]{hyperref}

\input xy
\usepackage[all]{xy}
\xyoption{line}
\xyoption{arrow}
\xyoption{color}
\SelectTips{cm}{}

\usepackage{tikz}
\usetikzlibrary{shapes,snakes}
\usetikzlibrary{decorations.markings}
\usetikzlibrary{decorations.pathreplacing}
\tikzstyle directed=[postaction={decorate,decoration={markings,
    mark=at position #1 with {\arrow{>}}}}]

\newcommand{\hackcenter}[1]{
 \xy (0,0)*{#1}; \endxy}

\tikzset{->-/.style={decoration={
  markings,
  mark=at position #1 with {\arrow{>}}},postaction={decorate}}}

\tikzset{middlearrow/.style={
        decoration={markings,
            mark= at position 0.5 with {\arrow{#1}} ,
        },
        postaction={decorate}
    }
}

\usepackage{bm}

\def\Kom{\textsf{Kom}}
\def\tKom{\widetilde{\textsf{Kom}}}
\def\tKomm{\widetilde{\textsf{Kom}^-}}
\def\sl{\mathfrak{sl}}

\def\uk{{\underline{k}}}
\def\uy{{\underline{y}}}
\def\uz{{\underline{z}}}
\def\u0{{\underline{0}}}
\def\um{{\underline{m}}}

\def\la{\langle}
\def\ra{\rangle}
\newcommand{\sE}{\cal{E}}
\newcommand{\sF}{\cal{F}}
\newcommand{\cE}{\cal{E}}
\newcommand{\cF}{\cal{F}}

\newcommand{\onel}{\1_{\lambda}}

\newcommand{\onelp}{\1_{\lambda'}}

\def\Id{\mathrm{Id}}

\theoremstyle{plain}
\newtheorem{theorem}{Theorem}

\newtheorem{corollary}[theorem]{Corollary}
\newtheorem{proposition}[theorem]{Proposition}
\newtheorem{lemma}[theorem]{Lemma}

\theoremstyle{definition}

\newtheorem{definition}[theorem]{Definition}

\theoremstyle{definition}
\newtheorem{remark}[theorem]{Remark}

\numberwithin{equation}{section}
\numberwithin{theorem}{section}

\newcommand{\sym}{{\rm Sym}}
\newcommand{\maps}{\colon}
\newcommand{\und}[1]{\underline{#1}}

\newcommand{\refequal}[1]{\xy {\ar@{=}^{#1}
(-1,0)*{};(1,0)*{}};
\endxy}

\newcommand{\To}{\Rightarrow}

\newcommand{\Hom}{{\rm Hom}}

\renewcommand{\to}{\rightarrow}

\def\Id{\mathrm{Id}}

\def\mf{\mathfrak}

\def\Br{{\mathrm{Br}}}

\def\Ext{{\mathrm{Ext}}}

\numberwithin{equation}{section}
\def\b{$\blacktriangleright$}

\let\hat=\widehat

\let\epsilon=\varepsilon

\usepackage{bbm}
\def\C{{\mathbb{C}}}

\def\Z{{\mathbbm Z}}

\def\cal#1{\mathcal{#1}}%
\def\1{\mathbbm{1}}%
\def\nn{\notag}

\def\la{\langle}
\def\ra{\rangle}

\renewcommand{\l}{\lambda}

\def\cal#1{\mathcal{#1}}

\newcommand\nc{\newcommand}
\nc\rnc{\renewcommand}
\nc\Kar{\operatorname{Kar}}
\nc\End{\operatorname{End}}

\newcommand{\scs}{\scriptstyle}

\newcommand{\Ucat}{\cal{U}}
\newcommand{\UcatD}{\dot{\cal{U}}}

\nc\Sym{\operatorname{Sym}}

\def\cC{\mathcal{C}}
\def\cD{\mathcal{D}}
\def\cK{\mathcal{K}}
\def\cP{\mathcal{P}}
\def\tcP{\widetilde{\mathcal{P}}}
\def\bA{\mathbb{A}}
\def\bY{\mathbb{Y}}
\def\teta{\widetilde{\eta}}

\allowdisplaybreaks

\title{Curved Rickard complexes and link homologies}

\begin{document}
\setcounter{tocdepth}{1}

\author{Sabin Cautis}
\email{cautis@math.ubc.ca}
\address{Department of Mathematics \\ University of British Columbia \\ Vancouver, BC}

\author{Aaron D. Lauda}
\email{lauda@usc.edu}
\address{Department of Mathematics\\ University of Southern California \\ Los Angeles, CA}

\author{Joshua Sussan}
\email{jsussan@mec.cuny.edu}
\address{Department of Mathematics \\ CUNY Medgar Evers \\ Brooklyn, NY}
\date{\today}

\maketitle

\begin{abstract}
Rickard complexes in the context of categorified quantum groups can be used to construct braid group actions. We define and study certain natural deformations of these complexes which we call curved Rickard complexes. One application is to obtain deformations of link homologies which generalize those of Batson-Seed \cite{BSeed} and Gorsky-Hogancamp \cite{GHy} to arbitrary representations/partitions. Another is to relate the deformed homology defined algebro-geometrically in \cite{CK4} to categorified quantum groups (this was the original motivation for this paper).
\end{abstract}

\tableofcontents

\section{Introduction}

Knot homologies have had many significant applications in knot theory and low dimensional topology. Rasmussen~\cite{Ras} proved Milnor's conjecture about the slice genus of torus knots by utilizing a spectral sequence coming from the work of Lee~\cite{Lee1,Lee2}. In another direction, Kronheimer and Mrowka proved that Khovanov homology detects the unknot using a spectral from Khovanov homology to instanton Floer homology~\cite{Kron-Mrow}.

One of the main ideas from \cite{CKL-skew} is that the braiding necessary to construct a link homology can be defined, as a by-product of skew Howe duality, using the theory of Rickard complexes (or Chuang-Rouquier complexes \cite{CR}). These are complexes which live in the homotopy category of categorified quantum groups and give rise to braid group actions.

\subsection{Curved Rickard complexes}

A typical Rickard complex has the following form
\begin{equation*}
\cal{F}_i^{(\l_i)} \onel  \xrightarrow{d^+} {\sE}_i^{(1)}{\sF}_i^{(\l_i+1)} \1_{\l} \langle 1 \rangle \xrightarrow{d^+} \cdots \xrightarrow{d^+}  \cal{E}_i^{(k)}\cal{F}_i^{(\l_i+k)} \onel  \la k \ra \xrightarrow{d^+} \cdots
\end{equation*}
Motivated by the construction in \cite{GHy} we deform such complexes by adding maps back in the other direction \begin{equation}\label{eq:1}
{\sF}_i^{(\l_i)} \1_{\l}
\substack{\xrightarrow{d^+} \\ \xleftarrow[ud^-]{}}
{\sE}_i^{(1)}{\sF}_i^{(\l_i+1)} \1_{\l} \langle 1 \rangle
\substack{\xrightarrow{d^+} \\ \xleftarrow[ud^-]{}}
\cdots
\substack{\xrightarrow{d^+} \\ \xleftarrow[ud^-]{}}
{\sE}_i^{(k)} {\sF}_i^{(\l_i+k)}\1_{\l} \langle k \rangle
\substack{\xrightarrow{d^+} \\ \xleftarrow[ud^-]{}}
\cdots
\end{equation}
where $u$ is a formal parameter of homological and internal degree $[2]\langle -2 \rangle$. The fact that these are ``curved'' complexes is a consequence of a detailed computation in the categorified quantum group (c.f. Proposition \ref{prop:curved}). We then adapt some facts in \cite{GHy} to show that these complexes braid. The deformed complexes from \cite{GHy} correspond to two term Rickard complexes of the form $\1_\l \to \sE_i \sF_i \1_\l \la 1 \ra$.

\subsection{Deformed link homologies}

Starting with \cite{CaKa-sl2} an algebro-geometric construction of $\mf{sl}_m$-link homologies was obtained using certain convolution varieties in the affine Grassmannian of $PGL_m$. Subsequent work of \cite{CKL-skew,Cautis-rigid,Cautis} related this construction to categorified quantum groups.

A deformation of this construction was obtained in \cite{CK4} using the geometry of the Beilinson-Drinfeld Grassmannian. This deformed $\mf{sl}_m$-link homology generalized the deformed $\mf{sl}_2$-link homology defined earlier by Batson and Seed~\cite{BSeed}. One of the motivations of the current paper is to give an interpretation of the construction from \cite{CK4} in terms of categorified quantum groups (c.f. Section \ref{sec:comparison}).

As an application of curved Rickard complexes we obtain the following:
\begin{itemize}
\item A deformation of (coloured) $\mf{sl}_m$-link homology (Theorem \ref{thm:deflink}).
\item A deformation of (coloured) HOMFLYPT link homology (Theorem \ref{thm:defHOMFLY}).
\item A deformation of clasps in the contexts of $\mf{sl}_m$ and HOMFLYPT link homologies (Corollaries \ref{cor:slm-higher-clasp} and \ref{cor:HOMFLY-higher-clasp}).
\end{itemize}
Recall that clasps were used in \cite{Cautis} (resp. \cite{Cautis-Rem}) to obtain $\mf{sl}_m$-link homologies corresponding to arbitrary representations (resp. HOMFLYPT homologies of links labeled by arbitrary partitions). Our HOMFLYPT deformation extends that of Gorsky and Hogancamp \cite{GHy} to arbitrary partitions.

These deformations, apart from (likely) playing a fundamental role in the structure of knot homologies, also have potential applications. For instance, \cite{GHy} shows how to use these deformations to partially resolve some conjectures relating HOMFLYPT homology to the geometry of the Hilbert scheme of points on $\C^2$~\cite{GNR,OR3,OR1,OR2}.  These results depend upon a spectral sequence from HOMFLYPT homology converging to the homology of the unlinked components as well as some earlier results about the HOMFLYPT homology of torus links \cite{EH-torus, Hog-higher, Mellit}.

\subsection{Further remarks and directions}

There are two ways to view curved Rickard complexes. The first, which is the focus in this paper, is as deformations of Rickard complexes. The second is as an action of the braid group on degree two endomorphisms of the identity. We expect that this latter interpretation will have further applications in the theory of categorified quantum groups. One such application that comes to mind, but which we do not pursue in this paper, is proving a categorical analogue of the classical isomorphism between the Kac-Moody and loop presentations of quantum affine algebras.

The braid group action mentioned above actually extends to an action on higher degree endomorphisms of the identity. From the point of view of Rickard complexes this corresponds to studying what one might call ``higher degree homotopies''. In this paper we only consider the simplest possible (and by many accounts the most natural) homotopy, namely the maps $ud^-$ from (\ref{eq:1}). We hope that the study of these higher homotopies will also lead to various applications (particularly in the context of knot homologies).

\subsection{Outline of paper}
We review foundational material on categorified quantum groups and the definition of Rickard complexes in Section \ref{sec-qgroup}. In Section \ref{sec-bubbles} we prove some key technical results about so-called bubbles in the categorified quantum group. Some aspects of the general theory of curved complexes are reviewed in Section \ref{sec-curved} followed by a discussion of curved Rickard complexes. Sections \ref{sec-slm} and \ref{sec-homfly} contain our main applications: deformations of $\mf{sl}_m$-link homologies and HOMFLYPT homologies respectively.


\subsection{Acknowledgements}
S.C.~is supported by an NSERC Discovery grant. A.D.L.~is partially supported by the NSF grants DMS-1255334 and DMS-1664240. J.S.~is partially supported by the NSF grant DMS-1807161, PSC-CUNY Award  61028-00 49, and Simons Foundation Collaboration Grant 516673.  The authors are grateful to Matt Hogancamp for many illuminating discussions on $y$-ification. J.S.~ would like to thank Shotaro Makisumi for explaining his work related to \cite{GHy}.

\section{The categorified quantum group $\cal{U}_Q$} \label{sec-qgroup}

\subsection{Conventions}
By a graded category we will mean a category equipped with an
auto-equivalence $\la 1 \ra$. We denote by $\la l \ra$ the auto-equivalence
obtained by applying $\la 1 \ra$ $l$ times. If $A,B$ are two objects then
$\Hom^l(A,B)$ will be short-hand for $\Hom(A,B \la l \ra)$. A graded additive
$\Bbbk$-linear $2$-category is a category enriched over graded additive
$\Bbbk$-linear categories, that is, a $2$-category $\cal{K}$ such that the Hom
categories $\Hom_{\cal{K}}(A,B)$ between objects $A$ and $B$ are graded
additive $\Bbbk$-linear categories and the composition maps
$\Hom_{\cal{K}}(A,B) \times \Hom_{\cal{K}}(B,C) \to \Hom_{\cal{K}}(A,C)$ form a
graded additive $\Bbbk$-linear functor.

Given a $1$-morphism $A$ in an additive $2$-category $\cal{K}$, we let $\oplus_{[n]} A$ denote the direct sum $\oplus_{k=0}^{n-1} A \la n-1-2k \ra$.

Given an additive category $\cC$, we let $\Kom(\cC)$ denote the homotopy category of
complexes in $\cC$. Write $\Kom^+(\cC)$, respectively $\Kom^-(\cC)$  for the corresponding subcategory of bounded below, respectively above, complexes.  By convention, we work with \emph{cochain} complexes, so an object $(X,d)$ of $\Kom(\cC)$ is a collection of objects $X^i$ in $\cC$ together with maps
\[
\cdots \xrightarrow{d_{i-2}} X^{i-1} \xrightarrow{d_{i-1}} X^i \xrightarrow{d_i} X^{i+1} \xrightarrow{d_{i+1}} \cdots
\]
such that $d_{i+1} d_i =0$ and only finitely many of the $X^i$'s are nonzero.
A morphism $f \maps (X,d) \to (Y,d')$ in $\Kom(\cC)$ consists of
a collection of morphisms $f_i \maps X^i \to Y^i$ in $\cC$ such that $f_{i+1} d_i = d'_i f_i$ modulo null-homotopic maps.
Recall that morphisms $f,g \maps (X,d) \to (Y,d')$ in $\Kom(\cC)$ are called homotopic
if there exist morphisms $h^i \maps X^{i} \to Y^{i-1}$ such that $f_i-g_i = h^{i+1}d_i+d'_{i-1}h^i$ for all $i$.
A morphism of complexes is said to be null-homotopic if it is homotopic to the zero map.
We let $X[n]$ denote the complex obtained from $X$ by shifted each object $X^i$ down by $n$.

Given an additive $2$-category $\cal{K}$, define $\Kom^+(\cal{K})$ to be the additive $2$-category with the same objects as $\cal{K}$ and additive hom categories $\Hom_{\Kom^+(\cal{K})}(A,B) := \Kom^+(\Hom_{\cal{K}}(A,B))$. The horizontal composition in $\Kom(\cal{K})$ is given using the horizontal composition from $\cal{K}$ together with the tensor product of complexes.  The $2$-category $\Kom^-(\cal{K})$ can be defined analogously.  When no confusion is likely to arise, we often write $\Kom(\cal{K})$ in place of $\Kom^{\pm}(\cal{K})$.

%
\subsection{Categorified quantum group} \label{sec:catquantum}
%

For this article we restrict our attention to simply-laced Kac-Moody algebras. These algebras are associated to a symmetric Cartan data consisting of
\begin{itemize}
\item a free $\Z$-module $X$ (the weight lattice),
\item for $i \in I$ ($I$ is an indexing set) there are elements $\alpha_i \in X$ (simple roots) and $\Lambda_i \in X$ (fundamental weights),
\item for $i \in I$ an element $h_i \in X^\vee = \Hom_{\Z}(X,\Z)$ (simple coroots),
\item a bilinear form $(\cdot,\cdot )$ on $X$.
\end{itemize}
Write $\langle \cdot, \cdot \rangle \maps X^{\vee} \times X
\to \Z$ for the canonical pairing. This data should satisfy:
\begin{itemize}
\item $(\alpha_i, \alpha_i) = 2$ for any $i\in I$,
\item $(\alpha_i,\alpha_j) \in \{ 0, -1\}$  for $i,j\in I$ with $i \neq j$,
\item $\la i,\lambda\ra :=\langle h_i, \lambda \rangle =  (\alpha_i,\lambda)$
  for $i \in I$ and $\lambda \in X$,
\item $\langle h_j, \Lambda_i \rangle =\delta_{ij}$ for all $i,j \in I$.
\end{itemize}
Hence $(a_{ij})_{i,j\in I}$ is a symmetrizable generalized Cartan matrix, where $a_{ij}=\langle
h_i, \alpha_j \rangle=(\alpha_i, \alpha_j)$.  We will sometimes denote the bilinear pairing $(\alpha_i,\alpha_j)$ by $i \cdot j$ and abbreviate $\la i,\lambda\ra$ to $\lambda_i$.
We denote by $X^+ \subset X$ the dominant weights which are of the form $\sum_i \lambda_i \Lambda_i$ where $\lambda_i \ge 0$.

We write $W=W_{\mf{g}}$ for the Weyl group of type $\mf{g}$ and $\Br_{\mf{g}}$ for the corresponding braid group.  The Weyl group $W$ acts on the weight lattice $X$ via
\begin{equation} \label{eq:Weyl-lambda}
 s_i(\l) = \l - \alpha_i^{\vee}(\l) \l = \l - \la i, \l \ra \l
\end{equation}
for each simply transposition $s_i \in W$.

\begin{definition}
Associated to a symmetric Cartan datum, define a {\em choice of scalars $Q$} consisting of:
\begin{itemize}
  \item $\left\{ t_{ij}  \mid \text{ for all $i,j \in I$} \right\}$,
\end{itemize}
such that
\begin{itemize}
\item $t_{ii}=0$ for all $i \in I$ and $t_{ij} \in \Bbbk^{\times}$ for $i\neq j$,
 \item $t_{ij}=t_{ji}$ when $a_{ij}=0$.
\end{itemize}
\end{definition}

The choice of scalars $Q$ controls the form of the KLR algebra $R_Q$ that governs the upward oriented strands.  The $2$-category
$\Ucat_Q(\mf{g})$ is controlled by the products $v_{ij}=t_{ij}^{-1}t_{ji}$ taken
over all pairs $i,j\in I$.  When the underlying graph of the simply-laced Kac-Moody algebra $\mf{g}$ is a tree, in particular a Dynkin diagram, all choices of $Q$ lead to isomorphic KLR-algebras and these isomorphisms extend to isomorphisms of categorified quantum groups $\cal{U}_Q(\mf{g}) \to \cal{U}_{Q'}(\mf{g})$ for scalars $Q$ and $Q'$~\cite{Lau-param}.

Let $\cal{U}_Q(\mf{g})$ denote the non-cyclic form of the categorified quantum group from~\cite{CLau}.  Though a cyclic form of the categorified quantum group has been defined \cite{BHLW2}, for our purposes the non-cyclic variant is the most natural version.  By \cite[Theorem 2.1]{BHLW2} the cyclic and non-cyclic variant are isomorphic as $2$-categories.

\begin{definition}
The $2$-category  $\Ucat_Q:= \Ucat_Q(\mf{g})$ is the graded linear $2$-category consisting of:
\begin{itemize}
\item \textbf{Objects} $\lambda$ for $\lambda \in X$.
\item \textbf{$1$-morphisms} are formal direct sums of (shifts of) compositions of
$$\onel, \quad \1_{\lambda+\alpha_i} \sE_i= \1_{\lambda+\alpha_i} \sE_i\onel, \quad \text{ and }\quad
\1_{\lambda-\alpha_i} \sF_i= \1_{\lambda-\alpha_i} \sF_i\onel$$
for $i \in I$ and $\lambda \in X$.  We denote the grading shift by $\la 1 \ra$, so that for each $1$-morphism $x$ in $\cal{U}_Q$ and $t\in \Z$ we a $1$-morphism $x\la t\ra$.

\item \textbf{$2$-morphisms} are $\Bbbk$-vector spaces spanned by compositions of coloured, decorated tangle-like diagrams illustrated below.
\begin{align}
\hackcenter{\begin{tikzpicture}[scale=0.8]
    \draw[thick, ->] (0,0) -- (0,1.5)
        node[pos=.5, shape=coordinate](DOT){};
    \filldraw  (DOT) circle (2.5pt);
    \node at (-.85,.85) {\tiny $\lambda +\alpha_i$};
    \node at (.5,.85) {\tiny $\lambda$};
    \node at (-.2,.1) {\tiny $i$};
\end{tikzpicture}} &\maps \cal{E}_i\onel \to \cal{E}_i\onel \la i\cdot i \ra  & \quad
 &
  \hackcenter{\begin{tikzpicture}[scale=0.8]
    \draw[thick, ->] (0,0) .. controls (0,.5) and (.75,.5) .. (.75,1.0);
    \draw[thick, ->] (.75,0) .. controls (.75,.5) and (0,.5) .. (0,1.0);
    \node at (1.1,.55) {\tiny $\lambda$};
    \node at (-.2,.1) {\tiny $i$};
    \node at (.95,.1) {\tiny $j$};
\end{tikzpicture}} \;\;\maps \cal{E}_i\cal{E}_j\onel  \to \cal{E}_j\cal{E}_i\onel\la -i\cdot j \ra
 \nn \smallskip\\
\hackcenter{\begin{tikzpicture}[scale=0.8]
    \draw[thick, <-] (.75,2) .. controls ++(0,-.75) and ++(0,-.75) .. (0,2);
    \node at (.4,1.2) {\tiny $\lambda$};
    \node at (-.2,1.9) {\tiny $i$};
\end{tikzpicture}} \;\; &\maps \onel  \to \cal{F}_i\cal{E}_i\onel\la  1 + {\lambda}_i  \ra   &
    &
\hackcenter{\begin{tikzpicture}[scale=0.8]
    \draw[thick, ->] (.75,2) .. controls ++(0,-.75) and ++(0,-.75) .. (0,2);
    \node at (.4,1.2) {\tiny $\lambda$};
    \node at (.95,1.9) {\tiny $i$};
\end{tikzpicture}} \;\; \maps \onel  \to\cal{E}_i\cal{F}_i\onel\la  1 - {\lambda}_i  \ra   \nn \smallskip \\
\hackcenter{\begin{tikzpicture}[scale=0.8]
    \draw[thick, ->] (.75,-2) .. controls ++(0,.75) and ++(0,.75) .. (0,-2);
    \node at (.4,-1.2) {\tiny $\lambda$};
    \node at (.95,-1.9) {\tiny $i$};
\end{tikzpicture}} \;\; & \maps \cal{F}_i\cal{E}_i\onel \to\onel\la  1 + {\lambda}_i  \ra   &
    &
\hackcenter{\begin{tikzpicture}[scale=0.8]
    \draw[thick, <-] (.75,-2) .. controls ++(0,.75) and ++(0,.75) .. (0,-2);
    \node at (.4,-1.2) {\tiny $\lambda$};
    \node at (-.2,-1.9) {\tiny $i$};
\end{tikzpicture}} \;\;\maps\cal{E}_i\cal{F}_i\onel  \to\onel\la  1 - {\lambda}_i  \ra  \nn
\end{align}
\end{itemize}
In this $2$-category (and those throughout the paper) we
read diagrams from bottom to top and right to left.
That is, in a diagram representing a $1$-morphism from $\lambda$ to $\mu$, the region on the right will be labeled $\lambda$ and the region on the left will be labeled $\mu$.
The identity $2$-morphism of the $1$-morphism
$\cal{E}_i \onel$ is
represented by an upward oriented line labeled by $i$ and the identity $2$-morphism of $\cal{F}_i \onel$ is
represented by a downward such line.

The $2$-morphisms satisfy the following relations:
\begin{enumerate}
\item \label{item_cycbiadjoint-cyc} The $1$-morphisms $\cal{E}_i \onel$ and $\cal{F}_i \onel$ are biadjoint (up to a specified degree shift). 

  \item The dot $2$-morphisms are cyclic with respect to this biadjoint structure.
\begin{equation}\label{eq_cyclic_dot-cyc}
\hackcenter{\begin{tikzpicture}[scale=0.8]
    \draw[thick, ->]  (0,.4) .. controls ++(0,.6) and ++(0,.6) .. (-.75,.4) to (-.75,-1);
    \draw[thick, <-](0,.4) to (0,-.4) .. controls ++(0,-.6) and ++(0,-.6) .. (.75,-.4) to (.75,1);
    \filldraw  (0,-.2) circle (2.5pt);
    \node at (-1,.9) { $\lambda$};
    \node at (.95,.8) {\tiny $i$};
\end{tikzpicture}}
\;\; = \;\;
\hackcenter{\begin{tikzpicture}[scale=0.8]
    \draw[thick, <-]  (0,-1) to (0,1);
    \node at (.8,-.4) { $\lambda+\alpha_i$};
    \node at (-.5,-.4) { $\lambda$};
    \filldraw  (0,.2) circle (2.5pt);
    \node at (-.2,.8) {\tiny $i$};
\end{tikzpicture}}
\;\; = \;\;
\hackcenter{\begin{tikzpicture}[scale=0.8]
    \draw[thick, ->]  (0,.4) .. controls ++(0,.6) and ++(0,.6) .. (.75,.4) to (.75,-1);
    \draw[thick, <-](0,.4) to (0,-.4) .. controls ++(0,-.6) and ++(0,-.6) .. (-.75,-.4) to (-.75,1);
    \filldraw  (0,-.2) circle (2.5pt);
    \node at (1.3,.9) { $\lambda + \alpha_i$};
    \node at (-.95,.8) {\tiny $i$};
\end{tikzpicture}}
\end{equation}

The $Q$-cyclic relations for crossings are given by
\begin{equation} \label{eq_cyclic}
\hackcenter{
\begin{tikzpicture}[scale=0.8]
    \draw[thick, <-] (0,0) .. controls (0,.5) and (.75,.5) .. (.75,1.0);
    \draw[thick, <-] (.75,0) .. controls (.75,.5) and (0,.5) .. (0,1.0);
    \node at (1.1,.65) { $\lambda$};
    \node at (-.2,.1) {\tiny $i$};
    \node at (.95,.1) {\tiny $j$};
\end{tikzpicture}}
\;\; := \;\; t_{ij}^{-1}
\hackcenter{\begin{tikzpicture}[scale=0.7]
    \draw[thick, ->] (0,0) .. controls (0,.5) and (.75,.5) .. (.75,1.0);
    \draw[thick, ->] (.75,0) .. controls (.75,.5) and (0,.5) .. (0,1.0);
    \draw[thick] (0,0) .. controls ++(0,-.4) and ++(0,-.4) .. (-.75,0) to (-.75,2);
    \draw[thick] (.75,0) .. controls ++(0,-1.2) and ++(0,-1.2) .. (-1.5,0) to (-1.55,2);
    \draw[thick, ->] (.75,1.0) .. controls ++(0,.4) and ++(0,.4) .. (1.5,1.0) to (1.5,-1);
    \draw[thick, ->] (0,1.0) .. controls ++(0,1.2) and ++(0,1.2) .. (2.25,1.0) to (2.25,-1);
    \node at (-.35,.75) {  $\lambda$};
    \node at (1.3,-.7) {\tiny $i$};
    \node at (2.05,-.7) {\tiny $j$};
    \node at (-.9,1.7) {\tiny $i$};
    \node at (-1.7,1.7) {\tiny $j$};
\end{tikzpicture}}
\quad = \quad t_{ji}^{-1}
\hackcenter{\begin{tikzpicture}[xscale=-1.0, scale=0.7]
    \draw[thick, ->] (0,0) .. controls (0,.5) and (.75,.5) .. (.75,1.0);
    \draw[thick, ->] (.75,0) .. controls (.75,.5) and (0,.5) .. (0,1.0);
    \draw[thick] (0,0) .. controls ++(0,-.4) and ++(0,-.4) .. (-.75,0) to (-.75,2);
    \draw[thick] (.75,0) .. controls ++(0,-1.2) and ++(0,-1.2) .. (-1.5,0) to (-1.55,2);
    \draw[thick, ->] (.75,1.0) .. controls ++(0,.4) and ++(0,.4) .. (1.5,1.0) to (1.5,-1);
    \draw[thick, ->] (0,1.0) .. controls ++(0,1.2) and ++(0,1.2) .. (2.25,1.0) to (2.25,-1);
    \node at (1.2,.75) {  $\lambda$};
    \node at (1.3,-.7) {\tiny $j$};
    \node at (2.05,-.7) {\tiny $i$};
    \node at (-.9,1.7) {\tiny $j$};
    \node at (-1.7,1.7) {\tiny $i$};
\end{tikzpicture}} .
\end{equation}

Sideways crossings are equivalently defined by the following identities:
\begin{equation} \label{eq_crossl-gen-cyc}
\hackcenter{
\begin{tikzpicture}[scale=0.8]
    \draw[thick, ->] (0,0) .. controls (0,.5) and (.75,.5) .. (.75,1.0);
    \draw[thick, <-] (.75,0) .. controls (.75,.5) and (0,.5) .. (0,1.0);
    \node at (1.1,.65) { $\lambda$};
    \node at (-.2,.1) {\tiny $i$};
    \node at (.95,.1) {\tiny $j$};
\end{tikzpicture}}
\;\; := \;\;
\hackcenter{\begin{tikzpicture}[scale=0.7]
    \draw[thick, ->] (0,0) .. controls (0,.5) and (.75,.5) .. (.75,1.0);
    \draw[thick, ->] (.75,-.5) to (.75,0) .. controls (.75,.5) and (0,.5) .. (0,1.0) to (0,1.5);
    \draw[thick] (0,0) .. controls ++(0,-.4) and ++(0,-.4) .. (-.75,0) to (-.75,1.5);
    \draw[thick, ->] (.75,1.0) .. controls ++(0,.4) and ++(0,.4) .. (1.5,1.0) to (1.5,-.5);
    \node at (1.85,.55) {  $\lambda$};
    \node at (1.75,-.2) {\tiny $j$};
    \node at (.55,-.2) {\tiny $i$};
    \node at (-.9,1.2) {\tiny $j$};
    \node at (.25,1.2) {\tiny $i$};
\end{tikzpicture}}
 \qquad \quad
 \hackcenter{
\begin{tikzpicture}[scale=0.8]
    \draw[thick, <-] (0,0) .. controls (0,.5) and (.75,.5) .. (.75,1.0);
    \draw[thick, ->] (.75,0) .. controls (.75,.5) and (0,.5) .. (0,1.0);
    \node at (1.1,.65) { $\lambda$};
    \node at (-.2,.1) {\tiny $i$};
    \node at (.95,.1) {\tiny $j$};
\end{tikzpicture}}
\;\; = \;\;
\hackcenter{\begin{tikzpicture}[xscale=-1.0, scale=0.7]
    \draw[thick, ->] (0,0) .. controls (0,.5) and (.75,.5) .. (.75,1.0);
    \draw[thick, ->] (.75,-.5) to (.75,0) .. controls (.75,.5) and (0,.5) .. (0,1.0) to (0,1.5);
    \draw[thick] (0,0) .. controls ++(0,-.4) and ++(0,-.4) .. (-.75,0) to (-.75,1.5);
    \draw[thick, ->] (.75,1.0) .. controls ++(0,.4) and ++(0,.4) .. (1.5,1.0) to (1.5,-.5);
    \node at (-1.1,.55) {  $\lambda$};
    \node at (1.75,-.2) {\tiny $i$};
    \node at (1,-.2) {\tiny $j$};
    \node at (-.9,1.2) {\tiny $i$};
    \node at (.25,1.2) {\tiny $j$};
\end{tikzpicture}}
\end{equation}

\item The $\cal{E}$'s (respectively $\cal{F}$'s) carry an action of the KLR algebra for a fixed choice of parameters $Q$.
The KLR algebra $R$ associated to a fixed set of parameters $Q$ is defined by finite $\Bbbk$-linear combinations of braid--like diagrams in the plane, where each strand is labeled by a vertex $i \in I$.  Strands can intersect and can carry dots, but triple intersections are not allowed.  Diagrams are considered up to planar isotopy that do not change the combinatorial type of the diagram. We recall the local relations.

\begin{enumerate}[i)]

\item The quadratic KLR relations are
\begin{equation}
\hackcenter{
\begin{tikzpicture}[scale=0.8]
    \draw[thick, ->] (0,0) .. controls ++(0,.5) and ++(0,-.4) .. (.75,.8) .. controls ++(0,.4) and ++(0,-.5) .. (0,1.6);
    \draw[thick, ->] (.75,0) .. controls ++(0,.5) and ++(0,-.4) .. (0,.8) .. controls ++(0,.4) and ++(0,-.5) .. (.75,1.6);
    \node at (1.1,1.25) { $\lambda$};
    \node at (-.2,.1) {\tiny $i$};
    \node at (.95,.1) {\tiny $j$};
\end{tikzpicture}}
 \qquad = \qquad
 \left\{
 \begin{array}{ccc}
     t_{ij}\;
     \hackcenter{
\begin{tikzpicture}[scale=0.8]
    \draw[thick, ->] (0,0) to (0,1.6);
    \draw[thick, ->] (.75,0) to (.75,1.6);
    \node at (1.1,1.25) { $\lambda$};
    \node at (-.2,.1) {\tiny $i$};
    \node at (.95,.1) {\tiny $j$};
\end{tikzpicture}}&  &  \text{if $(\alpha_i,\alpha_j)=0$ or $(\alpha_i,\alpha_j)=2$,}\\ \\
  t_{ij}
  \;      \hackcenter{
\begin{tikzpicture}[scale=0.8]
    \draw[thick, ->] (0,0) to (0,1.6);
    \draw[thick, ->] (.75,0) to (.75,1.6);
    \node at (1.1,1.25) { $\lambda$}; \filldraw  (0,.8) circle (2.75pt);
    \node at (-.2,.1) {\tiny $i$};
    \node at (.95,.1) {\tiny $j$};
\end{tikzpicture}}
  \;\; + \;\; t_{ji} \;
 \hackcenter{
\begin{tikzpicture}[scale=0.8]
    \draw[thick, ->] (0,0) to (0,1.6);
    \draw[thick, ->] (.75,0) to (.75,1.6);
    \node at (1.1,1.25) { $\lambda$}; \filldraw  (.75,.8) circle (2.75pt);
    \node at (-.2,.1) {\tiny $i$};
    \node at (.95,.1) {\tiny $j$};
\end{tikzpicture}}
   &  & \text{if $(\alpha_i,\alpha_j)=-1$.}
 \end{array}
 \right. \label{eq_r2_ij-gen-cyc}
\end{equation}

\item The dot sliding relations are
\begin{align}
\hackcenter{\begin{tikzpicture}[scale=0.8]
    \draw[thick, ->] (0,0) .. controls ++(0,.55) and ++(0,-.5) .. (.75,1)
        node[pos=.25, shape=coordinate](DOT){};
    \draw[thick, ->] (.75,0) .. controls ++(0,.5) and ++(0,-.5) .. (0,1);
    \filldraw  (DOT) circle (2.5pt);
    \node at (-.2,.15) {\tiny $i$};
    \node at (.95,.15) {\tiny $j$};
\end{tikzpicture}}
\;\; - \;\;
\hackcenter{\begin{tikzpicture}[scale=0.8]
    \draw[thick, ->] (0,0) .. controls ++(0,.55) and ++(0,-.5) .. (.75,1)
        node[pos=.75, shape=coordinate](DOT){};
    \draw[thick, ->] (.75,0) .. controls ++(0,.5) and ++(0,-.5) .. (0,1);
    \filldraw  (DOT) circle (2.5pt);
    \node at (-.2,.15) {\tiny $i$};
    \node at (.95,.15) {\tiny $j$};
\end{tikzpicture}}
\;\; = \;\;
\hackcenter{\begin{tikzpicture}[scale=0.8]
    \draw[thick, ->] (0,0) .. controls ++(0,.55) and ++(0,-.5) .. (.75,1);
    \draw[thick, ->] (.75,0) .. controls ++(0,.5) and ++(0,-.5) .. (0,1) node[pos=.75, shape=coordinate](DOT){};
    \filldraw  (DOT) circle (2.5pt);
    \node at (-.2,.15) {\tiny $i$};
    \node at (.95,.15) {\tiny $j$};
\end{tikzpicture}}
\;\; - \;\;
\hackcenter{\begin{tikzpicture}[scale=0.8]
    \draw[thick, ->] (0,0) .. controls ++(0,.55) and ++(0,-.5) .. (.75,1);
    \draw[thick, ->] (.75,0) .. controls ++(0,.5) and ++(0,-.5) .. (0,1) node[pos=.25, shape=coordinate](DOT){};
    \filldraw  (DOT) circle (2.5pt);
    \node at (-.2,.15) {\tiny $i$};
    \node at (.95,.15) {\tiny $j$};
\end{tikzpicture}}
 \;\; = \;\;
 \delta_{i,j}
\hackcenter{\begin{tikzpicture}[scale=0.8]
    \draw[thick, ->] (0,0) to  (0,1);
    \draw[thick, ->] (.75,0)to (.75,1) ;
    \node at (-.2,.15) {\tiny $i$};
    \node at (.95,.15) {\tiny $i$};
\end{tikzpicture}}.
\end{align}

\item The cubic KLR relations are
\begin{equation}
\hackcenter{\begin{tikzpicture}[scale=0.8]
    \draw[thick, ->] (0,0) .. controls ++(0,1) and ++(0,-1) .. (1.5,2);
    \draw[thick, ] (.75,0) .. controls ++(0,.5) and ++(0,-.5) .. (0,1);
    \draw[thick, ->] (0,1) .. controls ++(0,.5) and ++(0,-.5) .. (0.75,2);
    \draw[thick, ->] (1.5,0) .. controls ++(0,1) and ++(0,-1) .. (0,2);
    \node at (-.2,.15) {\tiny $i$};
    \node at (.95,.15) {\tiny $j$};
    \node at (1.75,.15) {\tiny $k$};
\end{tikzpicture}}
\;\;- \;\;
\hackcenter{\begin{tikzpicture}[scale=0.8]
    \draw[thick, ->] (0,0) .. controls ++(0,1) and ++(0,-1) .. (1.5,2);
    \draw[thick, ] (.75,0) .. controls ++(0,.5) and ++(0,-.5) .. (1.5,1);
    \draw[thick, ->] (1.5,1) .. controls ++(0,.5) and ++(0,-.5) .. (0.75,2);
    \draw[thick, ->] (1.5,0) .. controls ++(0,1) and ++(0,-1) .. (0,2);
    \node at (-.2,.15) {\tiny $i$};
    \node at (.95,.15) {\tiny $j$};
    \node at (1.75,.15) {\tiny $k$};
\end{tikzpicture}}
\;\; = \;\;   -(\alpha_i,\alpha_j) \; \delta_{i,k} \; t_{ij}
\hackcenter{\begin{tikzpicture}[scale=0.8]
    \draw[thick, ->] (0,0) to (0,2);
    \draw[thick, -> ] (.75,0) to (0.75,2);
    \draw[thick, ->] (1.5,0) to (1.5,2);
    \node at (-.2,.15) {\tiny $i$};
    \node at (.95,.15) {\tiny $j$};
    \node at (1.75,.15) {\tiny $i$};
\end{tikzpicture}}.
\end{equation}
\end{enumerate}

\item When $i \ne j$ one has the mixed relations  relating $\cal{E}_i \cal{F}_j$ and $\cal{F}_j \cal{E}_i$
\begin{equation}  \label{mixed_rel-cyc}
 \hackcenter{\begin{tikzpicture}[scale=0.8]
    \draw[thick,<-] (0,0) .. controls ++(0,.5) and ++(0,-.5) .. (.75,1);
    \draw[thick] (.75,0) .. controls ++(0,.5) and ++(0,-.5) .. (0,1);
    \draw[thick, ->] (0,1 ) .. controls ++(0,.5) and ++(0,-.5) .. (.75,2);
    \draw[thick] (.75,1) .. controls ++(0,.5) and ++(0,-.5) .. (0,2);
        \node at (-.2,.15) {\tiny $i$};
    \node at (.95,.15) {\tiny $j$};
\end{tikzpicture}}
\;\; = \;\; t_{ij}
\hackcenter{\begin{tikzpicture}[scale=0.8]
    \draw[thick, <-] (0,0) -- (0,2);
    \draw[thick, ->] (.75,0) -- (.75,2);
     \node at (-.2,.2) {\tiny $i$};
    \node at (.95,.2) {\tiny $j$};
\end{tikzpicture}}
\qquad \qquad
 \hackcenter{\begin{tikzpicture}[scale=0.8]
    \draw[thick] (0,0) .. controls ++(0,.5) and ++(0,-.5) .. (.75,1);
    \draw[thick, <-] (.75,0) .. controls ++(0,.5) and ++(0,-.5) .. (0,1);
    \draw[thick] (0,1 ) .. controls ++(0,.5) and ++(0,-.5) .. (.75,2);
    \draw[thick, ->] (.75,1) .. controls ++(0,.5) and ++(0,-.5) .. (0,2);
        \node at (-.2,.15) {\tiny $i$};
    \node at (.95,.15) {\tiny $j$};
\end{tikzpicture}}
\;\; = \;\; t_{ji}
\hackcenter{\begin{tikzpicture}[scale=0.8]
    \draw[thick, ->] (0,0) -- (0,2);
    \draw[thick, <-] (.75,0) -- (.75,2);
     \node at (-.2,.2) {\tiny $i$};
    \node at (.95,.2) {\tiny $j$};
\end{tikzpicture}} .
\end{equation}

\item Negative degree bubbles are zero.  That is for all $m \in \Z_{>0}$ one has
\begin{equation}
 \hackcenter{ \begin{tikzpicture} [scale=.8]
 \draw (-.15,.35) node { $\scs i$};
 \draw[ ]  (0,0) arc (180:360:0.5cm) [thick];
 \draw[<- ](1,0) arc (0:180:0.5cm) [thick];
\filldraw  [black] (.1,-.25) circle (2.5pt);
 \node at (-.2,-.5) {\tiny $m$};
 \node at (1.15,.8) { $\lambda  $};
\end{tikzpicture} } \;  = 0\quad \text{if $m < \l_i -1$}, \qquad \quad
\;
\hackcenter{ \begin{tikzpicture} [scale=.8]
 \draw (-.15,.35) node { $\scs i$};
 \draw  (0,0) arc (180:360:0.5cm) [thick];
 \draw[->](1,0) arc (0:180:0.5cm) [thick];
\filldraw  [black] (.9,-.25) circle (2.5pt);
 \node at (1,-.5) {\tiny $m$};
 \node at (1.15,.8) { $\lambda $};
\end{tikzpicture} } \;  = 0 \quad  \text{if $m < -\l_i -1$}.
\end{equation}
Furthermore, dotted bubbles of degree zero are scalar multiples of the identity $2$-morphisms
\begin{equation} \label{eq:degreezero}
 \hackcenter{ \begin{tikzpicture} [scale=.8]
 \draw (-.15,.35) node { $\scs i$};
 \draw  (0,0) arc (180:360:0.5cm) [thick];
 \draw[,<-](1,0) arc (0:180:0.5cm) [thick];
\filldraw  [black] (.1,-.25) circle (2.5pt);
 \node at (-.5,-.5) {\tiny $\l_i -1$};
 \node at (1.15,1) { $\lambda  $};
\end{tikzpicture} }
\;\; =
\Id_{\1_{\l}}
\quad \text{for $  \l_i \geq 1$}, \qquad \quad
\;
\hackcenter{ \begin{tikzpicture} [scale=.8]
 \draw (-.15,.35) node { $\scs i$};
 \draw  (0,0) arc (180:360:0.5cm) [thick];
 \draw[->](1,0) arc (0:180:0.5cm) [thick];
\filldraw  [black] (.9,-.25) circle (2.5pt);
 \node at (1.35,-.5) {\tiny $-\l_i -1$};
 \node at (1.15,1) { $\lambda $};
\end{tikzpicture} }
\;\; =
\Id_{\1_{\l}} \quad \text{if $   \l_i \leq -1$}.
\end{equation}
 We introduce formal symbols called \emph{fake bubbles}.  These are positive degree endomorphisms of $\onel$ that carry a formal label by a negative number of dots.

\begin{itemize}
  \item Degree zero fake bubbles are normalized by
  \begin{equation}
 \hackcenter{ \begin{tikzpicture} [scale=.8]
 \draw (-.15,.35) node { $\scs i$};
 \draw  (0,0) arc (180:360:0.5cm) [thick];
 \draw[,<-](1,0) arc (0:180:0.5cm) [thick];
\filldraw  [black] (.1,-.25) circle (2.5pt);
 \node at (-.5,-.55) {\tiny $\l_i -1$};
 \node at (1.15,1) { $\lambda  $};
\end{tikzpicture} } \;\; =  \Id_{\1_{\l}}
\quad \text{for $  \l_i < 1$}, \qquad \quad
\;
\hackcenter{ \begin{tikzpicture} [scale=.8]
 \draw (-.15,.35) node { $\scs i$};
 \draw  (0,0) arc (180:360:0.5cm) [thick];
 \draw[->](1,0) arc (0:180:0.5cm) [thick];
\filldraw  [black] (.9,-.25) circle (2.5pt);
 \node at (1.35,-.5) {\tiny $-\l_i -1$};
 \node at (1.15,1) { $\lambda $};
\end{tikzpicture} } \;\; =
\Id_{\1_{\l}} \quad \text{if $   \l_i > -1$}.
\end{equation}

\item Higher degree fake bubbles for $\l_i <0$ are defined inductively as
\begin{equation}
  \hackcenter{ \begin{tikzpicture} [scale=.8]
 \draw (-.15,.35) node { $\scs i$};
 \draw  (0,0) arc (180:360:0.5cm) [thick];
 \draw[,<-](1,0) arc (0:180:0.5cm) [thick];
\filldraw  [black] (.1,-.25) circle (2.5pt);
 \node at (-.65,-.55) {\tiny $\l_i -1+j$};
 \node at (1.15,1) { $\lambda  $};
\end{tikzpicture} } \;\; = \;\;
\left\{
  \begin{array}{ll}
    -
\displaystyle \sum_{\stackrel{\scs x+y=j}{\scs y\geq 1}} \hackcenter{ \begin{tikzpicture}[scale=.8]
 \draw (-.15,.35) node { $\scs i$};
 \draw  (0,0) arc (180:360:0.5cm) [thick];
 \draw[,<-](1,0) arc (0:180:0.5cm) [thick];
\filldraw  [black] (.1,-.25) circle (2.5pt);
 \node at (-.35,-.45) {\tiny $\overset{\l_i-1}{+x}$};
 \node at (.85,1) { $\lambda$};
\end{tikzpicture}  \;\;
\begin{tikzpicture}[scale=.8]
 \draw (-.15,.35) node { $\scs i$};
 \draw  (0,0) arc (180:360:0.5cm) [thick];
 \draw[->](1,0) arc (0:180:0.5cm) [thick];
\filldraw  [black] (.9,-.25) circle (2.5pt);
 \node at (1.45,-.5) {\tiny $\overset{-\l_i-1}{+y}$};
 \node at (1.15,1.1) { $\;$};
\end{tikzpicture}  }
  & \hbox{if $0 < j < -\l_i+1$;} \\
    0, & \hbox{if $j<0$.}
  \end{array}
\right.
\end{equation}

\item Higher degree fake bubbles for $\l_i >0$ are defined inductively by
\begin{equation}
\hackcenter{ \begin{tikzpicture} [scale=.8]
 \draw (-.15,.35) node { $\scs i$};
 \draw  (0,0) arc (180:360:0.5cm) [thick];
 \draw[->](1,0) arc (0:180:0.5cm) [thick];
\filldraw  [black] (.9,-.25) circle (2.5pt);
 \node at (1.5,-.5) {\tiny $-\l_i -1+j$};
 \node at (1.15,1) { $\lambda $};
\end{tikzpicture} } \;\; = \;\;
\left\{
  \begin{array}{ll}
    -
\displaystyle\sum_{\stackrel{\scs x+y=j}{\scs x\geq 1}} \hackcenter{ \begin{tikzpicture}[scale=.8]
 \draw (-.15,.35) node { $\scs i$};
 \draw  (0,0) arc (180:360:0.5cm) [thick];
 \draw[,<-](1,0) arc (0:180:0.5cm) [thick];
\filldraw  [black] (.1,-.25) circle (2.5pt);
 \node at (-.35,-.45) {\tiny $\overset{\l_i-1}{+x}$};
 \node at (.85,1) { $\lambda$};
\end{tikzpicture}  \;\;
\begin{tikzpicture}[scale=.8]
 \draw (.3,.125) node {};
 \draw  (0,0) arc (180:360:0.5cm) [thick];
 \draw[->](1,0) arc (0:180:0.5cm) [thick];
\filldraw  [black] (.9,-.25) circle (2.5pt);
 \node at (1.45,-.5) {\tiny $\overset{-\l_i-1}{+y}$};
 \node at (1.15,1.1) { $\;$};
\end{tikzpicture}  }
  & \hbox{if $0 < j < \l_i+1$;} \\
    0, & \hbox{if $j<0$.}
  \end{array}
\right.
\end{equation}
\end{itemize}
The above relations are sometimes referred to as the \textit{infinite Grassmannian relations}.

\item The $\mf{sl}_2$ relations (which we also refer to as the $\sE \sF$ and $\sF \sE$ decompositions) are:
\begin{equation}
\begin{split}
 \hackcenter{\begin{tikzpicture}[scale=0.8]
    \draw[thick] (0,0) .. controls ++(0,.5) and ++(0,-.5) .. (.75,1);
    \draw[thick,<-] (.75,0) .. controls ++(0,.5) and ++(0,-.5) .. (0,1);
    \draw[thick] (0,1 ) .. controls ++(0,.5) and ++(0,-.5) .. (.75,2);
    \draw[thick, ->] (.75,1) .. controls ++(0,.5) and ++(0,-.5) .. (0,2);
        \node at (-.2,.15) {\tiny $i$};
    \node at (.95,.15) {\tiny $i$};
     \node at (1.1,1.44) { $\lambda $};
\end{tikzpicture}}
\;\; + \;\;
\hackcenter{\begin{tikzpicture}[scale=0.8]
    \draw[thick, ->] (0,0) -- (0,2);
    \draw[thick, <-] (.75,0) -- (.75,2);
     \node at (-.2,.2) {\tiny $i$};
    \node at (.95,.2) {\tiny $i$};
     \node at (1.1,1.44) { $\lambda $};
\end{tikzpicture}}
\;\; = \;\;
\sum_{\overset{f_1+f_2+f_3}{=\l_i-1}}\hackcenter{
 \begin{tikzpicture}[scale=0.8]
 \draw[thick,->] (0,-1.0) .. controls ++(0,.5) and ++ (0,.5) .. (.8,-1.0) node[pos=.75, shape=coordinate](DOT1){};
  \draw[thick,<-] (0,1.0) .. controls ++(0,-.5) and ++ (0,-.5) .. (.8,1.0) node[pos=.75, shape=coordinate](DOT3){};
 \draw[thick,->] (0,0) .. controls ++(0,-.45) and ++ (0,-.45) .. (.8,0)node[pos=.25, shape=coordinate](DOT2){};
 \draw[thick] (0,0) .. controls ++(0,.45) and ++ (0,.45) .. (.8,0);
 \draw (-.15,.7) node { $\scs i$};
\draw (1.05,0) node { $\scs i$};
\draw (-.15,-.7) node { $\scs i$};
 \node at (.95,.65) {\tiny $f_3$};
 \node at (-.55,-.05) {\tiny $\overset{-\l_i-1}{+f_2}$};
  \node at (.95,-.65) {\tiny $f_1$};
 \node at (1.65,.3) { $\lambda $};
 \filldraw[thick]  (DOT3) circle (2.5pt);
  \filldraw[thick]  (DOT2) circle (2.5pt);
  \filldraw[thick]  (DOT1) circle (2.5pt);
\end{tikzpicture} }
\\
 \hackcenter{\begin{tikzpicture}[scale=0.8]
    \draw[thick,<-] (0,0) .. controls ++(0,.5) and ++(0,-.5) .. (.75,1);
    \draw[thick] (.75,0) .. controls ++(0,.5) and ++(0,-.5) .. (0,1);
    \draw[thick, ->] (0,1 ) .. controls ++(0,.5) and ++(0,-.5) .. (.75,2);
    \draw[thick] (.75,1) .. controls ++(0,.5) and ++(0,-.5) .. (0,2);
        \node at (-.2,.15) {\tiny $i$};
    \node at (.95,.15) {\tiny $i$};
     \node at (1.1,1.44) { $\lambda $};
\end{tikzpicture}}
\;\; + \;\;
\hackcenter{\begin{tikzpicture}[scale=0.8]
    \draw[thick, <-] (0,0) -- (0,2);
    \draw[thick, ->] (.75,0) -- (.75,2);
     \node at (-.2,.2) {\tiny $i$};
    \node at (.95,.2) {\tiny $i$};
     \node at (1.1,1.44) { $\lambda $};
\end{tikzpicture}}
\;\; = \;\;
\sum_{\overset{f_1+f_2+f_3}{=-\l_i-1}}\hackcenter{
 \begin{tikzpicture}[scale=0.8]
 \draw[thick,<-] (0,-1.0) .. controls ++(0,.5) and ++ (0,.5) .. (.8,-1.0) node[pos=.75, shape=coordinate](DOT1){};
  \draw[thick,->] (0,1.0) .. controls ++(0,-.5) and ++ (0,-.5) .. (.8,1.0) node[pos=.75, shape=coordinate](DOT3){};
 \draw[thick ] (0,0) .. controls ++(0,-.45) and ++ (0,-.45) .. (.8,0)node[pos=.25, shape=coordinate](DOT2){};
 \draw[thick, ->] (0,0) .. controls ++(0,.45) and ++ (0,.45) .. (.8,0);
 \draw (-.15,.7) node { $\scs i$};
\draw (1.05,0) node { $\scs i$};
\draw (-.15,-.7) node { $\scs i$};
 \node at (.95,.65) {\tiny $f_3$};
 \node at (-.55,-.05) {\tiny $\overset{\l_i-1}{+f_2}$};
  \node at (.95,-.65) {\tiny $f_1$};
 \node at (1.65,.3) { $\lambda $};
 \filldraw[thick]  (DOT3) circle (2.5pt);
  \filldraw[thick]  (DOT2) circle (2.5pt);
  \filldraw[thick]  (DOT1) circle (2.5pt);
\end{tikzpicture} } . \label{eq:sl2}
\end{split}
\end{equation}
\end{enumerate}
\end{definition}

It is sometimes convenient to use a shorthand notation for the bubbles that emphasizes their degrees.
\begin{equation}
\label{starnotation}
\hackcenter{ \begin{tikzpicture} [scale=.8]
 \draw  (-.75,1) arc (360:180:.45cm) [thick];
 \draw[<-](-.75,1) arc (0:180:.45cm) [thick];
     \filldraw  [black] (-1.55,.75) circle (2.5pt);
        \node at (-1.3,.3) { $\scriptstyle \ast + r$};
        \node at (-1.4,1.7) { $i $};
 \node at (-.2,1.5) { $\lambda $};
\end{tikzpicture}}
\;\; := \;\;
\hackcenter{ \begin{tikzpicture} [scale=.8]
 \draw  (-.75,1) arc (360:180:.45cm) [thick];
 \draw[<-](-.75,1) arc (0:180:.45cm) [thick];
     \filldraw  [black] (-1.55,.75) circle (2.5pt);
        \node at (-1.3,.3) { $\scriptstyle \l_i -1 + r$};
        \node at (-1.4,1.7) { $i $};
 \node at (-.2,1.5) { $\lambda $};
\end{tikzpicture}}
\qquad
\quad
\hackcenter{ \begin{tikzpicture} [scale=.8]
 \draw  (-.75,1) arc (360:180:.45cm) [thick];
 \draw[->](-.75,1) arc (0:180:.45cm) [thick];
     \filldraw  [black] (-1.55,.75) circle (2.5pt);
        \node at (-1.3,.3) { $\scriptstyle \ast + r$};
        \node at (-1.4,1.7) { $i $};
 \node at (-.2,1.5) { $\lambda $};
\end{tikzpicture}}
\;\; := \;\;
\hackcenter{ \begin{tikzpicture} [scale=.8]
 \draw  (-.75,1) arc (360:180:.45cm) [thick];
 \draw[->](-.75,1) arc (0:180:.45cm) [thick];
     \filldraw  [black] (-1.55,.75) circle (2.5pt);
        \node at (-1.3,.3) { $\scriptstyle -\l_i -1 + r$};
        \node at (-1.4,1.7) { $i $};
 \node at (-.2,1.5) { $\lambda $};
\end{tikzpicture}}
\end{equation}

\begin{definition}
A $2$-representation of $\Ucat_Q(\mf{g})$ is a graded additive $\Bbbk$-linear $2$-functor $\Ucat_Q(\mf{g}) \to \cal{K}$ for some graded, additive $2$-category $\cal{K}$.
\end{definition}

\subsection{Thick calculus}

The idempotent
completion, or Karoubi envelope $\dot{\cal{C}}$, of an additive category
$\cal{C}$ can be viewed as a minimal enlargement of the category $\cal{C}$ so
that idempotents split. The idempotent completion $\dot{\cal{K}}$ of a
$2$-category $\cal{K}$ is the $2$-category with the same objects as $\cal{K}$,
but with $\Hom$ categories given by the usual Karoubi envelope of
$\Hom_{\cal{K}}(A,B)$. Any additive $2$-functor $\cal{K} \to \cal{K}'$ that has
splitting of idempotent $2$-morphisms in $\cal{K}'$ extends uniquely to an
additive $2$-functor $\dot{\cal{K}} \to \cal{K'}$; see \cite[Section 3.4]{KL3}
for more details.

In order to define the Rickard complexes lifting the braid group action on integrable modules, we must first introduce the augmented graphical calculus for $\UcatD_Q$.  This so-called `thick calculus' describes $2$-morphisms in the Karoubi envelope of the $2$-category $\Ucat_Q$.  For more details in the $\mf{sl}_2$ case see \cite{KLMS}. For the simply laced case see \cite{Stosic1,Stosic2}, although care must be taken because all of the formulas in these references are for the unsigned version of the KLR-algebra where all $t_{ij}=1$.  Those formulas can be transferred to our conventions using the rescaling functors from Section 3.3 of \cite{Lau-param}.

In the Karoubi envelope $\UcatD_Q$ we can define divided power $1$-morphisms
 $\cal{E}^{(a)}_i\onel$ and $\cal{F}^{(b)}_i\onel$ by
\begin{equation}
\label{defEa}
 \cal{E}^{(a)}_i\onel\la t\ra := \left(\cal{E}^a_i\onel\left\la t-\frac{a(a-1)}{2}\right\ra,e_a\right) =: \;
\hackcenter{
\begin{tikzpicture} [scale=.75]
\draw[thick,  double, <-](0,1.15) to(0,-1.15);
 \node at (.25,-.75) {$\scs i$};
\node at (0,-1.35) {$\scs a $};
\node at (.5,.25) {$\scs \l $};
\node at (-.8,.25) {$\scs \l+a\alpha_i $};
\end{tikzpicture}}
\end{equation}
\begin{equation*}
 \cal{F}^{(a)}_i\onel\la t\ra :=
  \left(\cal{F}^a_i\onel\left\la t+\frac{a(a-1)}{2}\right\ra,{e}'_a\right) =:
\hackcenter{
\begin{tikzpicture} [scale=.75]
\draw[thick,  double, ->](0,1.15) to(0,-1.15);
 \node at (.25,-.75) {$\scs i$};
\node at (0,-1.35) {$\scs a $};
\node at (.5,.25) {$\scs \l $};
\node at (-.8,.25) {$\scs \l-a\alpha_i $};
\end{tikzpicture}}
\end{equation*}
where the idempotent $e_a$ is defined as follows
\begin{equation*}
 e_a := \delta_aD_a=\xy
(-18,-4)*{}; (9.5,-4)*{} **\dir{-};
(9.5,-4)*{}; (9.5,2)*{} **\dir{-};
(9.5,2)*{}; (-18,2)*{} **\dir{-};
(-18,2)*{}; (-18,-4)*{} **\dir{-};
(-7.5,-4)*{}; (-7.5,-8)*{} **\dir{-};
(-16,-4)*{}; (-16,-8)*{} **\dir{-};
(2.5,-4)*{}; (2.5,-8)*{} **\dir{-};
(7.5,-4)*{}; (7.5,-8)*{} **\dir{-};
{\ar (-16,2)*{}; (-16,8)*{}};
{\ar (-7.5,2)*{}; (-7.5,8)*{}};
{\ar (2.5,2)*{}; (2.5,8)*{}};
{\ar (7.5,2)*{}; (7.5,8)*{}};
(-7.5,5)*{\bullet};
(-16,5)*{\bullet};
(-11.5,6)*{\scriptstyle a-2};
(-20,6)*{\scriptstyle a-1};
(2.5,5)*{\bullet};
(-2.5,-6)*{\dots}; (-2.5,5)*{\dots};
(-4.25,-1)*{D_a};
(-19.5,0)*{};(11,0)*{};
\endxy
\end{equation*}
where all strands are labeled $i$.
Here $D_a$ is the longest braid on $a$-strands. The idempotents
${e}'_a$ are obtained from $e_a$ by a $180^\circ$ rotation.
We have $\cal{E}^a_i\onel \cong \oplus_{[a]!} \cal{E}^{(a)}_i\onel$ and
 $\cal{F}^{(b)}_i\onel \cong \oplus_{[b]!} \cal{F}^{(b)}_i\onel$.
Here we use the standard notation
 $$ [a] := \frac{q^a-q^{-a}}{q-q^{-1}}=q^{a-1}+q^{a-3}+\dots+q^{-a+1},\;\;
[a]!=\prod^a_{i=1} [i],
\quad \text{and}\quad
\left[ \begin{array}{c}a\\b\end{array}\right]=\frac{[a]!}{[b]![a-b]!}\, .$$
We define here some additional $2$-morphisms in $\UcatD_Q$, whose
 degrees can be read from the shift on the right-hand side. In this section we will assume that all strings are coloured by the Dynkin node $i$ that we omit for simplicity unless explicitly indicated in the diagrams.

\allowdisplaybreaks
\begin{minipage}{0.45\textwidth}
\begin{align*}
\hackcenter{ \begin{tikzpicture} [scale=.6]
\draw[thick,  double, <-](0,2).. controls ++(0,-.75) and ++(0,.3) ..(.6,1);
\draw[thick,  double,  <-](1.2,2).. controls ++(0,-.75) and ++(0,.3) ..(.6,1) to (.6,0);;
 \node at (0,2.2) {$\scs a$};
 \node at (1.2,2.2) {$\scs b$};
\node at (.6,-.2) {$\scs a+b$};
\node at (.3,.4) {$\scs i$};
\end{tikzpicture}} \; &:= \;  \text{$ \xy 0;/r.20pc/:
(-4,-8);(4,-1)*{} **\crv{(-4,-4.5) & (4,-4.5)};
(4,-8);(-4,-1)*{} **\crv{(4,-4.5) & (-4,-4.5)};
(-1,-1)*{}; (-7,-1)*{} **\dir{-};
(-7,-1)*{}; (-7,5)*{} **\dir{-};
(-7,5)*{}; (-1,5)*{} **\dir{-};
(-1,5)*{}; (-1,-1)*{} **\dir{-};
(1,-1)*{}; (7,-1)*{} **\dir{-};
(7,-1)*{}; (7,5)*{} **\dir{-};
(7,5)*{}; (1,5)*{} **\dir{-};
(1,5)*{}; (1,-1)*{} **\dir{-};
{\ar (-4,5)*{}; (-4,8)*{}};
{\ar (4,5)*{}; (4,8)*{}};
(-4,2)*{e_a};
(4,2)*{e_b};
(-2,-8)*{\scriptstyle b};
(6,-8)*{\scriptstyle a};
 (7,-4)*{\l};
\endxy$} \maps \scs \cal{E}^{(a+b)}_i\onel\rightarrow\cal{E}^{(a)}_i\cal{E}^{(b)}_i\onel\la -ab \ra,
\\
\hackcenter{ \begin{tikzpicture} [scale=.6]
\draw[thick,  double,](0,2).. controls ++(0,-.75) and ++(0,.3) ..(.6,1);
\draw[thick,  double,  ->](1.2,2).. controls ++(0,-.75) and ++(0,.3) ..(.6,1) to (.6,0);;
 \node at (0,2.2) {$\scs a$};
 \node at (1.2,2.2) {$\scs b$};
\node at (.6,-.2) {$\scs a+b$};
\node at (.3,.7) {$\scs i$};
\end{tikzpicture}} \; &:= \; \text{$\xy 0;/r.20pc/:
(-4,8);(-4,3)*{} **\dir{-};
(4,8);(4,3)*{} **\dir{-};
(-6,-3)*{}; (-6,3)*{} **\dir{-};
(-6,3)*{}; (6,3)*{} **\dir{-};
(6,3)*{}; (6,-3)*{} **\dir{-};
(6,-3)*{}; (-6,-3)*{} **\dir{-};
{\ar (0,-3)*{}; (0,-8)*{}};
(0,0)*{{e}'_{a+b}};
(-2,8)*{\scriptstyle a};
(6,8)*{\scriptstyle b};
 (7,-6)*{\l};
\endxy$} \maps \scs \cal{F}^{(a+b)}_i\onel\rightarrow\cal{F}^{(a)}_i\cal{F}^{(b)}_i\onel\la -ab \ra,
\\
\hackcenter{ \begin{tikzpicture} [scale=.6 ]
\draw[thick,  double, ->](0,0).. controls ++(0,.75) and ++(0,.76) ..(1,0);
 \node at (0,-.2) {$\scs a$};\node at (1.0,-.2) {$\scs a$};
\node at (-.1,.5) {$\scs i$};
\node at (1.2,.75) {$\scs \l$};
\end{tikzpicture}}\; &:= \;  \text{$\xy 0;/r.20pc/:
(-1,3)*{}; (-7,3)*{} **\dir{-};
(-7,3)*{}; (-7,-3)*{} **\dir{-};
(-7,-3)*{}; (-1,-3)*{} **\dir{-};
(-1,-3)*{}; (-1,3)*{} **\dir{-};
(1,3)*{}; (7,3)*{} **\dir{-};
(7,3)*{}; (7,-3)*{} **\dir{-};
(7,-3)*{}; (1,-3)*{} **\dir{-};
(1,-3)*{}; (1,3)*{} **\dir{-};
(-4,-8)*{}; (-4,-3)*{} **\dir{-};
(4,-3)*{}; (4,-8)*{} **\dir{-}?(1)*\dir{>};
 (-4,3)*{};(4,3)*{} **\crv{(-4,9) & (4,9)} ?(.55)*\dir{>};
 (-2,-8)*{\scriptstyle a};
(-4,0)*{e_a};
(4,0)*{{e}'_a};
 (7,6)*{\l};
\endxy$} \maps \scs\cal{E}^{(a)}_i\cal{F}^{(a)}_i\onel\rightarrow\onel\la a^2-a\l_i \ra,
\\
\hackcenter{ \begin{tikzpicture} [scale=.6 ]
\draw[thick,  double, <-](0,0).. controls ++(0,.75) and ++(0,.76) ..(1,0);
 \node at (0,-.2) {$\scs a$};\node at (1.0,-.2) {$\scs a$};
\node at (-.1,.5) {$\scs i$};
\node at (1.2,.75) {$\scs \l$};
\end{tikzpicture}} \; & := \; \text{$\xy 0;/r.20pc/:
(-1,3)*{}; (-7,3)*{} **\dir{-};
(-7,3)*{}; (-7,-3)*{} **\dir{-};
(-7,-3)*{}; (-1,-3)*{} **\dir{-};
(-1,-3)*{}; (-1,3)*{} **\dir{-};
(1,3)*{}; (7,3)*{} **\dir{-};
(7,3)*{}; (7,-3)*{} **\dir{-};
(7,-3)*{}; (1,-3)*{} **\dir{-};
(1,-3)*{}; (1,3)*{} **\dir{-};
(-4,-3)*{}; (-4,-8)*{} **\dir{-}?(1)*\dir{>};
(4,-3)*{}; (4,-8)*{} **\dir{-};
 (-4,3)*{};(4,3)*{} **\crv{(-4,9) & (4,9)} ?(.45)*\dir{<};
 (6,-8)*{\scriptstyle a};
(-4,0)*{{e}'_a};
(4,0)*{e_a};
 (7,6)*{\l};
\endxy$} \maps \scs\cal{F}^{(a)}_i\cal{E}^{(a)}_i\onel\rightarrow\onel\la a^2+a\l_i\ra,
\end{align*}
\end{minipage}
\hfill
\begin{minipage}{0.45\textwidth}
 \begin{align*}
\hackcenter{ \begin{tikzpicture} [scale=.6, rotate=180]
\draw[thick,  double,](0,2).. controls ++(0,-.75) and ++(0,.3) ..(.6,1);
\draw[thick,  double,  ->](1.2,2).. controls ++(0,-.75) and ++(0,.3) ..(.6,1) to (.6,0);;
 \node at (0,2.2) {$\scs b$};
 \node at (1.2,2.2) {$\scs a$};
\node at (.6,-.2) {$\scs a+b$};
\node at (.3,.7) {$\scs i$};
\end{tikzpicture}} \; &:= \; \text{$\xy 0;/r.20pc/:
(-4,-8);(-4,-3)*{} **\dir{-};
(4,-8);(4,-3)*{} **\dir{-};
(-6,-3)*{}; (-6,3)*{} **\dir{-};
(-6,3)*{}; (6,3)*{} **\dir{-};
(6,3)*{}; (6,-3)*{} **\dir{-};
(6,-3)*{}; (-6,-3)*{} **\dir{-};
{\ar (0,3)*{}; (0,8)*{}};
(0,0)*{e_{a+b}};
(-2,-8)*{\scriptstyle a};
(6,-8)*{\scriptstyle b};
 (7,6)*{\l};
\endxy$} \maps \scs \cal{E}^{(a)}_i\cal{E}^{(b)}_i\onel\rightarrow\cal{E}^{(a+b)}_i\onel\la -ab\ra,
\\
\hackcenter{ \begin{tikzpicture} [scale=.6, rotate=180]
\draw[thick,  double, <-](0,2).. controls ++(0,-.75) and ++(0,.3) ..(.6,1);
\draw[thick,  double,  <-](1.2,2).. controls ++(0,-.75) and ++(0,.3) ..(.6,1) to (.6,0);;
 \node at (0,2.2) {$\scs b$};
 \node at (1.2,2.2) {$\scs a$};
\node at (.6,-.2) {$\scs a+b$};
\node at (.3,.4) {$\scs i$};
\end{tikzpicture}} \; &:= \;
\text{$\xy 0;/r.20pc/:
(-4,8);(4,1)*{} **\crv{(-4,4.5) & (4,4.5)};
(4,8);(-4,1)*{} **\crv{(4,4.5) & (-4,4.5)};
(-1,1)*{}; (-7,1)*{} **\dir{-};
(-7,1)*{}; (-7,-5)*{} **\dir{-};
(-7,-5)*{}; (-1,-5)*{} **\dir{-};
(-1,-5)*{}; (-1,1)*{} **\dir{-};
(1,1)*{}; (7,1)*{} **\dir{-};
(7,1)*{}; (7,-5)*{} **\dir{-};
(7,-5)*{}; (1,-5)*{} **\dir{-};
(1,-5)*{}; (1,1)*{} **\dir{-};
{\ar (-4,-5)*{}; (-4,-8)*{}};
{\ar (4,-5)*{}; (4,-8)*{}};
(-4,-2)*{{e}'_a};
(4,-2)*{{e}'_b};
(-2,8)*{\scriptstyle b};
(6,8)*{\scriptstyle a};
 (7,4)*{\l};
\endxy$} \maps \scs \cal{F}^{(a)}_i\cal{F}^{(b)}_i\onel\rightarrow\cal{F}^{(a+b)}_i\onel\la -ab \ra,
\\
\hackcenter{ \begin{tikzpicture} [scale=.6 , rotate=180]
\draw[thick,  double, ->](0,0).. controls ++(0,.75) and ++(0,.76) ..(1,0);
 \node at (0,-.2) {$\scs a$};\node at (1.0,-.2) {$\scs a$};
\node at (-.1,.5) {$\scs i$};
\node at (1.2,.75) {$\scs \l$};
\end{tikzpicture}} \; &:= \;
\text{$\xy 0;/r.20pc/:
(-1,3)*{}; (-7,3)*{} **\dir{-};
(-7,3)*{}; (-7,-3)*{} **\dir{-};
(-7,-3)*{}; (-1,-3)*{} **\dir{-};
(-1,-3)*{}; (-1,3)*{} **\dir{-};
(1,3)*{}; (7,3)*{} **\dir{-};
(7,3)*{}; (7,-3)*{} **\dir{-};
(7,-3)*{}; (1,-3)*{} **\dir{-};
(1,-3)*{}; (1,3)*{} **\dir{-};
(-4,3)*{}; (-4,8)*{} **\dir{-}?(1)*\dir{>};
(4,3)*{}; (4,8)*{} **\dir{-};
 (-4,-3)*{};(4,-3)*{} **\crv{(-4,-9) & (4,-9)} ?(.45)*\dir{<};
 (6,8)*{\scriptstyle a};
(-4,0)*{e_a};
(4,0)*{{e}'_a};
 (7,-6)*{\l};
\endxy$} \maps \scs \onel\rightarrow\cal{E}^{(a)}_i\cal{F}^{(a)}_i\onel\la a^2-a\l_i \ra,
\\
\hackcenter{ \begin{tikzpicture} [scale=.6 , rotate=180]
\draw[thick,  double, <-](0,0).. controls ++(0,.75) and ++(0,.76) ..(1,0);
 \node at (0,-.2) {$\scs a$};\node at (1.0,-.2) {$\scs a$};
\node at (-.1,.5) {$\scs i$};
\node at (1.2,.75) {$\scs \l$};
\end{tikzpicture}} \; &:= \;
\text{$\xy 0;/r.20pc/:
(-1,3)*{}; (-7,3)*{} **\dir{-};
(-7,3)*{}; (-7,-3)*{} **\dir{-};
(-7,-3)*{}; (-1,-3)*{} **\dir{-};
(-1,-3)*{}; (-1,3)*{} **\dir{-};
(1,3)*{}; (7,3)*{} **\dir{-};
(7,3)*{}; (7,-3)*{} **\dir{-};
(7,-3)*{}; (1,-3)*{} **\dir{-};
(1,-3)*{}; (1,3)*{} **\dir{-};
(-4,8)*{}; (-4,3)*{} **\dir{-};
(4,3)*{}; (4,8)*{} **\dir{-}?(1)*\dir{>};
 (-4,-3)*{};(4,-3)*{} **\crv{(-4,-9) & (4,-9)} ?(.55)*\dir{>};
 (-2,8)*{\scriptstyle a};
(-4,0)*{{e}'_a};
(4,0)*{e_a};
 (7,-6)*{\lambda};
\endxy$} \maps \scs \onel\rightarrow\cal{F}^{(a)}_i\cal{E}^{(a)}_i\onel\la a^2+a\l_i \ra,
\end{align*}
\end{minipage}
where we use the short hand notation of thin strands labeled $a$ corresponds to $a$ thin strands labeled by $i \in I$. For example,
\[
\hackcenter{ \begin{tikzpicture} [scale=.75]
\draw[thick,  double, <- ](0,2).. controls ++(0,-.75) and ++(0,.3) ..(.6,1);
\draw[thick,  double, <- ](1.2,2).. controls ++(0,-.75) and ++(0,.3) ..(.6,1) to (.6,0);;
 \node at (0,2.2) {$\scs a$};
 \node at (1.2,2.2) {$\scs b$};
\node at (.6,-.2) {$\scs a+b$};
\end{tikzpicture}}
\;\; := \quad
\hackcenter{\begin{tikzpicture} [scale=0.75]
\draw[thick,  ->] (.4,-1.2) .. controls ++(0,.5) and ++(0,-.5) ..(-1.4,.1) to (-1.4,1.85);
\draw[thick,  ->] (1.4,-1.2) .. controls ++(0,.6) and ++(0,-.5) .. (-.4,.1) to (-.4,1.85);
\draw[thick,  ->] (-.4,-1.2) .. controls ++(0,.5) and ++(0,-.5) ..(1.4,.1) to (1.4,1.85);
\draw[thick,  ->] (-1.4,-1.2) .. controls ++(0,.6) and ++(0,-.5) ..(.4,.1) to (.4,1.85);
\draw[fill=white!20,] (-1.6,.3) rectangle (-.2,1.25);
\draw[fill=white!20,] (1.6,.3) rectangle (.2,1.25);
\node at (-.9,1.5) {$\dots$};
\node at (.9,1.5) {$\dots$};
\node at (-.8,-1.1) {$\dots$};
\node at (.9,-1.1) {$\dots$};
 \node at (-.9,.75) {$e_a$};
 \node at (.9,.75) {$e_b$};
\end{tikzpicture}}
\qquad \qquad
\hackcenter{ \begin{tikzpicture} [scale=.75]
\draw[thick,  double  ](0,-2).. controls ++(0,.75) and ++(0,-.3) ..(.6,-1);
\draw[thick,  double, -> ](1.2,-2).. controls ++(0,.75) and ++(0,-.3) ..(.6,-1) to (.6,0);;
 \node at (0,-2.2) {$\scs b$};
 \node at (1.2,-2.2) {$\scs a$};
\node at (.6,.2) {$\scs a+b$};
\end{tikzpicture}}
\;\; := \quad
\hackcenter{\begin{tikzpicture} [scale=0.75]
\draw[thick, -> ] (-.6,-.5) to (-.6,1.85);
\draw[thick,  ->] (.6,-.5) to (.6,1.85);
\draw[thick, -> ] (-1.2,-.5) to (-1.2,1.85);
\draw[thick,  ->] (1.2,-.5) to (1.2,1.85);
\draw[fill=white!20,] (-1.3,.1) rectangle (1.3,1.25);
\node at (0,-.2) {$\dots$};
\node at (0,1.65) {$\dots$};
 \node at (0,.75) {$e_{a+b}$};
\end{tikzpicture}} .
\]

The thick cap and cups satisfy zig-zag equations so that diagrams in $\UcatD_Q$ related by isotopy define identical $2$-morphisms in $\UcatD_Q$.
Furthermore, the splitters and mergers defined above satisfy associativity and coassociativity relations \cite[Proposition 2.4]{KLMS}  making it possible to define
\[
\hackcenter{
\begin{tikzpicture} [scale=0.75]
\draw[thick,  double] (0,1.15) to (0,1.85);
\node at (0,3.2) {$\dots$};
\draw[thick,  ] (0,1.85) .. controls ++(.25,.1) and ++(0,-.5) .. (.6,3.5);
\draw[thick,  ] (0,1.85) .. controls ++(-.25,.1) and ++(0,-.5) .. (-.6,3.5);
\draw[thick,  ] (0,1.85) .. controls ++(.25,.1) and ++(0,-.7) .. (1.2,3.5);
\draw[thick,  ] (0,1.85) .. controls ++(-.25,.1) and ++(0,-.7) .. (-1.2,3.5);
%
  \node at (0,.95) {$\scs a$};
\end{tikzpicture}}
\;\; := \quad
\hackcenter{\begin{tikzpicture} [scale=0.75]
\draw[thick,  ] (-.6,-.5) to (-.6,1.85);
\draw[thick,  ] (.6,-.5) to (.6,1.85);
\draw[thick,  ] (-1.2,-.5) to (-1.2,1.85);
\draw[thick,  ] (1.2,-.5) to (1.2,1.85);
\draw[fill=white!20,] (-1.3,.1) rectangle (1.3,1.25);
\node at (0,-.2) {$\dots$};
\node at (0,1.65) {$\dots$};
 \node at (0,.75) {$D_a$};
\end{tikzpicture}}
\qquad \qquad \quad
\hackcenter{
\begin{tikzpicture} [scale=0.75]
\draw[thick,  ] (-.6,-.5) .. controls ++(0,.5) and ++(-.4,-.1) .. (0,1.15);
\draw[thick,  ] (.6,-.5) .. controls ++(0,.5) and ++(.4,-.1) .. (0,1.15);
\draw[thick,  ] (-1.2,-.5) .. controls ++(0,.7) and ++(-.4,-.1) .. (0,1.15);
\draw[thick,  ] (1.2,-.5) .. controls ++(0,.7) and ++(.4,-.1) .. (0,1.15);
\draw[thick,  double, ] (0,1.15) to (0,1.85);
\node at (0,0) {$\dots$};
  \node at (0,2.05) {$\scs a$};
\end{tikzpicture}} \;\; := \quad
\hackcenter{\begin{tikzpicture} [scale=0.75]
\draw[thick,  ] (-.6,-.5) to (-.6,1.85);
\draw[thick,  ] (.6,-.5) to (.6,1.85);
\draw[thick,  ] (-1.2,-.5) to (-1.2,1.85);
\draw[thick,  ] (1.2,-.5) to (1.2,1.85);
\draw[fill=white!20,] (-1.3,.1) rectangle (1.3,1.25);
\node at (0,-.2) {$\dots$};
\node at (0,1.65) {$\dots$};
 \node at (0,.75) {$e_a$};
\end{tikzpicture}}
\]
unambiguously.

The center of the nilHecke algebra ${\rm NH}_a$ is isomorphic to the ring of symmetric functions $\Z[x_1,\dots,x_a]^{S_a}$. Hence, any $x\in \Sym_a$  defines an endomorphism of
$\cal{E}^{(a)}\onel$ (respectively $\cal{F}^{(a)}\onel$) in $\dot{\cal{U}}_Q$ since multiplication by symmetric functions $x\in \sym_a$  commutes with the idempotent $e_a$:
\[
\hackcenter{
\begin{tikzpicture} [scale=.75]
\draw[thick,  double, <-](0,1.15) to(0,-1.15);
 \node at (0,1.35) {$\scs a$};
\node at (0,-1.35) {$\scs a $};
\end{tikzpicture}}
\;\; := \;\;
\hackcenter{\begin{tikzpicture} [scale=0.75]
\draw[thick, -> ] (-.6,-.5) to (-.6,1.85);
\draw[thick, -> ] (.6,-.5) to (.6,1.85);
\draw[thick,  ->] (-1.2,-.5) to (-1.2,1.85);
\draw[thick,  ->] (1.2,-.5) to (1.2,1.85);
\draw[fill=white!20,] (-1.3,.1) rectangle (1.3,1.25);
\node at (0,-.2) {$\dots$};
\node at (0,1.65) {$\dots$};
 \node at (0,.75) {$e_a$};
\end{tikzpicture}}~,
\qquad \qquad \quad
\hackcenter{
\begin{tikzpicture} [scale=.75]
\draw[thick,  double, <-](0,1.15) to(0,-1.15);
\node  at (.3,0) {$x$};
 \filldraw  [black] (0,0) circle (2.5pt);
 \node at (0,1.35) {$\scs a$};
\node at (0,-1.35) {$\scs a$};
\end{tikzpicture}} \; \; \text{for $x \in \sym_a$, }
\qquad \qquad \quad
\hackcenter{
\begin{tikzpicture} [scale=.75]
\draw[thick,  double, <-](0,1.15) to(0,-1.15);
\filldraw  [black] (0,.5) circle (2.5pt);
\filldraw  [black] (0,-.6) circle (2.5pt);
\node at (.3,.5) {$x$};
\node  at (.3,-.6) {$y$};
 \node at (0,1.35) {$\scs a$};
\node at (0,-1.35) {$\scs a $};
\end{tikzpicture}}
\;\; = \;\;
\hackcenter{
\begin{tikzpicture} [scale=.75]
\draw[thick,  double, <-](0,1.15) to(0,-1.15);
\filldraw  [black] (0,0) circle (2.5pt);
\node  at (.35,0) {$xy$};
 \node at (0,1.35) {$\scs a$};
\node at (0,-1.35) {$\scs a $};
\end{tikzpicture}}
\]
where the product $xy$ is well defined since $xe_aye_a = xy e_a$.

For any composition $\mu=(\mu_1,\dots, \mu_n)$ write $\und{x}^{\mu}:= x_1^{\mu_1} x_2^{\mu_2} \dots x_n^{\mu_n}$.  We depict these diagrammatically as
\begin{equation}
  \und{x}^{\mu}  \;\; = \;\;
  \hackcenter{\begin{tikzpicture}
    \draw[thick, ] (-1.2,0) -- (-1.2,1.5);
    \draw[thick, ] (-.6,0) -- (-.6,1.5);
    \draw[thick,  ] (.6,0) -- (.6,1.5);
        \draw[thick,  ] (1.2,0) -- (1.2,1.5);
    \node at (-1.2,.75) {$\bullet$};
    \node at (-1.43,.9) {$\scs \mu_1$};
    \node at (-.6,.75) {$\bullet$};
    \node at (-.83,.9) {$\scs \mu_2$};
    \node at (.6,.75) {$\bullet$};
    \node at (.27,.9) {$\scs \mu_{n-1}$};
    \node at (1.2,.75) {$\bullet$};
    \node at (1.5,.9) {$\scs \mu_n$};
    \node at (0, .35) {$\cdots$};
\end{tikzpicture}}
\;\; = \;\;
  \hackcenter{\begin{tikzpicture}
    \draw[thick, ] (-1.2,0) -- (-1.2,1.5);
    \draw[thick, ] (-.6,0) -- (-.6,1.5);
    \draw[thick,  ] (.6,0) -- (.6,1.5);
        \draw[thick, ] (1.2,0) -- (1.2,1.5);
    \node[draw, fill=white!20 ,rounded corners ] at (0,.75) {$ \qquad \quad \und{x}^{\mu} \quad \qquad$};
    \node at (0, .35) {$\cdots$};
\end{tikzpicture}} .
\end{equation}

 For any Schur polynomial $s_{\mu}$ corresponding to the partition $\mu = (\mu_1, \dots , \mu_k)$  one can show that
\[
\hackcenter{
\begin{tikzpicture} [scale=.75]
\draw[thick,  double, ](0,1.15) to(0,-1.15);
\node[draw, fill=white!20 ,rounded corners ] at (0,0) {$ s_{\mu}$};
 \node at (0,1.35) {$\scs k$};
\node at (0,-1.35) {$\scs k $};
\end{tikzpicture}}
\; \; =\;\;
\hackcenter{
\begin{tikzpicture} [scale=0.65]
\draw[thick,  ] (0,-1.15).. controls++(-.4,.1)and ++(0,-.5).. (-.6,0) .. controls ++(0,.5) and ++(-.4,-.1) .. (0,1.15);
\draw[thick,  ] (0,-1.15).. controls ++(.4,.1) and ++(0,-.5).. (.6,0) .. controls ++(0,.5) and ++(.4,-.1) .. (0,1.15);
\draw[thick,  ]  (0,-1.15).. controls ++(-.4,.1)and ++(0,-.7) .. (-1.2,0) .. controls ++(0,.7) and ++(-.4,-.1) .. (0,1.15);
\draw[thick,  ] (0,-1.15) .. controls ++(.4,.1) and ++(0,-.7) .. (1.2,0) .. controls ++(0,.7) and ++(.4,-.1) .. (0,1.15);
\node[draw, fill=white!20 ,rounded corners ] at (0,.3 ) {$ \qquad \und{x}^{\mu+\delta}\qquad $};
\draw[thick,  double, ] (0,1.15) to (0,1.85);
\draw[thick,  double ] (0,-1.15) to (0,-1.65);
\node at (0,-.4) {$\dots$};
  \node at (0,2.05) {$\scs k$};
\end{tikzpicture}}
\]
where $\mu+\delta$ is the partition $(\mu_1+k-1, \mu_2+k-2, \dots, \mu_k+ 0)$.  We denote by $h_m = s_{(m)}$ and $\varepsilon_m =s_{(1^m)}$ the complete and elementary symmetric function, respectively.

We now present some important identities holding in the thick calculus.
Note that by \cite[Corollary 4.7 and Proposition 4.8]{KLMS}
\begin{alignat}{2} \label{eq:bub-slide}
\hackcenter{ \begin{tikzpicture} [scale=.7]
\draw[thick, double, ->] (0,0) to (0,2.5) ;
  \node at (0,-.2) { $a $};
      \node at (-.2,.4) { $i $};
      \node at (1.4,1.55) { $i $};
 \draw  (1.5,1) arc (360:180:.45cm) [thick];
 \draw[,<-](1.5,1) arc (0:180:.45cm) [thick];
 \filldraw  [black] (.7,.75) circle (2.5pt);
 \node at (1.0,.3) {\tiny $\ast+j$};
 \node at (1.15, 2.2) { $\lambda $};
\end{tikzpicture}}
&\; = \;
\sum_{\overset{x+y+z}{=j}} (-1)^{x+y}\;
\hackcenter{ \begin{tikzpicture} [scale=.7]
\draw[thick, double, ->] (2.5,0) to (2.5,2.5) ;
 \filldraw  [black] (2.5,1.75) circle (2.5pt);
  \filldraw  [black] (2.5,1) circle (2.5pt);
 \node at (2.9,1.75) { $\varepsilon_x$};
 \node at (2.9,1) { $\varepsilon_y$};
  \node at (2.5,-.2) { $a $};
  \node at (2.7,.4) { $i $};
 \draw  (1.7,1) arc (360:180:.45cm) [thick];
 \draw[,<-](1.7,1) arc (0:180:.45cm) [thick];
 \filldraw  [black] (.9,.75) circle (2.5pt);
 \node at (1.1,.3) {\tiny $\ast+z$};
 \node at (1.6,1.55) { $i $};
 \node at (3.25, 2.2) { $\lambda $};
\end{tikzpicture}}
\quad \;\;
&&
\hackcenter{ \begin{tikzpicture} [scale=.7]
\draw[thick, double, ->] (2.5,0) to (2.5,2.5) ;
  \node at (2.5,-.2) { $a $};
  \node at (2.7,.4) { $i $};
 \draw  (1.7,1) arc (360:180:.45cm) [thick];
 \draw[,<-](1.7,1) arc (0:180:.45cm) [thick];
 \filldraw  [black] (.9,.75) circle (2.5pt);
 \node at (1.1,.3) {\tiny $\ast+j$};
 \node at (1.6,1.55) { $i $};
 \node at (3.25, 2.2) { $\lambda $};
\end{tikzpicture}}
\; = \;
\sum_{\overset{x+y+z}{=j}}  \;
\hackcenter{ \begin{tikzpicture} [scale=.7]
\draw[thick, double, ->] (0,0) to (0,2.5) ;
 \filldraw  [black] (0,1.75) circle (2.5pt);
  \filldraw  [black] (0,1) circle (2.5pt);
 \node at (-.4,1.75) { $h_x $};
 \node at (-.4,1) { $h_y$};
  \node at (0,-.2) { $a $};
      \node at (-.2,.4) { $i $};
      \node at (1.4,1.55) { $i $};
 \draw  (1.5,1) arc (360:180:.45cm) [thick];
 \draw[,<-](1.5,1) arc (0:180:.45cm) [thick];
 \filldraw  [black] (.7,.75) circle (2.5pt);
 \node at (1.0,.3) {\tiny $\ast+z$};
 \node at (1.15, 2.2) { $\lambda $};
\end{tikzpicture}}
\\
\label{eq:ccbub-slide}
\hackcenter{ \begin{tikzpicture} [scale=.7]
\draw[thick, double, ->] (0,0) to (0,2.5) ;
  \node at (0,-.2) { $a $};
      \node at (-.2,.4) { $i $};
      \node at (1.4,1.55) { $i $};
 \draw  (1.5,1) arc (360:180:.45cm) [thick];
 \draw[,->](1.5,1) arc (0:180:.45cm) [thick];
 \filldraw  [black] (.7,.75) circle (2.5pt);
 \node at (1.0,.3) {\tiny $\ast+j$};
 \node at (1.15, 2.2) { $\lambda $};
\end{tikzpicture}}
&\; = \;
\sum_{\overset{x+y+z}{=j}}
\hackcenter{ \begin{tikzpicture} [scale=.7]
\draw[thick, double, ->] (2.5,0) to (2.5,2.5) ;
 \filldraw  [black] (2.5,1.75) circle (2.5pt);
  \filldraw  [black] (2.5,1) circle (2.5pt);
 \node at (2.9,1.75) { $h_x$};
 \node at (2.9,1) { $h_y$};
  \node at (2.5,-.2) { $a $};
  \node at (2.7,.4) { $i $};
 \draw  (1.7,1) arc (360:180:.45cm) [thick];
 \draw[,->](1.7,1) arc (0:180:.45cm) [thick];
 \filldraw  [black] (.9,.75) circle (2.5pt);
 \node at (1.1,.3) {\tiny $\ast+z$};
 \node at (1.6,1.55) { $i $};
 \node at (3.25, 2.2) { $\lambda $};
\end{tikzpicture}}
\quad \;\;
&&
\hackcenter{ \begin{tikzpicture} [scale=.7]
\draw[thick, double, ->] (2.5,0) to (2.5,2.5) ;
  \node at (2.5,-.2) { $a $};
  \node at (2.7,.4) { $i $};
 \draw  (1.7,1) arc (360:180:.45cm) [thick];
 \draw[,->](1.7,1) arc (0:180:.45cm) [thick];
 \filldraw  [black] (.9,.75) circle (2.5pt);
 \node at (1.1,.3) {\tiny $\ast+j$};
 \node at (1.6,1.55) { $i $};
 \node at (3.25, 2.2) { $\lambda $};
\end{tikzpicture}}
\; = \;
\sum_{\overset{x+y+z}{=j}} (-1)^{x+y}\; \;
\hackcenter{ \begin{tikzpicture} [scale=.7]
\draw[thick, double, ->] (0,0) to (0,2.5) ;
 \filldraw  [black] (0,1.75) circle (2.5pt);
  \filldraw  [black] (0,1) circle (2.5pt);
 \node at (-.4,1.75) { $\varepsilon_x $};
 \node at (-.4,1) { $\varepsilon_y$};
  \node at (0,-.2) { $a $};
      \node at (-.2,.4) { $i $};
      \node at (1.4,1.55) { $i $};
 \draw  (1.5,1) arc (360:180:.45cm) [thick];
 \draw[,->](1.5,1) arc (0:180:.45cm) [thick];
 \filldraw  [black] (.7,.75) circle (2.5pt);
 \node at (1.0,.3) {\tiny $\ast+z$};
 \node at (1.15, 2.2) { $\lambda $};
\end{tikzpicture}}.
\end{alignat}

This implies
\begin{equation} \label{eq:dot4bub}
\hackcenter{ \begin{tikzpicture} [scale=.7]
\draw[thick, double, ->] (2.5,0) to (2.5,2.5) ;
 \filldraw  [black] (2.5,1.5) circle (2.5pt);
 \node at (2.9,1.5) { $\varepsilon_1$};
  \node at (2.5,-.2) { $a $};
  \node at (2.7,.4) { $i $};
      %
 \node at (3.25, 2.2) { $\lambda $};
\end{tikzpicture}}
\;\; =\;\;
\frac{1}{2} \left( \;
\hackcenter{ \begin{tikzpicture} [scale=.7]
\draw[thick, double, ->] (2.5,0) to (2.5,2.5) ;
  \node at (2.5,-.2) { $a $};
  \node at (2.7,.4) { $i $};
 \draw  (1.7,1) arc (360:180:.45cm) [thick];
 \draw[,<-](1.7,1) arc (0:180:.45cm) [thick];
 \filldraw  [black] (.9,.75) circle (2.5pt);
 \node at (1.1,.3) {\tiny $\ast+1$};
 \node at (1.6,1.55) { $i $};
 \node at (3.25, 2.2) { $\lambda $};
\end{tikzpicture}}
  - \;\;
\hackcenter{ \begin{tikzpicture} [scale=.7]
\draw[thick, double, ->] (0,0) to (0,2.5) ;
  \node at (0,-.2) { $a $};
      \node at (-.2,.4) { $i $};
      \node at (1.4,1.55) { $i $};
 \draw  (1.5,1) arc (360:180:.45cm) [thick];
 \draw[,<-](1.5,1) arc (0:180:.45cm) [thick];
 \filldraw  [black] (.7,.75) circle (2.5pt);
 \node at (1.0,.3) {\tiny $\ast+1$};
 \node at (1.15, 2.2) { $\lambda $};
\end{tikzpicture}}
\; \right)
\;\; =\;\;
\frac{1}{2} \left( \;
\hackcenter{ \begin{tikzpicture} [scale=.7]
\draw[thick, double, ->] (0,0) to (0,2.5) ;
  \node at (0,-.2) { $a $};
      \node at (-.2,.4) { $i $};
      \node at (1.4,1.55) { $i $};
 \draw  (1.5,1) arc (360:180:.45cm) [thick];
 \draw[,->](1.5,1) arc (0:180:.45cm) [thick];
 \filldraw  [black] (.7,.75) circle (2.5pt);
 \node at (1.0,.3) {\tiny $\ast+1$};
 \node at (1.15, 2.2) { $\lambda $};
\end{tikzpicture}}
  - \;\;
\hackcenter{ \begin{tikzpicture} [scale=.7]
\draw[thick, double, ->] (2.5,0) to (2.5,2.5) ;
  \node at (2.5,-.2) { $a $};
  \node at (2.7,.4) { $i $};
 \draw  (1.7,1) arc (360:180:.45cm) [thick];
 \draw[,->](1.7,1) arc (0:180:.45cm) [thick];
 \filldraw  [black] (.9,.75) circle (2.5pt);
 \node at (1.1,.3) {\tiny $\ast+1$};
 \node at (1.6,1.55) { $i $};
 \node at (3.25, 2.2) { $\lambda $};
\end{tikzpicture}}
\; \right) .
\end{equation}

For $i\cdot j =-1$, one can calculate in a similar way that the mixed bubble slides have the form
 \begin{align} \label{eq:jslide}
\hackcenter{ \begin{tikzpicture} [scale=.75]
\draw[thick, double, ->] (2.5,0) to (2.5,2.5) ;
  \node at (2.5,-.2) { $a $};
  \node at (2.7,.4) { $i $};
 \draw[blue]  (1.7,1) arc (360:180:.45cm) [thick];
 \draw[blue,<-](1.7,1) arc (0:180:.45cm) [thick];
 \filldraw  [black] (.9,.75) circle (2.5pt);
 \node at (1.1,.3) {\tiny $\ast+m$};
 \node at (1.6,1.65) { $j $};
 \node at (3.25, 2.2) { $\lambda $};
\end{tikzpicture}}
&=
\sum_{k=0}^{a}
t_{ij}^{a- k} t_{ji}^{k-a}\;
\hackcenter{ \begin{tikzpicture} [scale=.75]
\draw[thick, double, ->] (0,0) to (0,2.5) ;
 \filldraw  [black] (0,1.5) circle (2.5pt);
 \node at (-.6,1.45) { $\varepsilon_{a-k} $};
  \node at (0,-.2) { $a $};
      \node at (-.2,.4) { $i $};
 \draw[blue]  (1.5,1) arc (360:180:.45cm) [thick];
 \draw[blue,<-](1.5,1) arc (0:180:.45cm) [thick];
 \filldraw  [black] (.7,.75) circle (2.5pt);
    \node at (1.4,1.65) { $j $};
 \node at (1.4,.3) {\tiny $\ast+m-a+k$};
 \node at (1.15, 2.2) { $\lambda $};
\end{tikzpicture}}
=
\sum_{k'= 0}^{m}
t_{ij}^{k'} t_{ji}^{-k'}\;
\hackcenter{ \begin{tikzpicture} [scale=.75]
\draw[thick, double, ->] (0,0) to (0,2.5) ;
 \filldraw  [black] (0,1.5) circle (2.5pt);
 \node at (-.6,1.45) { $\varepsilon_{k'} $};
  \node at (0,-.2) { $a $};
      \node at (-.2,.4) { $i $};
 \draw[blue]  (1.5,1) arc (360:180:.45cm) [thick];
 \draw[blue,<-](1.5,1) arc (0:180:.45cm) [thick];
 \filldraw  [black] (.7,.75) circle (2.5pt);
    \node at (1.4,1.65) { $j $};
 \node at (1.4,.3) {\tiny $\ast+m-k'$};
 \node at (1.15, 2.2) { $\lambda $};
\end{tikzpicture}} ,
\\
 \label{eq:jslide2}
 \hackcenter{ \begin{tikzpicture} [scale=.75]
\draw[thick, double, ->] (0,0) to (0,2.5) ;
      \node at (-.2,.4) { $i $};
 \draw[blue]  (1.5,1) arc (360:180:.45cm) [thick];
 \draw[blue,<-](1.5,1) arc (0:180:.45cm) [thick];
 \filldraw  [black] (.7,.75) circle (2.5pt);
    \node at (1.4,1.65) { $j $};
 \node at (1.4,.3) {\tiny $\ast+m$};
 \node at (1.0, 2.2) { $\lambda $};
\end{tikzpicture}}
\;\;& = \;\;
\sum_{s=0}^{m}(-t_{ji}^{-1}t_{ij})^{m-s}
\hackcenter{ \begin{tikzpicture} [scale=.75]
\draw[thick, double, ->] (2.5,0) to (2.5,2.5) ;
 \filldraw  [black] (2.5,1.5) circle (2.5pt);
 \node at (3.2,1.45) { $h_{m-s} $};
  \node at (2.5,-.2) { $a $};
  \node at (2.7,.4) { $i $};
 \draw[blue]  (1.7,1) arc (360:180:.45cm) [thick];
 \draw[blue,<-](1.7,1) arc (0:180:.45cm) [thick];
 \filldraw  [black] (.9,.75) circle (2.5pt);
 \node at (1.1,.3) {\tiny $\ast+s$};
 \node at (1.6,1.65) { $j $};
 \node at (3.25, 2.2) { $\lambda $};
\end{tikzpicture}}
\;\; = \;\;
\sum_{s+t=m}
(-v_{ji})^{t}
\hackcenter{ \begin{tikzpicture} [scale=.75]
\draw[thick, double, ->] (2.5,0) to (2.5,2.5) ;
 \filldraw  [black] (2.5,1.5) circle (2.5pt);
 \node at (3.0,1.45) { $h_{t} $};
  \node at (2.5,-.2) { $a $};
  \node at (2.7,.4) { $i $};
 \draw[blue]  (1.7,1) arc (360:180:.45cm) [thick];
 \draw[blue,<-](1.7,1) arc (0:180:.45cm) [thick];
 \filldraw  [black] (.9,.75) circle (2.5pt);
 \node at (1.1,.3) {\tiny $\ast+s$};
 \node at (1.6,1.65) { $j $};
 \node at (3.25, 2.2) { $\lambda $};
\end{tikzpicture}} .
\end{align}

Equation~\eqref{eq:jslide} implies
\begin{equation} \label{eq:X0}
\hackcenter{ \begin{tikzpicture} [scale=.75]
\draw[thick, double, ->] (2.5,0) to (2.5,2.5) ;
  \node at (2.5,-.2) { $a $};
  \node at (2.7,.4) { $i $};
 \draw[blue]  (1.7,1) arc (360:180:.45cm) [thick];
 \draw[blue,<-](1.7,1) arc (0:180:.45cm) [thick];
 \filldraw  [black] (.9,.75) circle (2.5pt);
 \node at (1.1,.3) {\tiny $\ast+1$};
 \node at (1.6,1.65) { $j $};
 \node at (3.25, 2.2) { $\lambda $};
\end{tikzpicture}}
=
t_{ij}t_{ji}^{-1} \;
\hackcenter{ \begin{tikzpicture} [scale=.75]
\draw[thick, double, ->] (0,0) to (0,2.5) ;
 \filldraw  [black] (0,1.5) circle (2.5pt);
 \node at (-.4,1.5) { $\varepsilon_1 $};
  \node at (0,-.2) { $a $};
      \node at (-.2,.4) { $i$};
 \node at (1.0, 2.2) { $\lambda $};
\end{tikzpicture}}
\;\;
+
\hackcenter{ \begin{tikzpicture} [scale=.75 ]
\draw[thick, double, ->] (0,0) to (0,2.5) ;
  \node at (0,-.2) { $a $};
      \node at (-.2,.4) { $i $};
 \draw[blue]  (1.5,1) arc (360:180:.45cm) [thick];
 \draw[blue,<-](1.5,1) arc (0:180:.45cm) [thick];
 \filldraw  [black] (.7,.75) circle (2.5pt);
          \node at (1.4,1.65) { $j$};
 \node at (1.0,.3) {\tiny $\ast+1$};
 \node at (1.15, 2.2) { $\lambda $};
\end{tikzpicture}}.
\end{equation}

\begin{lemma} \cite[Lemma 4.16]{KLMS}
\label{squarefloplem}
There is an equality of diagrams
\begin{align}
\hackcenter{ \begin{tikzpicture} [scale=.7]
\draw[thick, directed=0.8] (0,1.5) .. controls ++(0,.75) and ++(0,.75) .. (2,1.5);
    \draw[thick, directed=0.8] (2,1) .. controls ++(0,-.75) and ++(0,-.75) .. (0,1);
\draw[thick, double, ->] (0,0) to (0,2.5) ;
  \node at (0,-.2) { $a $};
    \node at (-.2,.4) { $i $};
\draw[thick, double, <-] (2,0) to (2,2.5) ;
  \node at (2,-.2) { $b $};
      \node at (2.5,2) { $\lambda $};
\end{tikzpicture}}
\;\; + \;\;
\hackcenter{ \begin{tikzpicture} [scale=.7]
    \draw[thick, directed=0.8] (0,.5) .. controls ++(0,.75) and ++(0,.75) .. (2,.5);
    \draw[thick, directed=0.8] (2,2) .. controls ++(0,-.75) and ++(0,-.75) .. (0,2);
\draw[thick, double, ->] (0,0) to (0,2.5) ;
  \node at (0,-.2) { $a $};
    \node at (-.2,.4) { $i $};
\draw[thick, double, <-] (2,0) to (2,2.5) ;
  \node at (2,-.2) { $b $};
      \node at (2.5,2) { $\lambda $};
\end{tikzpicture}}
\;\; &= \;\;
(-1)^{a-b}
\sum_{\overset{p+q+r}{=b-a+1-\l_i}}
\hackcenter{ \begin{tikzpicture} [scale=.7]
\draw[thick, double, ->] (0,0) to (0,2.5) ;
 \filldraw  [black] (0,1.5) circle (2.5pt);
 \node at (-.4,1.5) { $h_p $};
  \node at (0,-.2) { $a $};
      \node at (-.2,.4) { $i $};
          \node at (2.9,.6) { $i $};
              \node at (1.4,1.55) { $i $};
\draw[thick, double, <-] (2.25,0) to (2.25,2.5) ;
 \filldraw  [black] (2.25,1.5) circle (2.5pt);
 \node at (2.65,1.5) { $h_q $};
 \node at (2.25,-.2) {$b$};
 \draw  (1.5,1) arc (360:180:.45cm) [thick];
 \draw[,<-](1.5,1) arc (0:180:.45cm) [thick];
 \filldraw  [black] (.7,.75) circle (2.5pt);
 \node at (1.0,.2) {\tiny $\ast+r$};
       \node at (2.7,2.2) { $\lambda $};
\end{tikzpicture}} .
\end{align}
\end{lemma}

\begin{lemma} \label{lem:A2}
For any $a,b, \beta \geq 0$ the identities
\begin{align}
 \sum_{\overset{p+q+r}{=\beta}}
\hackcenter{ \begin{tikzpicture} [scale=.7]
\draw[thick, double, ->] (0,0) to (0,2.5) ;
 \filldraw  [black] (0,1.5) circle (2.5pt);
 \node at (-.4,1.5) { $h_p $};
  \node at (0,-.2) { $a $};
      \node at (-.2,.4) { $i $};
          \node at (2.4,.6) { $i $};
\draw[thick, double, <-] (2.0,0) to (2.0,2.5) ;
 \filldraw  [black] (2.0,1.5) circle (2.5pt);
 \node at (2.4,1.5) { $h_q $};
 \node at (2.0,-.2) {$b$};
 \draw  (1.5,1) arc (360:180:.45cm) [thick];
 \draw[,<-](1.5,1) arc (0:180:.45cm) [thick];
 \filldraw  [black] (.7,.75) circle (2.5pt);
 \node at (1.4,1.65) { $i $};
 \node at (.8,.3) {\tiny $\ast+r$};
       \node at (2.4,2.2) { $\lambda $};
\end{tikzpicture}}
& =
\sum_{ q+y+z=\beta}
(-1)^{y}
\hackcenter{ \begin{tikzpicture} [scale=.7]
\draw[thick, double, ->] (.5,0) to (.5,2.5) ;
   \filldraw  [black] (.5,1.5) circle (2.5pt);
 \node at (.1,1.55) { $\varepsilon_y $};
  \node at (.5,-.2) { $a $};
      \node at (.3,.3) { $i $};
      \node at (1.8,.5) { $i $};
\draw[thick, double, <-] (1.5,0) to (1.5,2.5) ;
 \filldraw  [black] (1.5,1.5) circle (2.5pt);
    \node at (1.9,1.5) { $h_q $};
    \node at (1.5,-.2) {$b$};
 \draw  (-.5,1) arc (360:180:.45cm) [thick];
 \draw[,<-](-.5,1) arc (0:180:.45cm) [thick];
    \filldraw  [black] (-1.3,.75) circle (2.5pt);
     \node at (-.9,1.65) { $i $};
     \node at (-1.2,.3) {\tiny $\ast+z$};
       \node at (2.2,2.2) { $\lambda $};
\end{tikzpicture}}
=
\sum_{ q+y+z=\beta}
(-1)^{q}
\hackcenter{ \begin{tikzpicture} [scale=.7]
\draw[thick, double, ->] (.5,0) to (.5,2.5) ;
   \filldraw  [black] (.5,1.5) circle (2.5pt);
 \node at (.1,1.55) { $h_y $};
  \node at (.5,-.2) { $a $};
      \node at (.3,.3) { $i $};
      \node at (1.8,.5) { $i $};
\draw[thick, double, <-] (1.5,0) to (1.5,2.5) ;
 \filldraw  [black] (1.5,1.5) circle (2.5pt);
    \node at (1.9,1.5) { $\varepsilon_q $};
    \node at (1.5,-.2) {$b$};
 \draw  (3.1,1) arc (360:180:.45cm) [thick];
 \draw[,<-](3.1,1) arc (0:180:.45cm) [thick];
    \filldraw  [black] (2.4,.65) circle (2.5pt);
     \node at (2.7,1.65) { $i $};
     \node at (2.6,.3) {\tiny $\ast+z$};
       \node at (2.4,2.2) { $\lambda $};
\end{tikzpicture}}
\end{align}
hold.
\end{lemma}

\begin{proof}
The first identity is proved using \eqref{eq:bub-slide} as follows:
\begin{align}
  & \sum_{\overset{p+q+r}{=\beta}}
\hackcenter{ \begin{tikzpicture} [scale=.75]
\draw[thick, double, ->] (0,0) to (0,2.5) ;
 \filldraw  [black] (0,1.5) circle (2.5pt);
 \node at (-.4,1.5) { $h_p $};
  \node at (0,-.2) { $a $};
      \node at (-.2,.4) { $i $};
          \node at (2.9,.6) { $i $};
\draw[thick, double, <-] (2.0,0) to (2.0,2.5) ;
 \filldraw  [black] (2.0,1.5) circle (2.5pt);
 \node at (2.4,1.5) { $h_q $};
 \node at (2.0,-.2) {$b$};
 \draw  (1.5,1) arc (360:180:.45cm) [thick];
 \draw[,<-](1.5,1) arc (0:180:.45cm) [thick];
 \filldraw  [black] (.7,.75) circle (2.5pt);
 \node at (1.4,1.65) { $i $};
 \node at (.8,.3) {\tiny $\ast+r$};
       \node at (2.4,2.2) { $\lambda $};
\end{tikzpicture}}
 \; = \;
\sum_{\overset{p+q+r}{=\beta}}
\sum_{\overset{x+y+z}{=r}}(-1)^{x+y}
\hackcenter{ \begin{tikzpicture} [scale=.75]
\draw[thick, double, ->] (.5,0) to (.5,2.5) ;
 \filldraw  [black] (.5,1.75) circle (2.5pt);
  \filldraw  [black] (.5,1.25) circle (2.5pt);
   \filldraw  [black] (.5,.75) circle (2.5pt);
 \node at (.1,1.75) { $h_p $};
 \node at (.1,1.25) { $\varepsilon_x $};
 \node at (.1,.75) { $\varepsilon_y $};
  \node at (.5,-.2) { $a $};
      \node at (.3,.3) { $i $};
      \node at (2.3,.5) { $i $};
\draw[thick, double, <-] (2.0,0) to (2.0,2.5) ;
 \filldraw  [black] (2.0,1.5) circle (2.5pt);
    \node at (2.4,1.5) { $h_q $};
    \node at (2.0,-.2) {$b$};
 \draw  (-.5,1) arc (360:180:.45cm) [thick];
 \draw[,<-](-.5,1) arc (0:180:.45cm) [thick];
    \filldraw  [black] (-1.3,.75) circle (2.5pt);
     \node at (-.9,1.65) { $i $};
     \node at (-1.2,.3) {\tiny $\ast+z$};
       \node at (2.4,2.2) { $\lambda $};
\end{tikzpicture}}
\\
&\;\;=
\sum_{ q+y+z=0}^{\beta}
(-1)^{y}
 \delta_{\beta-q-y-z,0}
\hackcenter{ \begin{tikzpicture} [scale=.75]
\draw[thick, double, ->] (.5,0) to (.5,2.5) ;
   \filldraw  [black] (.5,.75) circle (2.5pt);
 \node at (.1,.75) { $\varepsilon_y $};
  \node at (.5,-.2) { $a $};
      \node at (.3,.3) { $i $};
      \node at (2.3,.5) { $i $};
\draw[thick, double, <-] (1.5,0) to (1.5,2.5) ;
 \filldraw  [black] (1.5,1.5) circle (2.5pt);
    \node at (1.9,1.5) { $h_q $};
    \node at (1.5,-.2) {$b$};
 \draw  (-.5,1) arc (360:180:.45cm) [thick];
 \draw[,<-](-.5,1) arc (0:180:.45cm) [thick];
    \filldraw  [black] (-1.3,.75) circle (2.5pt);
     \node at (-.9,1.65) { $i $};
     \node at (-1.2,.3) {\tiny $\ast+z$};
       \node at (2.4,2.2) { $\lambda $};
\end{tikzpicture}}
 \;\;=
\sum_{ q+y+z=\beta}
(-1)^{y}
\hackcenter{ \begin{tikzpicture} [scale=.75]
\draw[thick, double, ->] (.5,0) to (.5,2.5) ;
   \filldraw  [black] (.5,1.5) circle (2.5pt);
 \node at (.1,1.55) { $\varepsilon_y $};
  \node at (.5,-.2) { $a $};
      \node at (.3,.3) { $i $};
      \node at (2.3,.5) { $i $};
\draw[thick, double, <-] (1.5,0) to (1.5,2.5) ;
 \filldraw  [black] (1.5,1.5) circle (2.5pt);
    \node at (1.9,1.5) { $h_q $};
    \node at (1.5,-.2) {$b$};
 \draw  (-.5,1) arc (360:180:.45cm) [thick];
 \draw[,<-](-.5,1) arc (0:180:.45cm) [thick];
    \filldraw  [black] (-1.3,.75) circle (2.5pt);
     \node at (-.9,1.65) { $i $};
     \node at (-1.2,.3) {\tiny $\ast+z$};
       \node at (2.4,2.2) { $\lambda $};
\end{tikzpicture}}
\end{align}
where we have used the fundamental relation between complete and elementary symmetric functions.

The second identity is proven similarly by pushing the interior bubble to the right.
\end{proof}

\begin{lemma}  \label{lem:dots-j}
\begin{align}
& \hackcenter{ \begin{tikzpicture} [scale=.7]
\draw[thick, double, ->] (-1.5,0) to (-1.5,2.5) ;
  \node at (-1.5,-.2) { $a $};
      \node at (-1.7,.5) { $i $};
\draw[thick, double, <-] (0,0) to (0,2.5) ;
\filldraw  [black] (0,1.5) circle (2.5pt);
 \node at (-.4,1.5) { $\varepsilon_1 $};
  \node at (0,-.2) { $b$};
      \node at (-.2,.5) { $i $};
 %
 \node at (.7, 2.2) { $\lambda $};
\end{tikzpicture}}
 - \;
\hackcenter{ \begin{tikzpicture} [scale=.7]
\draw[thick, double, ->] (-1.5,0) to (-1.5,2.5) ;
  \node at (-1.5,-.2) { $a $};
      \node at (-1.7,.5) { $i $};
\draw[thick, double, <-] (0,0) to (0,2.5) ;
\filldraw  [black] (-1.5,1.5) circle (2.5pt);
 \node at (-1.9,1.5) { $\varepsilon_1 $};
  \node at (0,-.2) { $b$};
      \node at (-.2,.5) { $i $};
 \node at (.7, 2.2) { $\lambda $};
\end{tikzpicture}}
  = \;\;
t_{ij}^{-1} t_{ji}    \;
\left( \;
\hackcenter{ \begin{tikzpicture} [scale=.7]
\draw[thick, double, ->] (-1.5,0) to (-1.5,2.5) ;
  \node at (-1.5,-.2) { $a $};
      \node at (-1.7,.5) { $i $};
\draw[thick, double, <-] (0,0) to (0,2.5) ;
  \node at (0,-.2) { $b$};
      \node at (-.2,.5) { $i $};
 \draw[blue]  (1.5,1) arc (360:180:.45cm) [thick];
 \draw[blue,<-](1.5,1) arc (0:180:.45cm) [thick];
 \filldraw  [black] (.7,.75) circle (2.5pt);
 \node at (1.0,.3) {\tiny $\ast+1$};
   \node at (1.5,1.55) { $j $};
 \node at (.7, 2.2) { $\lambda $};
\end{tikzpicture}}
\;\; - \;\; \;
\hackcenter{ \begin{tikzpicture} [scale=.7]
\draw[thick, double, ->] (-1.5,0) to (-1.5,2.5) ;
  \node at (-1.5,-.2) { $a $};
      \node at (-1.7,.5) { $i $};
\draw[thick, double, <-] (0,0) to (0,2.5) ;
  \node at (0,-.2) { $b$};
      \node at (-.2,.5) { $i $};
 \draw[blue]  (-2.5,1) arc (360:180:.45cm) [thick];
 \draw[blue,<-](-2.5,1) arc (0:180:.45cm) [thick];
 \filldraw  [black] (-3.3,.75) circle (2.5pt);
 \node at (-3.0,.3) {\tiny $\ast+1$};
      \node at (-2.4,1.55) { $j $};
 \node at (.9, 2.2) { $\lambda $};
\end{tikzpicture}}
\; \right).
\end{align}
\end{lemma}

\begin{proof}
Using \eqref{eq:X0} the result follows.
\end{proof}

\subsection{Braid group actions}  

We now recall how a categorical $\cal{U}_Q$ gives rise to a braid group action in the homotopy category (c.f. \cite{CautisKam}). To do this we consider the following complexes (usually called Rickard complexes).

\begin{equation}\label{Rickard}
\tau_{i}  \onel =
\begin{cases}
\left[ \cal{E}_i^{(-\l_i)} \onel  \xrightarrow{d^+} \cdots \xrightarrow{d^+}  \cal{E}_i^{(-\l_i+k)} \cal{F}_i^{(k)}  \onel  \la k \ra \xrightarrow{d^+} \cdots \right]
& \text{if } \l_i \leq 0 \smallskip\\
\left[ \cal{F}_i^{(\l_i)} \onel  \xrightarrow{d^+} \cdots \xrightarrow{d^+}  \cal{E}_i^{(k)}\cal{F}_i^{(\l_i+k)} \onel  \la k \ra \xrightarrow{d^+} \cdots \right]
& \text{if } \l_i \geq 0
\end{cases}\end{equation}
where the left most term is in cohomological degree zero and with
\begin{equation} \label{eq:dab}
d^+ \; = \; d^+_{(a,b)} \; := \;  (-1)^{\l_i+1}
 \hackcenter{ \begin{tikzpicture} [scale=.7]
 \draw[thick, directed=0.6] (1.5,1.5) .. controls ++(0,-.65) and ++(0,-.65) .. (0,1.5) ;
\draw[thick, double, ->] (0,0) to (0,2.5) ;
  \node at (0,-.2) { $a$};
      \node at (-.2,.5) { $i $};
\draw[thick, double, <- ] (1.5,0) to (1.5,2.5) ;
    \node at (1.5,-.2) {$b$};
        \node at (1.8,.5) { $i $};
       \node at (1.9,2.2) { $\lambda $};
\end{tikzpicture}}
\end{equation}
where $b-a =\l_i$. Similarly we define
\[
 \onel \tau'_{i}  :=
\begin{cases}
\left[ \cdots \xrightarrow{d^-} \onel  \cal{E}_i^{(k)} \cal{F}_i^{(-\l_i+k)}  \la -k \ra \xrightarrow{d^-} \cdots \xrightarrow{d^-} \onel  \cal{F}^{(-\l_i)}  \right]   & \text{if } \l_i \leq 0 \smallskip\\
\left[ \cdots \xrightarrow{d^-} \onel \cal{E}_i^{(\l_i +k)} \cal{F}_i^{(k)}   \la -k \ra \xrightarrow{d^-} \cdots \xrightarrow{d^-} \onel  \cal{E}^{(\l_i)}   \right] & \text{if } \l_i \geq 0 \\
\end{cases}
\]
where the right most term is in cohomological degree zero and with
\begin{equation}
d^{-} \; = \; d^-_{(a,b)} \; = \;
 \hackcenter{ \begin{tikzpicture} [scale=.8]
 \draw[thick, directed=0.6] (0,1) .. controls ++(0,.65) and ++(0,.65) .. (1.5,1) ;
\draw[thick, double, ->] (0,0) to (0,2.5) ;
  \node at (0,-.2) { $a$};
      \node at (-.2,.3) { $i $};
\draw[thick, double, <-] (1.5,0) to (1.5,2.5) ;
    \node at (1.5,-.2) {$b$};
        \node at (1.8,.5) { $i $};
       \node at (-.4,2.2) { $\lambda $};
\end{tikzpicture}}
\end{equation}
where $a-b =\l_i$.

\begin{remark}
In \cite{Cautis} a slightly different form of the complex $\tau_i\onel$ is used when $\l_i \geq 0$, namely
\[
\left[ \cal{F}_i^{(\l_i)} \onel  \xrightarrow{\hat{d}^{+}} \cdots \xrightarrow{\hat{d}^{+}_{}}\cal{F}_i^{(\l_i+k)}\cal{E}_i^{(k)} \onel  \la k \ra \xrightarrow{\hat{d}^{+}} \cdots \right]
\qquad \hat{d}_{}^{+} \;\; = \;\;
\hackcenter{ \begin{tikzpicture} [scale=.7]
 \draw[thick, directed=0.6] (0,1.5) .. controls ++(0,-.65) and ++(0,-.65) .. (1.5,1.5) ;
\draw[thick, double, <-] (0,0) to (0,2.5) ;
  \node at (0,-.2) { $\l_i+k$};
      \node at (-.2,.5) { $i $};
\draw[thick, double, -> ] (1.5,0) to (1.5,2.5) ;
    \node at (1.5,-.2) {$k$};
        \node at (1.8,.5) { $i $};
       \node at (1.9,2.2) { $\lambda $};
\end{tikzpicture}} .
\qquad  
\]
However, this version of the complex is isomorphic to the complex in \eqref{Rickard} by Lemma 5.3 and Proposition 5.10 of \cite{KLMS} with the isomorphism in  homological degree $a$ given by
\begin{equation}
\cal{F}_i^{(\l_i+k)} \cal{E}_i^{(k)} \onel
\xrightarrow{
\xy
(0,0)*{
\begin{tikzpicture}[scale=.75]
	\draw[ultra thick,->,black] (1,0) to [out=90,in=270] (0,1.5);
	\draw[ultra thick, black,->] (1,1.5) to [out=270,in=90] (0,0);
\end{tikzpicture}};
\endxy
}
\cal{E}_i^{(k)}\cal{F}_i^{(\l_i+k)} \onel
\xrightarrow{
(-1)^{k(\l_i+k)}
\xy
(0,0)*{
\begin{tikzpicture}[scale=.75]
	\draw[ultra thick,<-,black] (1,0) to [out=90,in=270] (0,1.5);
	\draw[ultra thick, black,<-] (1,1.5) to [out=270,in=90] (0,0);
\end{tikzpicture}};
\endxy
}
\cal{F}_i^{(\l_i+k)} \cal{E}_i^{(k)} \onel .
\end{equation}
\end{remark}

\begin{definition} \label{def:integrable}
Let $\cal{K}$ be a graded $\Bbbk$-linear $2$-category. A $2$-representation $F\maps \cal{U}_Q \to \cal{K}$ is said to be \emph{integrable} if it satisfied the following  additional properties:
\begin{enumerate}

    \item the objects $ \bm{\l}  := F(\lambda)$ of $\cal{K}$ are  zero for all but finitely many $\l$,

    \item the hom categories $\Hom_{\cal{K}}(\bm{\l},\bm{\l'})$ are idempotent complete with finite dimensional hom spaces in each degree, in other words, hom categories have the Krull-Schmidt property,

    \item \label{eq:finite} for any weight $\l$, $\Hom(\sf{1}_{\bm{\l}},\sf{1}_{\bm{\l}} \la \ell \ra)=0$ if $\ell <0$  and is one-dimensional if $\ell = 0$ and $\sf{1}_{\bm{\l}} \neq 0$, where $\sf{1}_{\bm{\l}}$ denotes the identity morphism on $F(\l)$.
\end{enumerate}
\end{definition}

Note that since $\cal{K}$ is idempotent complete, a $2$-representation $F \maps \cal{U}_Q \to \cal{K}$ extends uniquely to a $2$-representations $\dot{F} \maps \dot{\cal{U}}_Q \to \cal{K}$.  Furthermore, passing to homotopy categories of complexes, this induces a $2$-functor $\dot{F} \maps \Kom(\dot{\cal{U}}_Q) \to \textsf{Kom}(\cal{K})$.

\begin{proposition}[Theorem 6.3 \cite{CautisKam}]
\label{originalbraidprop}
Let $F\maps \cal{U}_Q \to \cal{K}$ be an integrable $2$-representation. Then the complexes $\tau_i, \tau_i'$ satisfy the following braid group relations in $\Kom(\cal{K})$
\begin{align*}
\tau_i \tau_i' \1_\l \cong \1_\l \cong \tau_i' \tau_i \1_\l & \\
\tau_i \tau_j \1_\l \cong \tau_j \tau_i \1_\l & \ \ \text{ if } \la i,j \ra = 0 \\
\tau_i \tau_j \tau_i \1_\l \cong \tau_j \tau_i \tau_j \1_\l & \ \ \text{ if } \la i,j \ra = -1.
\end{align*}
\end{proposition}

\begin{remark}
The results in \cite[Section 6]{CautisKam} actually require something much weaker than a full action of the $2$-category $\cal{U}_Q$.  However, results of \cite{Cautis-rigid} together with ~\cite{Brundan2}, see also \cite{CLau}, give the full action of $\cal{U}_Q$.
\end{remark}

\section{Bubbles and homotopies} \label{sec-bubbles}

\subsection{Notation}

For $j \in I$ we write
\begin{equation}
\label{defofbibi'}
b_{j} \; =\;  b_{j}(\l)
\;\; := \;\;
\hackcenter{ \begin{tikzpicture} [scale=.8]
 \draw[purple]  (-.75,1) arc (360:180:.45cm) [thick];
 \draw[purple, <-](-.75,1) arc (0:180:.45cm) [thick];
     \filldraw  [black] (-1.55,.75) circle (2.5pt);
        \node at (-1.3,.3) { $\scriptstyle \ast + 1$};
        \node at (-1.4,1.7) { $j $};
 \node at (-.2,1.5) { $\lambda $};
\end{tikzpicture}}
\end{equation}
Note that any $2$-morphism in $\End^2_{\cal{U}_Q(\mf{g})}(\1_{\l})$ can be written as a linear combination  of such $b_j(\l)$~\cite[Proposition 3.6]{KL3}.

 \begin{definition}
For any $\l \in X$, define a bilinear pairing $(\cdot, \cdot)_Q$ between $\End^2_{\cal{U}_Q(\mf{g})}(\1_{\l})$ and ${\mf h}^*_\Bbbk$ by
\begin{equation} \label{eq:Qform}
 (b_j, \alpha_i)_Q \;\; =\;\;  (-1)^{(\alpha_i,\alpha_{j})}v_{j i} ( \alpha_{j},\alpha_i).
\end{equation}
\end{definition}

It is an easy exercise to check that the Weyl group $W_{\mf{g}}$ acts on $\bigoplus_{\l}\End^2_{\cal{U}_Q(\mf{g})}(\1_{\l})$ via
\[
 s_i(b_j(\l)) := b_j(s_i(\l)) - (b_j,\alpha_i)_Q b_i(s_i(\l)).
\]
Note that we can simplify this equation by omitting the domains to obtain
$$s_i(b_j) := b_j - (b_j,\alpha_i)_Q b_i.$$

\begin{remark} \label{weylactionstandard}
Consider the choice of coefficients for $\cal{U}_Q(\mf{g})$ where
\[
t_{ij} =
\left\{
  \begin{array}{ll}
     -1 & \hbox{ if $i \cdot j =-1$ and $i \longrightarrow j$} \\
      1 & \hbox{ otherwise. }
  \end{array}
\right.
\]
This choice of parameters $Q$ corresponds is the natural choice of KLR algebra $R_Q$ that describes $\Ext$-algberas between perverse sheaves on the quiver variety~\cite{VV}.  In this case we have $v_{ij}=t_{ij}^{-1} t_{ji} = -1$ whenever $(\alpha_i, \alpha_j)=-1$ and subsequently $(b_{j},\alpha_i)_Q = ( \alpha_{j}, \alpha_i )$ for all $i,j \in I$.

Subsequently,  the Weyl group action above gives an action on $\bigoplus_{\l}\End^2_{\cal{U}_Q(\mf{g})}(\1_{\l})$  via:
\begin{equation}
s_i \maps
\hackcenter{ \begin{tikzpicture} [scale=.8]
 \draw  (1.5,1) arc (360:180:.45cm) [thick];
 \draw[,<-](1.5,1) arc (0:180:.45cm) [thick];
 \filldraw  [black] (.7,.75) circle (2.5pt);
 \node at (.6,.3) {\tiny $\ast+1$};
 \node at (.6,1.55) { $j$};
       \node at (1.8,2.0) { $\l$};
\end{tikzpicture}}
\mapsto
\hackcenter{ \begin{tikzpicture} [scale=.8]
 \draw[blue]  (1.5,1) arc (360:180:.45cm) [thick];
 \draw[blue,<-](1.5,1) arc (0:180:.45cm) [thick];
 \filldraw  [black] (.7,.75) circle (2.5pt);
 \node at (.6,.3) {\tiny $\ast+1$};
 \node at (.6,1.55) { $j$};
       \node at (1.8,2.0) { $s_i(\l)$};
\end{tikzpicture}}
\; - \alpha_i^{\vee}(\alpha_j) \;
 \hackcenter{ \begin{tikzpicture} [scale=.8]
 \draw  (1.5,1) arc (360:180:.45cm) [thick];
 \draw[,<-](1.5,1) arc (0:180:.45cm) [thick];
 \filldraw  [black] (.7,.75) circle (2.5pt);
 \node at (.6,.3) {\tiny $\ast+1$};
 \node at (.6,1.55) { $i$};
       \node at (1.8,2.0) { $s_i(\l)$};
\end{tikzpicture}}
\end{equation}
In this case, the natural map $ \mf{h}^{\ast}  \to \End^2_{\cal{U}_Q(\mf{g})}(\onel)$ given by $\alpha_j \mapsto b_j$ intertwines the Weyl group actions.
\end{remark}

\subsection{The homotopies}

The following is the main result of this section. It is an immediate corollary of the subsequent three Propositions.

\begin{theorem}
\label{squareflopcor}
For all $i \in I$  and $b \in \End^2_{\cal{U}_Q(\mf{g})}(\1_\l)$ we have
\begin{equation} \label{eq:lsquare}
(b,\alpha_i)_Q [d^-d^+ + d^+d^-]
=  s_i(b) \cdot \Id_{\tau_i \1_{\lambda}}  - \Id_{\tau_i \1_{\lambda}}\cdot b.
\end{equation}
\end{theorem}

\begin{proposition} \label{prop:ihomo}
For all $\l$ and $i \in I$ there is a homotopy equivalence
\begin{equation} \label{eq:i-homotopy}
  \Id_{\tau_i\onel} \cdot \left( \;
  \hackcenter{ \begin{tikzpicture} [scale=.7]
 \draw  (1.5,1) arc (360:180:.45cm) [thick];
 \draw[,<-](1.5,1) arc (0:180:.45cm) [thick];
 \filldraw  [black] (.7,.75) circle (2.5pt);
 \node at (.6,.3) {\tiny $\ast+1$};
 \node at (.6,1.55) { $i$};
       \node at (1.8,2.0) { $\lambda $};
\end{tikzpicture}}
  \;\right)
  \;  \simeq \;
\left( \;
  \hackcenter{ \begin{tikzpicture} [scale=.7]
 \draw  (1.5,1) arc (360:180:.45cm) [thick];
 \draw[,<-](1.5,1) arc (0:180:.45cm) [thick];
 \filldraw  [black] (.7,.75) circle (2.5pt);
 \node at (.6,.3) {\tiny $\ast+1$};
 \node at (.6,1.55) { $i$};
       \node at (1.8,2.0) { $s_i(\lambda)$};
\end{tikzpicture}} \; \right) \cdot \Id_{\tau_i\onel}
\end{equation}
given by $2d^-$.
\end{proposition}

\begin{proof}
Lemmas \ref{squarefloplem} and \ref{lem:A2} with $b-a=\lambda_i$ imply
\begin{align}
& 2 d_{(a-1,b-1)}^+ d_{(a,b)}^- +  2 d_{(a+1,b+1)}^-  d_{(a,b)}^+ \;\;  = \;\;
2(-1)^{\l_i+1}  (-1)^{-\l_i}
\sum_{ q+y+z=1}
(-1)^{y}
\hackcenter{ \begin{tikzpicture} [scale=.7]
\draw[thick, double, ->] (.5,0) to (.5,2.5) ;
   \filldraw  [black] (.5,1.5) circle (2.5pt);
 \node at (.1,1.55) { $\varepsilon_y $};
  \node at (.5,-.2) { $a $};
      \node at (.3,.3) { $i $};
      \node at (1.8,.5) { $i $};
\draw[thick, double, <-] (1.5,0) to (1.5,2.5) ;
 \filldraw  [black] (1.5,1.5) circle (2.5pt);
    \node at (1.9,1.5) { $h_q $};
    \node at (1.5,-.2) {$b$};
 \draw  (-.5,1) arc (360:180:.45cm) [thick];
 \draw[,<-](-.5,1) arc (0:180:.45cm) [thick];
    \filldraw  [black] (-1.3,.75) circle (2.5pt);
     \node at (-.9,1.65) { $i $};
     \node at (-1.2,.3) {\tiny $\ast+z$};
       \node at (2.2,2.2) { $\lambda $};
\end{tikzpicture}}
\\ \label{eq:i-homotopyA}
& \qquad   = \;\;- 2
\hackcenter{ \begin{tikzpicture} [scale=.7]
\draw[thick, double, ->] (.5,0) to (.5,2.5) ;
  \node at (.5,-.2) { $a $};
      \node at (.3,.3) { $i $};
      \node at (1.8,.5) { $i $};
\draw[thick, double, <-] (1.5,0) to (1.5,2.5) ;
    \node at (1.5,-.2) {$b$};
 \draw  (-.5,1) arc (360:180:.45cm) [thick];
 \draw[,<-](-.5,1) arc (0:180:.45cm) [thick];
    \filldraw  [black] (-1.3,.75) circle (2.5pt);
     \node at (-.9,1.65) { $i $};
     \node at (-1.2,.3) {\tiny $\ast+1$};
       \node at (2.2,2.2) { $\lambda $};
\end{tikzpicture}}
 - \; 2   \left(\;
\hackcenter{ \begin{tikzpicture} [scale=.7]
\draw[thick, double, ->] (.5,0) to (.5,2.5) ;
  \node at (.5,-.2) { $a $};
      \node at (.3,.3) { $i $};
      \node at (1.8,.5) { $i $};
\draw[thick, double, <-] (1.5,0) to (1.5,2.5) ;
 \filldraw  [black] (1.5,1.5) circle (2.5pt);
    \node at (1.9,1.5) { $\varepsilon_1 $};
    \node at (1.5,-.2) {$b$};
       \node at (2.2,2.2) { $\lambda $};
\end{tikzpicture}}
\;  -  \; \hackcenter{ \begin{tikzpicture} [scale=.7]
\draw[thick, double, ->] (.5,0) to (.5,2.5) ;
   \filldraw  [black] (.5,1.5) circle (2.5pt);
 \node at (.1,1.55) { $\varepsilon_1 $};
  \node at (.5,-.2) { $a $};
      \node at (.3,.3) { $i $};
      \node at (1.8,.5) { $i $};
\draw[thick, double, <-] (1.5,0) to (1.5,2.5) ;
    \node at (1.5,-.2) {$b$};
       \node at (2.2,2.2) { $\lambda $};
\end{tikzpicture}}
\; \right)
\end{align}
(with $(-1)^{\l_i+1}$ coming from the differential and $(-1)^{-\l_i}$ coming from (the square flop) Lemma~\ref{squarefloplem}).
The result then follows using \eqref{eq:dot4bub} since
\begin{align}
& \hackcenter{ \begin{tikzpicture} [scale=.7]
\draw[thick, double, ->] (-1.5,0) to (-1.5,2.5) ;
  \node at (-1.5,-.2) { $a $};
      \node at (-1.7,.5) { $i $};
\draw[thick, double, <-] (0,0) to (0,2.5) ;
\filldraw  [black] (0,1.5) circle (2.5pt);
 \node at (-.4,1.5) { $\varepsilon_1 $};
  \node at (0,-.2) { $b$};
      \node at (-.2,.5) { $i $};
 %
 \node at (.9, 2.2) { $\lambda $};
\end{tikzpicture}}
\; - \;
\hackcenter{ \begin{tikzpicture} [scale=.7]
\draw[thick, double, ->] (-1.5,0) to (-1.5,2.5) ;
  \node at (-1.5,-.2) { $a $};
      \node at (-1.7,.5) { $i $};
\draw[thick, double, <-] (0,0) to (0,2.5) ;
\filldraw  [black] (-1.5,1.5) circle (2.5pt);
 \node at (-1.9,1.5) { $\varepsilon_1 $};
  \node at (0,-.2) { $b$};
      \node at (-.2,.5) { $i $};
 %
 \node at (.9, 2.2) { $\lambda $};
\end{tikzpicture}}
  = \;\;
\frac{1}{2}   \left( \;
\hackcenter{ \begin{tikzpicture} [scale=.7]
\draw[thick, double, ->] (-1.5,0) to (-1.5,2.5) ;
  \node at (-1.5,-.2) { $a $};
      \node at (-1.7,.5) { $i $};
\draw[thick, double, <-] (0,0) to (0,2.5) ;
  \node at (0,-.2) { $b$};
      \node at (-.2,.5) { $i $};
 \draw[black]  (1.5,1) arc (360:180:.45cm) [thick];
 \draw[black,<-](1.5,1) arc (0:180:.45cm) [thick];
 \filldraw  [black] (.7,.75) circle (2.5pt);
 \node at (1.0,.3) {\tiny $\ast+1$};
   \node at (1.4,1.55) { $i $};
 \node at (.9, 2.2) { $\lambda $};
\end{tikzpicture}}
\;\; - \;\;
\hackcenter{ \begin{tikzpicture} [scale=.7]
\draw[thick, double, ->] (-1.5,0) to (-1.5,2.5) ;
  \node at (-1.5,-.2) { $a $};
      \node at (-1.7,.5) { $i $};
\draw[thick, double, <-] (0,0) to (0,2.5) ;
  \node at (0,-.2) { $b$};
      \node at (-.2,.5) { $i $};
 \draw[black]  (-2.5,1) arc (360:180:.45cm) [thick];
 \draw[black,<-](-2.5,1) arc (0:180:.45cm) [thick];
 \filldraw  [black] (-3.3,.75) circle (2.5pt);
 \node at (-3.0,.3) {\tiny $\ast+1$};
      \node at (-2.4,1.55) { $i $};
 \node at (.9, 2.2) { $\lambda $};
\end{tikzpicture}}
\; \right)
\end{align}
so that \eqref{eq:i-homotopyA} is equal to $(s_i (b_i))\Id_{\cal{E}_i^{(a)}\cal{F}_i^{(b)}\onel} - \Id_{\cal{E}_i^{(a)}\cal{F}_i^{(b)}\onel} b_i$.
\end{proof}

\begin{proposition}
For all $\l$ and $i,j \in I$ such that $i \cdot j = -1$ there is a homotopy equivalence
\begin{equation} \label{eq:i-homotopy}
 \Id_{\tau_i\onel} \cdot \left( \;
  \hackcenter{ \begin{tikzpicture} [scale=.7]
 \draw[blue]  (1.5,1) arc (360:180:.45cm) [thick];
 \draw[blue,<-](1.5,1) arc (0:180:.45cm) [thick];
 \filldraw  [black] (.7,.75) circle (2.5pt);
 \node at (.6,.3) {\tiny $\ast+1$};
 \node at (.6,1.55) { $j$};
       \node at (1.8,2.0) { $\lambda $};
\end{tikzpicture}}
  \;\right)
  \;  \simeq \;
\left(
 \hackcenter{ \begin{tikzpicture} [scale=.7]
 \draw[blue]  (1.5,1) arc (360:180:.45cm) [thick];
 \draw[blue,<-](1.5,1) arc (0:180:.45cm) [thick];
 \filldraw  [black] (.7,.75) circle (2.5pt);
 \node at (.6,.3) {\tiny $\ast+1$};
 \node at (.6,1.55) { $j$};
       \node at (1.8,2.0) { $s_i(\lambda)$};
\end{tikzpicture}}
\; - \;t_{ij} t_{j i}^{-1}
 \hackcenter{ \begin{tikzpicture} [scale=.7]
 \draw  (1.5,1) arc (360:180:.45cm) [thick];
 \draw[,<-](1.5,1) arc (0:180:.45cm) [thick];
 \filldraw  [black] (.7,.75) circle (2.5pt);
 \node at (.6,.3) {\tiny $\ast+1$};
 \node at (.6,1.55) { $i$};
       \node at (1.8,2.0) { $s_i(\lambda)$};
\end{tikzpicture}}
\right) \cdot \Id_{\tau_i\onel}
\end{equation}
given by $t_{ji}^{-1} t_{ij} d^- = v_{ji} d^-$.
\end{proposition}

\begin{proof}
By Lemma~\ref{squarefloplem} and Lemma~\ref{lem:A2}
\begin{align} \nn
& v_{ji} d_{(a-1,b-1)}^+ d_{(a,b)}^-  + v_{ji} d_{(a+1,b+1)}^-  d_{(a,b)}^+  \;\; =\;\;
-v_{ji}
\sum_{ q+y+z=\beta}
(-1)^{y}
\hackcenter{ \begin{tikzpicture} [scale=.7]
\draw[thick, double, ->] (.5,0) to (.5,2.5) ;
   \filldraw  [black] (.5,1.5) circle (2.5pt);
 \node at (.1,1.55) { $\varepsilon_y $};
  \node at (.5,-.2) { $a $};
      \node at (.3,.3) { $i $};
      \node at (1.8,.5) { $i $};
\draw[thick, double, <-] (1.5,0) to (1.5,2.5) ;
 \filldraw  [black] (1.5,1.5) circle (2.5pt);
    \node at (1.9,1.5) { $h_q $};
    \node at (1.5,-.2) {$b$};
 \draw  (-.5,1) arc (360:180:.45cm) [thick];
 \draw[,<-](-.5,1) arc (0:180:.45cm) [thick];
    \filldraw  [black] (-1.3,.75) circle (2.5pt);
     \node at (-.9,1.65) { $i $};
     \node at (-1.2,.3) {\tiny $\ast+z$};
       \node at (2.2,2.2) { $\lambda $};
\end{tikzpicture}}
\\
&  \quad  = \;\; -v_{ji}
\hackcenter{ \begin{tikzpicture} [scale=.7]
\draw[thick, double, ->] (.5,0) to (.5,2.5) ;
  \node at (.5,-.2) { $a $};
      \node at (.3,.3) { $i $};
      \node at (1.8,.5) { $i $};
\draw[thick, double, <-] (1.5,0) to (1.5,2.5) ;
    \node at (1.5,-.2) {$b$};
 \draw  (-.5,1) arc (360:180:.45cm) [thick];
 \draw[,<-](-.5,1) arc (0:180:.45cm) [thick];
    \filldraw  [black] (-1.3,.75) circle (2.5pt);
     \node at (-.9,1.65) { $i $};
     \node at (-1.2,.3) {\tiny $\ast+1$};
       \node at (2.2,2.2) { $\lambda $};
\end{tikzpicture}}
 - \; v_{ji}   \left(\;
\hackcenter{ \begin{tikzpicture} [scale=.7]
\draw[thick, double, ->] (.5,0) to (.5,2.5) ;
  \node at (.5,-.2) { $a $};
      \node at (.3,.3) { $i $};
      \node at (1.8,.5) { $i $};
\draw[thick, double, <-] (1.5,0) to (1.5,2.5) ;
 \filldraw  [black] (1.5,1.5) circle (2.5pt);
    \node at (1.9,1.5) { $\varepsilon_1 $};
    \node at (1.5,-.2) {$b$};
       \node at (2.2,2.2) { $\lambda $};
\end{tikzpicture}}
\;  -  \; \hackcenter{ \begin{tikzpicture} [scale=.7]
\draw[thick, double, ->] (.5,0) to (.5,2.5) ;
   \filldraw  [black] (.5,1.5) circle (2.5pt);
 \node at (.1,1.55) { $\varepsilon_1 $};
  \node at (.5,-.2) { $a $};
      \node at (.3,.3) { $i $};
      \node at (1.8,.5) { $i $};
\draw[thick, double, <-] (1.5,0) to (1.5,2.5) ;
    \node at (1.5,-.2) {$b$};
       \node at (2.2,2.2) { $\lambda $};
\end{tikzpicture}}
\; \right) \nn
\\
& \quad \refequal{\text{Lem \ref{lem:dots-j}}} \;
-v_{ji}
\hackcenter{ \begin{tikzpicture} [scale=.7]
\draw[thick, double, ->] (.5,0) to (.5,2.5) ;
  \node at (.5,-.2) { $a $};
      \node at (.3,.3) { $i $};
      \node at (1.8,.5) { $i $};
\draw[thick, double, <-] (1.5,0) to (1.5,2.5) ;
    \node at (1.5,-.2) {$b$};
 \draw  (-.5,1) arc (360:180:.45cm) [thick];
 \draw[,<-](-.5,1) arc (0:180:.45cm) [thick];
    \filldraw  [black] (-1.3,.75) circle (2.5pt);
     \node at (-.9,1.65) { $i $};
     \node at (-1.2,.3) {\tiny $\ast+1$};
       \node at (2.2,2.2) { $\lambda $};
\end{tikzpicture}}
\; - \; v_{ji} v_{ij} \;
\left( \;
\hackcenter{ \begin{tikzpicture} [scale=.7]
\draw[thick, double, ->] (-1.5,0) to (-1.5,2.5) ;
  \node at (-1.5,-.2) { $a $};
      \node at (-1.7,.5) { $i $};
\draw[thick, double, <-] (0,0) to (0,2.5) ;
  \node at (0,-.2) { $b$};
      \node at (-.2,.5) { $i $};
 \draw[blue]  (1.5,1) arc (360:180:.45cm) [thick];
 \draw[blue,<-](1.5,1) arc (0:180:.45cm) [thick];
 \filldraw  [black] (.7,.75) circle (2.5pt);
 \node at (1.0,.3) {\tiny $\ast+1$};
   \node at (1.4,1.55) { $j $};
 \node at (.7, 2.2) { $\lambda $};
\end{tikzpicture}}
\;\; - \;\; \;
\hackcenter{ \begin{tikzpicture} [scale=.7]
\draw[thick, double, ->] (-1.5,0) to (-1.5,2.5) ;
  \node at (-1.5,-.2) { $a $};
      \node at (-1.7,.5) { $i $};
\draw[thick, double, <-] (0,0) to (0,2.5) ;
  \node at (0,-.2) { $b$};
      \node at (-.2,.5) { $i $};
 \draw[blue]  (-2.5,1) arc (360:180:.45cm) [thick];
 \draw[blue,<-](-2.5,1) arc (0:180:.45cm) [thick];
 \filldraw  [black] (-3.3,.75) circle (2.5pt);
 \node at (-3.0,.3) {\tiny $\ast+1$};
      \node at (-2.4,1.55) { $j$};
 \node at (.8, 2.2) { $\lambda $};
\end{tikzpicture}}
\; \right) \nn
\\ \nn
& \quad = (s_i (b_j))\Id_{\cal{E}_i^{(a)}\cal{F}_i^{(b)}\onel} - \Id_{\cal{E}_i^{(a)}\cal{F}_i^{(b)}\onel} b_j
\end{align}
where we used that $v_{ij} v_{ji} = 1$ in the second to last equality.
\end{proof}

\begin{proposition}
For all $\l$ and $i,k \in I$ such that $i \cdot k = 0$ the following maps are equal (and in particular homotopic)
\begin{equation} \label{eq:i-homotopy}
  \Id_{\tau_i\onel} \cdot \left( \;
  \hackcenter{ \begin{tikzpicture} [scale=.7]
 \draw[green]  (1.5,1) arc (360:180:.45cm) [thick];
 \draw[green,<-](1.5,1) arc (0:180:.45cm) [thick];
 \filldraw  [black] (.7,.75) circle (2.5pt);
 \node at (.6,.3) {\tiny $\ast+1$};
 \node at (.6,1.55) { $k$};
       \node at (1.8,2.0) { $\lambda $};
\end{tikzpicture}}
  \;\right)
  \;  \simeq \;
\left(
 \hackcenter{ \begin{tikzpicture} [scale=.7]
 \draw[green]  (1.5,1) arc (360:180:.45cm) [thick];
 \draw[green,<-](1.5,1) arc (0:180:.45cm) [thick];
 \filldraw  [black] (.7,.75) circle (2.5pt);
 \node at (.6,.3) {\tiny $\ast+1$};
 \node at (.6,1.55) { $k$};
       \node at (2.0,2.0) { $s_i(\lambda) $};
\end{tikzpicture}}
\right) \cdot \Id_{\tau_i\onel} .
\end{equation}
\end{proposition}

\begin{proof}
This is immediate from the bubble slide rules
\begin{equation} \label{eq:Xk}
\hackcenter{ \begin{tikzpicture} [scale=.7]
\draw[thick, double, ->] (2.5,0) to (2.5,2.5) ;
  \node at (2.5,-.2) { $a $};
  \node at (2.7,.4) { $i $};
 \draw[green]  (1.7,1) arc (360:180:.45cm) [thick];
 \draw[green,<-](1.7,1) arc (0:180:.45cm) [thick];
 \filldraw  [black] (.9,.75) circle (2.5pt);
 \node at (1.1,.3) {\tiny $\ast+1$};
 \node at (1.6,1.65) { $k $};
 \node at (3.25, 2.2) { $\lambda $};
\end{tikzpicture}}
=\;\;
\hackcenter{ \begin{tikzpicture} [scale=.7 ]
\draw[thick, double, ->] (0,0) to (0,2.5) ;
  \node at (0,-.2) { $a $};
      \node at (-.2,.4) { $i $};
 \draw[green]  (1.5,1) arc (360:180:.45cm) [thick];
 \draw[green,<-](1.5,1) arc (0:180:.45cm) [thick];
 \filldraw  [black] (.7,.75) circle (2.5pt);
          \node at (1.4,1.65) { $k$};
 \node at (1.0,.3) {\tiny $\ast+1$};
 \node at (1.15, 2.2) { $\lambda $};
\end{tikzpicture}}
\qquad \quad
\hackcenter{ \begin{tikzpicture} [scale=.7 ]
\draw[thick, double, <-] (0,0) to (0,2.5) ;
  \node at (0,-.2) { $b $};
      \node at (-.2,.55) { $i $};
 \draw[green]  (1.5,1) arc (360:180:.45cm) [thick];
 \draw[green,<-](1.5,1) arc (0:180:.45cm) [thick];
 \filldraw  [black] (.7,.75) circle (2.5pt);
          \node at (1.4,1.65) { $k$};
 \node at (1.0,.3) {\tiny $\ast+1$};
 \node at (1.15, 2.2) { $\lambda $};
\end{tikzpicture}}
\; \; =\;\;
\hackcenter{ \begin{tikzpicture} [scale=.7]
\draw[thick, double, <-] (2.5,0) to (2.5,2.5) ;
  \node at (2.5,-.2) { $b $};
  \node at (2.7,.55) { $i $};
 \draw[green]  (1.7,1) arc (360:180:.45cm) [thick];
 \draw[green,<-](1.7,1) arc (0:180:.45cm) [thick];
 \filldraw  [black] (.9,.75) circle (2.5pt);
 \node at (1.1,.3) {\tiny $\ast+1$};
 \node at (1.6,1.65) { $k $};
 \node at (3.25, 2.2) { $\lambda $};
\end{tikzpicture}}
\end{equation}
so that no homotopy is required.
\end{proof}

\subsection{Some remarks}
There is a natural internal braid group action on the quantum group that interacts with the braid group action on integrable modules that is categorified by Rickard complexes and which generalizes the action on bubbles discussed above.  In \cite{AL-braid} $2$-functors $\cal{T}_i$ categorifying the internal braid group action were defined from $\cal{U}_Q$ to its homotopy category of complexes $\Kom(\cal{U}_Q)$.  It was shown in \cite{AL-braid2} that these $2$-functors interact with Rickard complexes giving $2$-natural isomorphisms
\begin{equation}
  \beth\maps \tau_{i} \onelp (-)\onel \To \cal{T}_{i}(-)\tau_{i}\onel
\end{equation}
where $(-)$ denotes an input from $\cal{U}_Q$.  This amounts to defining for each one morphism $\onelp x\onel \in \Kom(\Ucat_Q)$ a chain homotopy equivalence
\begin{equation}
 \beth_{x\onel} \maps \tau_{i} \onelp x\onel \To \cal{T}_{i}(x\onel)\tau_{i}\onel,
\end{equation}
and for each $2$-morphism $f \maps \onelp x\onel \to \onelp y\onel$ in $\Kom(\Ucat_Q)$ a chain map $\cal{T}_i(f) \maps \cal{T}_i(\onelp x\onel) \to \cal{T}_i(\onelp y\onel)$ giving a commutative diagram
\begin{equation}\label{eq:comm-square}
\xy
 (-20,8)*+{\tau_{i}y \onel}="tl";
 (20,8)*+{\cal{T}_i(y\onel)\tau_i\onel}="tr";
 (-20,-8)*+{\tau_{i}x \onel}="bl";
 (20,-8)*+{\cal{T}_i(x\onel)\tau_{i} \onel}="br";
    {\ar^{\Id_{\tau_{i}}f} "bl";"tl"};
    {\ar_{\beth_{x\onel}} "bl";"br"};
    {\ar_{\cal{T}_{i}(f)\Id_{\tau_{i}\onel}} "br";"tr"};
    {\ar^{\beth_{y\onel}} "tl";"tr"};
\endxy
\end{equation}
 in $\Kom(\Ucat_Q)$.  For identity 1-morphisms $\onel$, the maps $\beth_{\onel}$ are identities since $\cal{T}_i(\onel) = \1_{s_i(\l)}$ and $\tau_i \onel = \1_{s_i(\l)}\tau_i$.   One can show that the braid group action via 2-functors $\cal{T}_i$ on the space of endomorphisms $\bigoplus_{\l}\End^2(\onel)$ factors through the corresponding Weyl group, so that $\cal{T}_i^2 = \Id$ on this space. A bubble $b_j(\l)$ defines a degree $2$ endomorphism  of $\onel$ and the homotopy for the naturality square above defines a chain homotopy from $\tau_i\onel b_j(\l)$ to $\cal{T}_i(b_j(\l))\tau_i\onel$.  This homotopy equivalence agrees with the homotopies defined in Theorem~\ref{squareflopcor}.

From this perspective, there is nothing special about degree $2$ bubbles.  One can similarly show that the braid group action induced by the 2-functors $\cal{T}_i$ on $\bigoplus_{\l}\End^{2k}(\onel)$ factors through the Weyl group and the commutative square \eqref{eq:comm-square} gives homotopies from $\tau_i\onel x \onel$ to $\cal{T}_i(x\onel)\tau_i\onel$ for any 2-morphisms $x\onel \in \End^{2k}(\onel)$.  Further, the homotopies for degree $2k$ bubbles can be chosen so that they square to zero, so they naturally fit into the framework of curved complexes studied in the next section.

\section{Curved Rickard complexes} \label{sec-curved}

\subsection{2-categories and curved complexes} \label{subsecgencurved}

In this subsection we review some definitions and properties of curved complexes in the context of $2$-categories.

Let us fix a $\Bbbk$-linear $2$-category $\cal{K}$. Recall that $\Kom(\cal{K})$ is used to denote the homotopy category of complexes in $\cal{K}$. The objects here are the same but the $1$-morphisms are complexes $(V,\Delta)$ of $1$-morphisms in $\cal{K}$
\begin{equation}\label{eq:seq}
V_1 \xrightarrow{\Delta} V_2 \xrightarrow{\Delta} \dots \xrightarrow{\Delta} V_n
\end{equation}
which we consider up to homotopy equivalence. One can rewrite this more compactly as $V = \oplus_i V_i [-i]$, where the $[\cdot]$ denotes a shift in homological grading, and $\Delta: V \to V[1]$. Notice that, by definition, $\Delta^2=0$ in this case. We now explain how one can relax this condition to obtain the homotopy category $\tKom(\cal{K})[u]$ of curved complexes.

First, for two objects $A,B \in \cal{K}$ we denote by ${}_B\cal{K}_A := \Hom_\cal{K}(A,B)$ the $1$-category of maps between $A$ and $B$. This category is naturally an $(\End(\1_B),\End(\1_A))$-bilinear category. Here $\1_A$ is the identify functor of $A$ and $\End(\1_A)$ is its endomorphism algebra (and similarly for $\1_B$).

For $A,B \in \cal{K}$ as above, let $(z_B,z_A) \in (\End^2(\1_B),\End^2(\1_A))$.  We would like to consider a sequence of maps
\begin{equation}\label{eq:seq2}
V_1 \substack{\xrightarrow{\Delta^+} \\ \xleftarrow[\Delta^-]{}} V_2 \substack{\xrightarrow{\Delta^+} \\ \xleftarrow[\Delta^-]{}} \dots \substack{\xrightarrow{\Delta^+} \\ \xleftarrow[\Delta^-]{}} V_n
\end{equation}
such that $(\Delta^+ + \Delta^-)^2 = z_B 1_V -1_V z_A$.
We will often abuse notation and write $z_B 1_V -1_V z_A$ simply as $z_B-z_A$.
This is all fine except that the maps $\Delta^-$ have homological degree $-1$ instead of $+1$.

To fix this we enlarge our $2$-category by introducing a formal variable $u$ of homological degree $2$.  More precisely, for any $\Bbbk$-linear $2$-category $\cD$ we can consider the $2$-category $\cD[u]$ where the objects and $1$-morphisms are the same but where the $2$-morphisms are formally tensored with $\Bbbk[u]$. In other words, if $V,V'$ are $1$-morphisms then
$$\Hom_{\cD[u]}(V,V') = \Hom_{\cD}(V,V') \otimes_{\Bbbk} \Bbbk[u].$$
In particular, any degree zero map in $\cD[u]$ can be written as a sum $\phi = \sum_{i \geq 0}\phi_i u^i$ for some maps $\phi_i$ of degree $-2i$ in $\cD$.
In this language, we can combine the maps in (\ref{eq:seq2}) as
$$\Delta := \Delta^+ + u \Delta^- : V \mapsto V[1].$$

\begin{definition} \label{defoffac}
Using the setup above, a $(z_B,z_A)$-factorization in ${}_B\cal{K}_A$ is a map $\Delta: V \mapsto V[1]$ such that $\Delta^2 = (z_B-z_A)u$. Such a pair $(V,\Delta)$ is called a curved complex in ${}_B\cal{K}_A$ with curvature $(z_B,z_A)$ and connection $\Delta$.
\end{definition}

\begin{remark}
In \cite{GHy}, the construction analogous to Definition \ref{defoffac} would be called a {\it strict} $y$-ification.
\end{remark}

\begin{definition}
For two $(z_B,z_A)$-factorizations $(V,\Delta)$ and $(V',\Delta')$, a morphism $f \colon (V,\Delta) \rightarrow (V', \Delta')$ is a morphism (which we always assume to be degree zero) $f \colon V \rightarrow V'$ such that $f \circ \Delta = \Delta' \circ f$.
Two morphisms $f,g \colon (V,\Delta) \rightarrow (V', \Delta')$ of $(z_B,z_A)$-factorizations are homotopic if there exists $H \colon (V,\Delta) \rightarrow (V', \Delta')$ such that $f-g=H \circ \Delta + \Delta' \circ H$. We will denote the homotopy category of factorizations in $\cal{K}$ by $\tKom(\cal{K})[u]$.
\end{definition}

For objects $A,B,C \in \cal{K}$ we have natural maps
$${}_C\cal{K}_B \otimes {}_B\cal{K}_A \to {}_C\cal{K}_A.$$
For objects $(V, \Delta) \in {}_B\cal{K}_A$ and $(V', \Delta') \in {}_C\cal{K}_B$ let us denote the image by $V'*V \in {}_C\cal{K}_A$ (although later we will drop the $*$ in order to simplify notation). It has a connection $\Delta '* \Delta$ defined by $(\Delta '* \Delta) (x' * x)=\Delta'(x') * x + (-1)^i x' * \Delta(x)$ where $x'$ is in homological degree $i$ of $V'$.

\begin{lemma}
If $(V,\Delta) \in {}_B\cal{K}_A$ is a $(z_B,z_A)$-factorization and $(V',\Delta) \in {}_C\cal{K}_B$ is a $(z_C,z_B)$-factorization then $(V'*V,\Delta'*\Delta)$ is a $(z_C,z_A)$-factorization.
\end{lemma}
\begin{proof}
Assume that $x'$ lies in homological degree $i$ of $V'$.
Then we have the following chain of equalities
\begin{align*}
(\Delta'*\Delta)^2 (x' * x) &=(\Delta'*\Delta)(\Delta'(x') * x +(-1)^i x' * \Delta(x)) \\
&= (\Delta')^2(x') * x + (-1)^{i+1} \Delta'(x') * \Delta(x) + (-1)^i \Delta'(x') * \Delta(x)
+ x' * \Delta^2(x) \\
&= (z_C - z_B)u x' * x + (z_B - z_A)u x' * x \\
&= (z_C-z_A)u x' * x.
\end{align*}
\end{proof}

The following result is a straightforward extension of the analogous classical fact in homological algebra.

\begin{proposition} \label{conelem}
\cite[Lemma 2.5]{GHy}
If $f \colon (V_1,\Delta_1) \rightarrow (V_2,\Delta_2) $ is an isomorphism of $(z_B,z_A)$-factorizations in $\tKom(\cal{K})[u] $, then $Cone(f) \cong 0 $ in $\tKom(\cal{K})[u]$.
\end{proposition}


The following is a generalization of \cite[Lemma 2.19]{GHy} where it is assumed that the null-homotopy $h$ satisfies $h^2=0$.  For abelian 2-categories this assumption suffices, but in the additive setting (in particular for additive 2-categories) we require a new proof.

\begin{proposition} \label{contractlem}
If $(V,\Delta^+) \in \Kom^-(\cal{K})$ is contractible, then any deformation $(V,\Delta^+ + u \Delta^-)$ is contractible in $\tKom(\cal{K})[u]$.
\end{proposition}

\begin{proof}
Let $(V,\Delta^+) \in \Kom^-(_{B}\cal{K}_{A})$ be a contractible complex and suppose that it deforms to a factorization $(V,\Delta)$.   Then $\Delta^2 = u z$ for some central element $z$  and $\Delta^+ \Delta^- + \Delta^- \Delta^+ = z$.  Fix a null-homotopy $h \in \End^{-1}(X)$ satisfying $\Delta^+ h+h \Delta^+ = \Id_{V}$.
Then the following hold by induction:
\begin{align}
\left(h \Delta^-\right)^{\ell} \Delta^+ - \Delta^+ \left( h\Delta^-\right)^{\ell}
& =
-\Delta^-\left( h\Delta^-\right)^{\ell-1}
+ z \sum_{j=0}^{\ell-1}  \left( h\Delta^-\right)^{\ell-1-j} h \left( h\Delta^-\right)^{j}
\label{XXX2}
\\ \label{XXX1}
\Delta^+ \left( h\Delta^-\right)^{\ell}h  + \left( h\Delta^-\right)^{\ell} h \Delta^+ 
& =
 \left( \Delta^- h\right)^{\ell}
+ \left( h\Delta^-\right)^{\ell}
- z \sum_{j=0}^{\ell-1} \left( h\Delta^-\right)^{\ell-1-j} h \left( h\Delta^-\right)^{j}h .
\end{align}

Then define a null-homotopy of the deformed complex $(V,\Delta)$ via the map
 \begin{equation}
 H \;:= \;
 \sum_{\ell \geq 0} \sum_{j_1+ j_2 = \ell}(-1)^{\ell} (h \Delta^-)^{j_1} h (h\Delta^-)^{j_2} u^{\ell},
\end{equation}
which is a finite sum since $V$ is bounded below and $(h\Delta^-) \in \End^{-2}(V)$.
Write $H^{\ell}$ for the term of $H$ corresponding to $u^{\ell}$.
To see that
$
\Delta H + H \Delta = \Id_{V}
$
observe that by using  \eqref{XXX2} on the second term we have
\begin{align} \label{XXX3}
\Delta^+ H^{\ell}
 +H^{\ell} \Delta^+
& =
(-1)^{\ell}\sum_{j_1+ j_2 = \ell} \Delta^+ (h \Delta^-)^{j_1} h (h\Delta^-)^{j_2} u^{\ell}
+
(-1)^{\ell} \sum_{j_1+ j_2 = \ell}(h \Delta^-)^{j_1} h (h\Delta^-)^{j_2} \Delta^+ u^{\ell}
\\
& \refequal{\eqref{XXX2}} (-1)^{\ell}\sum_{j_1+ j_2 = \ell} \Delta^+ (h \Delta^-)^{j_1} h (h\Delta^-)^{j_2} u^{\ell}
 \nn \\
& \;\;
 +  (-1)^{\ell} \sum_{j_1+ j_2 = \ell}(h \Delta^-)^{j_1} h \Delta^+ (h\Delta^-)^{j_2} u^{\ell}
\nn \\
& \quad -
(-1)^{\ell} \sum_{j_1+ j_2 = \ell}(h \Delta^-)^{j_1} h  \Delta^-(h\Delta^-)^{j_2-1}  u^{\ell}
\nn \\
& \qquad +
(-1)^{\ell} z\sum_{j_1+ j_2 = \ell}\sum_{k=0}^{j_2-1} (h \Delta^-)^{j_1} h \left( h\Delta^-\right)^{j_2-1-k} h \left( h\Delta^-\right)^{k} u^{\ell} \nn.
\end{align}
Now simplify the second term after the equality:
\begin{align}
  &(-1)^{\ell} \sum_{j_1+ j_2=\ell}(h \Delta^-)^{j_1} h \Delta^+   (h\Delta^-)^{j_2} u^{\ell} \hspace{2in}
\nn \\
  &\qquad \refequal{\eqref{XXX1}}
- (-1)^{\ell} \sum_{j_1+ j_2 = \ell} \Delta^+ (h \Delta^-)^{j_1} h(h\Delta^-)^{j_2} u^{\ell}
\nn
\\
& \qquad \qquad+ (-1)^{\ell} \sum_{j_1+ j_2 = \ell}(\Delta^-h )^{j_1}    (h\Delta^-)^{j_2} u^{\ell}
+ (-1)^{\ell} \sum_{j_1+ j_2 = \ell}(h \Delta^-)^{j_1}    (h\Delta^-)^{j_2} u^{\ell}
\nn \\
& \qquad \qquad  -
(-1)^{\ell} z \sum_{j_1+ j_2=\ell} \sum_{k=0}^{j_1-1}   (h \Delta^-)^{j_1-1-k} h (h \Delta^-)^k   h  (h\Delta^-)^{j_2} u^{\ell}
\nn.
\end{align}
Then putting these terms back into \eqref{XXX3} gives
\begin{align}
\Delta^+ H^{\ell} + H^{\ell} \Delta^+
& =
-(-1)^{\ell} \sum_{j_1+ j_2 = \ell}(h \Delta^-)^{j_1} (h\Delta^-)^{j_2}  u^{\ell}
\nn \\
& \qquad +
(-1)^{\ell} z\sum_{j_1+ j_2 = \ell}\sum_{k=0}^{j_2-1} (h \Delta^-)^{j_1} h \left( h\Delta^-\right)^{j_2-1-k} h \left( h\Delta^-\right)^{k} u^{\ell}
\nn \\
& \qquad + (-1)^{\ell} \sum_{j_1+ j_2 = \ell}(\Delta^-h )^{j_1}    (h\Delta^-)^{j_2} u^{\ell}
+ (-1)^{\ell} \sum_{j_1+ j_2 = \ell}(h \Delta^-)^{j_1}    (h\Delta^-)^{j_2} u^{\ell}
\nn \\
& \qquad  -
(-1)^{\ell} z \sum_{j_1+ j_2=\ell} \sum_{k=0}^{j_1-1}   (h \Delta^-)^{j_1-1-k}h (h \Delta^-)^k h   (h\Delta^-)^{j_2} u^{\ell}
\end{align}
Simplifying the above we get,
\begin{align}
\Delta^+ H^{\ell} + H^{\ell} \Delta^+
& = (-1)^{\ell} \sum_{j_1+ j_2 = \ell}(\Delta^-h )^{j_1}    (h\Delta^-)^{j_2} u^{\ell}
\nn \\
& \qquad +
(-1)^{\ell} z\sum_{j_1=0}^{\ell-1}\sum_{k=0}^{\ell-j_1-1} (h \Delta^-)^{j_1} h \left( h\Delta^-\right)^{\ell-j_1-1-k} h \left( h\Delta^-\right)^{k} u^{\ell}
\nn \\
& \qquad  -
(-1)^{\ell} z \sum_{j_2=0}^{\ell-1} \sum_{k=0}^{\ell-j_2-1}   (h \Delta^-)^{\ell-j_2-1-k}h (h \Delta^-)^k h   (h\Delta^-)^{j_2} u^{\ell}
\end{align}
so that
\begin{align} \label{eq:commdH}
\Delta^+ H^{\ell} + H^{\ell} \Delta^+
& = (-1)^{\ell} \sum_{j_1+ j_2 = \ell}(\Delta^-h )^{j_1}    (h\Delta^-)^{j_2} u^{\ell}.
\end{align}

For the remaining piece $u\Delta^-$ of $\Delta = \Delta^+ + u \Delta^-$ observe
\begin{align}
\left[u\Delta^-, H^{\ell} \right]
& =
  \sum_{j_1+ j_2 = \ell}(-1)^{\ell} \Delta^- (h \Delta^-)^{j_1} h (h\Delta^-)^{j_2} u^{\ell+1}
 +
  \sum_{j_1+ j_2 = \ell}(-1)^{\ell} (h \Delta^-)^{j_1} h (h\Delta^-)^{j_2} \Delta^- u^{\ell+1}
 \nn \\
 & =
  \sum_{j_1+ j_2 = \ell}(-1)^{\ell}  (\Delta^-h )^{j_1+1}   (h\Delta^-)^{j_2} u^{\ell+1}
 +
  (-1)^{\ell} (h \Delta^-)^{\ell} h  \Delta^- u^{\ell+1}
\nn \\
 & =
  -\sum_{j_1+ j_2 = \ell+1}(-1)^{\ell+1}  (\Delta^-h )^{j_1}   (h\Delta^-)^{j_2} u^{\ell+1}.
\label{eq:commDH}
\end{align}
Combing \eqref{eq:commdH} and \eqref{eq:commDH} we get
$[\Delta^+ + u\Delta^-, H] = \Id_{V}$.
\end{proof}

\begin{proposition} \label{extensionprop}
\cite[Proposition 2.20]{GHy}
Let $(C,\Delta^+), (C', \Delta'^+) \in \Kom({}_B\cal{K}_A)$ be invertible complexes and $\phi_0 \colon C \rightarrow C'$ be a homotopy equivalence.  Then $\phi_0$ extends to a homotopy equivalence
$\phi \colon (C,\Delta) \rightarrow (C',\Delta')$ of $(z_B,z_A)$-factorizations where
$(C,\Delta)$ and $(C',\Delta')$ are deformations of $C$ and $C'$ respectively.
\end{proposition}

\begin{proof}
Since $C$ and $C'$ are invertible and homotopically equivalent, we have
\begin{equation*}
\Hom(C,C') \cong \Hom(\Id,C' \otimes C^{-1}) \cong \End(\1_B).
\end{equation*}
Thus any chain map in $\Hom(C,C')$ of non-zero homological degree is null-homotopic.  Now we will show how to extend $\phi_0$ to a chain map
$\phi = \sum_{i \ge 0} \phi_i u^i$
between deformations where the homological degree of $\phi_i$ is $-2i$.
Note that $\Delta=\Delta^+ + \Delta^- u$ and $\Delta'=\Delta'^+ + \Delta'^- u$.

Assume that we have constructed $\phi_i$ for all $i<k$ and thus that for such $i$ we may assume
\begin{equation} \label{phiassump}
\Delta' \phi_i = \phi_i \Delta
\end{equation}
which is equivalent to assuming that for all $i < k$ that
\begin{equation} \label{phiassump2}
\Delta'^+ \phi_i - \phi_i \Delta^+ + \Delta'^- \phi_{i-1} - \phi_{i-1} \Delta^- =0.
\end{equation}
Then consider the element
\begin{equation*}
\theta_k=\Delta'^-  \phi_{k-1} - \phi_{k-1}  \Delta^-.
\end{equation*}
We will now check that $\theta_k$ is a chain map between $C$ and $C'$.
Since the homological degree of $\theta_k$ is $1-2k$, and the homological degrees of $\Delta^+$ and $\Delta'^+$ are odd as well, we must verify  $\Delta'^+  \theta_k =-\theta_k  \Delta^+$.

\begin{equation*} \label{chaincheck1}
\Delta'^+ \theta_k + \theta_k \Delta^+ =
\Delta'^+ \Delta'^- \phi_{k-1} - \phi_{k-1} \Delta^- \Delta^+
- \Delta'^+ \phi_{k-1} \Delta^-
+ \Delta'^- \phi_{k-1} \Delta^+.
\end{equation*}
Using \eqref{phiassump2} on each of the last two terms in \eqref{chaincheck1}, we get
\begin{align} \label{chaincheck2}
\Delta'^+ \theta_k + \theta_k \Delta^+
=
&(\Delta'^+ \Delta'^- \phi_{k-1} - \phi_{k-1} \Delta^- \Delta^+ )+ \\ \nonumber
&( \Delta'^- \Delta'^+ \phi_{k-1}
+ \Delta'^- \Delta'^- \phi_{k-2}
- \Delta'^- \phi_{k-2} \Delta^- ) + \\ \nonumber
&(- \phi_{k-1} \Delta^+ \Delta^-
+ \Delta'^- \phi_{k-2} \Delta^-
- \phi_{k-2} \Delta^- \Delta^- ) .
\end{align}
Using the fact that $\Delta^2 = \Delta'^2=(z_B-z_A)$, we get that \eqref{chaincheck2} is equal to
\begin{equation*}
(z_B-z_A) \phi_{k-1} - \phi_{k-1}(z_B-z_A) =0
\end{equation*}
since $\phi_{k-1}$ is comprised of $(\End(\1_B),\End(\1_A))$-bimodule homomorphisms.

Since the homological degree of $\theta_k$ is $1-2k$ and it is a chain map between invertible, equivalent complexes, by the above reasoning, $\theta_k$ must be null-homotopic.  Thus there exists a map $h_k \colon C \rightarrow C'$ of homological degree $-2k$ such that
\begin{equation*}
\theta_k=\Delta'^+ h_k-  h_k \Delta^+.
\end{equation*}
Now define $\phi_k=-h_k$. Then the part of $\Delta' \circ \phi-\phi \circ \Delta$ of homological degree $1-2k$ is
\begin{align*}
\Delta'^+ \phi_k - \phi_k \Delta^+ + \Delta'^- \phi_{k-1} - \phi_{k-1} \Delta^- = \Delta'^+ \phi_k - \phi_k \Delta^+ + \theta_k =-(\Delta'^+ h_k - h_k \Delta^+) + \theta_k = 0.
\end{align*}
Thus building $\phi$ in this way, we see that it is a chain map between deformations.

By Proposition \ref{contractlem}, $Cone(\phi)$ is contractible so $\phi$ is a homotopy equivalence of deformations.
\end{proof}

\subsection{Curved Rickard complexes}

We take our $2$-category $\cal{K}$ to be $\Ucat_Q$ and consider the corresponding homotopy category of curved complexes $\tKom(\Ucat_Q)[u]$ where $u$ is a formal indeterminate of bi-degree $[2]\la -2 \ra$.

\begin{remark}
In Section \ref{subsecgencurved}, $u$ has degree just $[2]$ but now, since our underlying $2$-category has an extra grading, we need to impose an additional grading on $u$.
\end{remark}

For a parameter $c \in \Bbbk$ we define
\begin{align*}
\tau_{i,c} \1_{\l} &:=
\begin{cases}
\left[ \sE_i^{(-\l_i)} \1_\l
\substack{\xrightarrow{d^+} \\ \xleftarrow[cud^-]{}}
{\sE}_i^{(-\l_i+1)} {\sF}_i^{(1)} \1_\l \langle 1 \rangle
\substack{\xrightarrow{d^+} \\ \xleftarrow[cud^-]{}}
\cdots
\substack{\xrightarrow{d^+} \\ \xleftarrow[cud^-]{}}
{\sE}_i^{(-\l_i+k)} {\sF}_i^{(k)} \1_{\l} \langle k \rangle
\substack{\xrightarrow{d^+} \\ \xleftarrow[cud^-]{}}
\cdots \right] &\text{ if } \l_i \le 0 \\
\left[ {\sF}_i^{(\l_i)} \1_{\l}
\substack{\xrightarrow{d^+} \\ \xleftarrow[cud^-]{}}
{\sE}_i^{(1)}{\sF}_i^{(\l_i+1)} \1_{\l} \langle 1 \rangle
\substack{\xrightarrow{d^+} \\ \xleftarrow[cud^-]{}}
\cdots
\substack{\xrightarrow{d^+} \\ \xleftarrow[cud^-]{}}
{\sE}_i^{(k)} {\sF}_i^{(\l_i+k)}\1_{\l} \langle k \rangle
\substack{\xrightarrow{d^+} \\ \xleftarrow[cud^-]{}}
\cdots \right]
&\text{ if } \l_i \ge 0
\end{cases}
\end{align*}
so that $\Delta = d^+ + cud^-$.  In particular, taking $c=0$ recovers the original Rickard complex. Note that for any value of $c$, $\tau_{i,c} \1_{\l}$ is a curved complex with connection $\Delta = d^+ + ucd^-$.  To see this, choose any $j \in I$ with $(b_j,\alpha_i)_Q \neq 0$ and use \eqref{eq:lsquare} to conclude that
\begin{equation}\label{eq:calc}
(d^+ + c d^-)^2
= c\left(d^+d^- + d^- d^+\right)
=  \frac{c}{(b_{j}, \alpha_i)_Q}\left(s_i(b_j) \cdot \Id_{\tau_i}  - \Id_{\tau_i}\cdot b_j \right).
\end{equation}

Similarly we define
\begin{align*}
\1_{\l} \tau'_{i,c} :=
\begin{cases}
\left[
\cdots
\substack{\xrightarrow{d^-} \\ \xleftarrow[cud^+]{}}
\1_{\l} {\sE}_i^{(k)}  {\sF}_i^{(-\l_i+k)} \langle -k \rangle
\substack{\xrightarrow{d^-} \\ \xleftarrow[cud^+]{}}
\cdots
\substack{\xrightarrow{d^-} \\ \xleftarrow[cud^+]{}}
\1_{\l} {\sE}_i^{(1)} {\sF}_i^{(-\l_i+1)} \langle -1 \rangle
\substack{\xrightarrow{d^-} \\ \xleftarrow[cud^+]{}}
\1_{\l} {\sF}_i^{(-\l_i)} \right]
\quad \quad &\text{ if } \l_i \le 0
\\
\left[ \cdots
\substack{\xrightarrow{d^-} \\ \xleftarrow[cud^+]{}}
\1_{\l} {\sE}_i^{(\l_i+k)}{\sF}_i^{(k)} \langle -k \rangle
\substack{\xrightarrow{d^-} \\ \xleftarrow[cud^+]{}}
\cdots
\substack{\xrightarrow{d^-} \\ \xleftarrow[cud^+]{}}
\1_{\l} {\sE}_i^{(\l_i+1)}{\sF}_i^{(1)} \langle -1 \rangle
\substack{\xrightarrow{d^-} \\ \xleftarrow[cud^+]{}}
\1_{\l} {\sE}_i^{(\l_i)} \right]
\quad \quad &\text{ if } \l_i \ge 0
\end{cases}
\end{align*}
so that $\Delta = d^- + cud^+$.

\begin{proposition} \label{prop:curved}
For any $b \in \End^2_{\cal{U}_Q(\mf{g})}(\1_\l)$   we have that $\tau_{i,(b,\alpha_i)_Q} \1_{\l}$ is an $(s_i(b), b)$-factorization and likewise $\1_{\l} \tau'_{i,(b,\alpha_i)_Q}$ is a $(b, s_i(b))$-factorization.
\end{proposition}
\begin{proof}
This follows immediately from Theorem~\ref{squareflopcor} and the subsequent calculation in (\ref{eq:calc}).
\end{proof}

\begin{remark}
We will write $\tau_{i,(b,\alpha_i)_Q} \1_\l$ to implicitly mean the curved complex which is a $(s_i(b),b)$-factorization.
\end{remark}

\begin{remark}
One can always specialize the formal variable $u$ to a scalar. However, this requires one to identify $[1]\la -1\ra$ with the trivial shift (which has the effect of killing one of the gradings).
\end{remark}

\begin{lemma}\label{lem:square}
If $c \ne 0$ then $\tau_{i,c} \1_\l \cong \tau'_{i,c} \1_\l$ inside the localized category $\tKom(\Ucat_Q)[u^\pm]$.
\end{lemma}
\begin{proof}
Suppose $\l_i \ge 0$ (the case $\l_i \le 0$ is similar). If we ignore the differentials then
$$\tau_{i,c} \1_\l \cong \bigoplus_{k \ge 0} \sE_i^{(k)} \sF_i^{(\l_i+k)} \1_\l [-k] \la k \ra \ \ \text{ and } \ \ \tau'_{i,c} \1_\l \cong \bigoplus_{k \ge 0} \sE_i^{(k)} \sF_i^{(\l_i+k)} \1_\l [k] \la -k \ra.$$
We define a map $\teta_i \colon \tau_{i,c} \1_\l \rightarrow \tau'_{i,c} \1_\l $ by using, for $k \ge 0$, the maps
$$(cu)^k: \sE_i^{(k)} \sF_i^{(\l_i+k)} [-k] \la k \ra \to \sE_i^{(k)} \sF_i^{(\l_i+k)} [k] \la -k \ra.$$
It is not hard to check that this defines a map of curved complexes. Similarly we can define $\teta_i': \tau'_{i,c} \1_\l \rightarrow \tau_{i,c} \1_\l$ by using
$$(cu)^{-k}: \sE_i^{(k)} \sF_i^{(\l_i+k)} [k] \la -k \ra \to \sE_i^{(k)} \sF_i^{(\l_i+k)} [-k] \la k \ra.$$
This is also easily seen to be a map of curved complexes. Moreover, $\teta_i$ and $\teta_i'$ are clearly inverses of each other.
\end{proof}

\begin{proposition}
\label{deformedbraidprop}
If $\cK$ is an integrable 2-representation of $\cal{U}_Q(\mf{g})$ then inside $\tKom(\cK)[u]$ the curved complexes $\tau_{i,c}$ and $\tau'_{i,c}$ satisfy the braid group relations of $\Br_{\mf{g}}$. More precisely, for any $b \in \End^2( \1_\l )$, we have
\begin{align*}
\tau'_{i,(s_i(b), \alpha_i)} \tau_{i,(b, \alpha_i)} \1_\l \cong \1_\l \cong \tau_{i,(s_i(b), \alpha_i)} \tau'_{i, (b, \alpha_i)} \1_\l & \\
\tau_{i,(s_j(b), \alpha_i)} \tau_{j, (b, \alpha_j)} \1_\l \cong \tau_{j,(s_i(b), \alpha_j)} \tau_{i,(b,\alpha_i)} \1_\l & \ \ \text{ if } \la i,j \ra = 0 \\
\tau_{i, (s_js_i(b), \alpha_i)} \tau_{j, (s_i(b), \alpha_j)} \tau_{i,(b,\alpha_i)} \1_\l \cong \tau_{j, (s_is_j(b), \alpha_j)} \tau_{i, (s_j(b), \alpha_i)} \tau_{j, (b,\alpha_j)} \1_\l & \ \ \text{ if } \la i,j \ra = -1.
\end{align*}
where we suppress the subscript of $Q$ from the pairing $(\cdot,\cdot)_Q$ for readability.
\end{proposition}
\begin{proof}
Proposition \ref{originalbraidprop} gives homotopy equivalences of undeformed complexes for each braid relation.  By Proposition~\ref{extensionprop} these homotopies extend uniquely to homotopies of the corresponding curved complexes.
\end{proof}

Finally, we have the following relatively straightforward result which we will use later.

\begin{lemma}\label{deformedEFlemma}
For $b \in \End^2_{\cal{U}_Q(\mf{g})}( \1_\l )$, the $1$-morphisms $\cE_i^{(k)} \1_\l$ and $\1_\l \cF_i^{(k)}$ (thought of as complexes with only one term) are $(b,b)$-factorizations if $(b,\alpha_i)_Q = 0$.
\end{lemma}
\begin{proof}
One must check that acting on the left by $b$ is the same as acting on the right by $b$. The condition $(b,\alpha_i)_Q = 0$ guarantees this, see \eqref{eq:Xk}.
\end{proof}

\section{Application: deformed $\sl_m$ homology} \label{sec-slm}
Since the $2$-categories $\cal{U}_Q(\mf{sl}_m)$ and $\cal{U}_{Q'}(\mf{sl}_m)$ are isomorphic for all choice of scalars $Q$ and $Q'$ (see \cite[Theorem 3.5]{Lau-param}) throughout this section and the next we fix the choice of scalars from Remark~\ref{weylactionstandard} so that $(b_j,\alpha_i)_Q =  ( \alpha_j, \alpha_i )$.

\subsection{Background}\label{sec:background1}

We briefly review the construction of $\sl_m$ homology from \cite{Cautis}. The starting point is the $2$-category $\Ucat_Q$ where the Cartan data is that of $\sl_{2N}$ for some large $N$. The $N$ will depend on the link (the more complicated the presentation of the link the bigger the $N$) or we can consider the limit and take the more canonical choice $N = \infty$.

For any fixed $d \in \Z$ a weight $\l=(\l_1, \dots ,\l_{2N-1})$ in $\Ucat_Q(\mf{sl}_{2N})$ corresponds to a sequences $\uk = (k_1, \dots, k_{2N})$ of integers $k_i \in \Z$ determined by
\begin{align}
\l_i &= k_{i+1} - k_i \nn
\\
\sum_i k_i &= d
\end{align}
(when such a solution exists). We will use $\l$ and $\uk$ interchangeably.

For our purposes we set $d = mN$ and take $\Ucat_Q^m$ to be the quotient category which kills any weight $\uk$ where either $k_i < 0$ or $k_i > m$ for some $i$. This is equivalent to killing any weight which is zero in the $\sl_{2N}$ irreducible representation $V_{m \Lambda_N}$ with highest weight $m \Lambda_N$.

In this notation, the roots correspond to
$$\alpha_i = (0,\dots,0,-1,1,0,\dots,0)$$
where the $-1$ occurs in position $i$ and $m \Lambda_N = (\u0,\um) = (0,\dots,0,m,\dots,m)$ where there are a total of $N$ $0$'s and $N$ $m$'s. The Weyl group, which can be identified with $S_{2N}$, then acts by permuting these sequences in the usual way. As usual we denote by $\dot{\Ucat}_Q^m$ the Karoubi envelope of $\Ucat_Q^m$.  In this way $\dot{\Ucat}_Q^m$ is an integrable $2$-representation in the sense of Definition~\ref{def:integrable}.

Given a weight $\uk$, denote by $\rho(\uk)$ the sequence obtained by removing all $k_j \in \{ 0, m \}$.  Let $S(\uk)$ denote the set of weights $\uk'$ such that $\rho(\uk) = \rho(\uk')$.

\begin{lemma} \cite[Lemma 7.1]{Cautis} \label{lem:permute}
For any non-zero weight $\uk$ as above and any $\uk' \in S(\uk)$ there is a canonical isomorphism $\1_{\uk} \to \1_{\uk'}$ in $\dot{\Ucat}_Q^m$.
\end{lemma}

\begin{proof}
The proof amounts to the fact that if $\uk' \in S(\uk)$, then there exists a sequence of elementary transpositions $s_i$ taking $\uk'$ to $\uk$ with the property that one of the entries being switched is in $\{0,m\}$. However, if $k_i$ or $k_{i+1}$ is in $\{0,m\}$, then $\tau_i^2\1_{\uk}  \cong \1_{\uk}$.
\end{proof}

\begin{lemma}\label{lem:homs}
The space of $1$-morphisms $\Hom_{\dot{\Ucat}_Q^m}(\1_{(\u0,\um)}, \1_{(\u0,\um)})$ is spanned by direct sums of $ \1_{(\u0,\um)}$ (together with shifts). Moreover,
$$\End^*_{\dot{\Ucat}_Q^m}( \1_{(\u0,\um)}) \cong \Bbbk[e_1, \dots, e_m]$$
where $e_j$ is the degree $2j$ fake bubble labeled by $N$.
\end{lemma}
\begin{proof}
The first claim is a consequence of the fact that $(\u0,\um)$ is a highest weight in $\dot{\Ucat}_Q^m$, which means that all $\cal{E}_i$ act by zero. From this it also follows that all counterclockwise bubbles labeled by $i \in I$ vanish in $\End^*_{\dot{\Ucat}_Q^m}( \1_{(\u0,\um)})$. Note that the $i$-labeled counterclockwise bubble with no dots has degree two if $i \neq N$ and has degree $2(m+1)$ if $i = N$.  Hence, in the quotient $\dot{\Ucat}_Q^m$, we have
\[
0 \;\; = \;\; \hackcenter{ \begin{tikzpicture} [scale=.8]
 \draw[<-]  (-.75,1) arc (360:180:.45cm) [thick];
 \draw[->](-.75,1) arc (0:180:.45cm) [thick];
        \node at (-1.4,1.7) { $i $};
 \node at (-.2,1.5) { $(\u0,\um)$};
\end{tikzpicture}}
\;\; = \;\;
\hackcenter{ \begin{tikzpicture} [scale=.8]
 \draw  (-.75,1) arc (360:180:.45cm) [thick];
 \draw[->](-.75,1) arc (0:180:.45cm) [thick];
     \filldraw  [black] (-1.55,.75) circle (2.5pt);
        \node at (-1.3,.3) { $\scriptstyle (-\l_i -1) +\l_i+ 1$};
        \node at (-1.4,1.7) { $i $};
 \node at (-.2,1.5) { $(\u0,\um)$};
\end{tikzpicture}}
\;\; =: \;\;
\left\{
  \begin{array}{ll}
   \hackcenter{ \begin{tikzpicture} [scale=.8]
 \draw  (-.75,1) arc (360:180:.45cm) [thick];
 \draw[->](-.75,1) arc (0:180:.45cm) [thick];
     \filldraw  [black] (-1.55,.75) circle (2.5pt);
        \node at (-1.3,.3) { $\scriptstyle \ast + 1$};
        \node at (-1.4,1.7) { $i $};
 \node at (-.2,1.5) { $(\u0,\um)$};
\end{tikzpicture}} & \hbox{if $i \neq N$;} \\
    \hackcenter{ \begin{tikzpicture} [scale=.8]
 \draw  (-.75,1) arc (360:180:.45cm) [thick];
 \draw[->](-.75,1) arc (0:180:.45cm) [thick];
     \filldraw  [black] (-1.55,.75) circle (2.5pt);
        \node at (-1.3,.3) { $\scriptstyle \ast + m+1$};
        \node at (-1.4,1.7) { $N$};
 \node at (-.2,1.5) { $(\u0,\um)$};
\end{tikzpicture}}, & \hbox{otherwise,}
  \end{array}
\right.
\] so that all positive degree $i$ labeled bubbles with $i \neq N$ vanish and all $N$-labeled bubbles of degree greater than $m$ also vanish.  The $N$-labeled clockwise bubbles of degrees $0$ to $m$ are all fake bubbles, so the vanishing of weights $\1_\uk = 0$ if $k_i \not\in [0,m]$ does not require these to be zero and they form a linearly independent set of endomorphisms in $\End^*_{\dot{\Ucat}_Q^m}( \1_{(\u0,\um)})$.
\end{proof}

Now consider an oriented link $L$ whose components are coloured by elements of $[0,m]$. A colour $k$ should be thought of as a labeling of the strand by the fundamental representation $V_{\Lambda_k}$ of $\sl_m$. We now explain how $L$ induces a $1$-morphism
$$\Psi(L) \in \Hom_{\Kom(\dot{\Ucat}_Q^m)}(\1_{(\u0,\um)}, \1_{(\u0,\um)}).$$

To obtain this $1$-morphism from $L$ we decompose the link into a composition of caps, cups and crossings as shown in Figure \ref{fig:1}.
At each level of this decomposition we can associate an object $\uk$ of $\dot{\Ucat}_Q^m$ where we add in 0's or $m$'s if needed so that $\uk \in \Z^{2N}$ with $\sum_{i}k_i = mN$.  By Lemma~\ref{lem:permute} there is a canonical isomorphism between any two objects $\uk$ associated to a given level of the decomposition of $L$.
 To a cap/cup we associate the $1$-morphisms
\begin{align}
\label{eq:cap} & \cE_i^{(k_i)}: (\dots,k_i,m-k_i,\dots) \to (\dots,0,m,\dots) \\
\label{eq:cup} & \cF_i^{(k_i)}: (\dots,0,m,\dots) \to (\dots,k_i,m-k_i,\dots)
\end{align}
where $i$ denotes the position of the cap/cup. To the four over crossings in Figure \ref{fig:1} we associate maps $\tau_i\1_{\uk}$, $\tau_i\1_{\uk} [k_i] \la -k_i \ra$, $\tau_i \1_{\uk}[m-k_{i+1}]\la -m+k_{i+1} \ra$ and $\tau_i\1_{\uk} [-k_{i+1}+k_i] \la k_{i+1}-k_i \ra$ respectively. The corresponding four under crossings are the associated inverse maps.

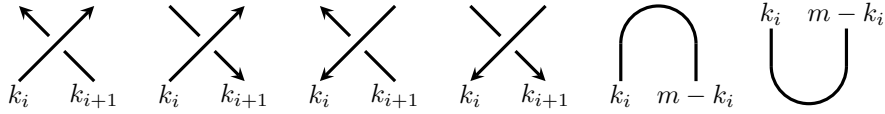
\begin{figure}[!h]
\centering
\begin{tikzpicture}[>=stealth]
\draw [->](0,0) -- (1,1) [very thick];
\draw [->](0.4,0.6) -- (0,1) [very thick];
\draw [-](1,0) -- (0.6,0.4) [very thick];
\draw [shift={+(0,-0.2)}](0,0) node {$k_i$};
\draw [shift={+(0,-0.2)}](1,0) node {$k_{i+1}$};

\draw [shift={+(2,0)}][->](0,0) -- (1,1) [very thick];
\draw [shift={+(2,0)}][-](0.4,0.6) -- (0,1) [very thick];
\draw [shift={+(2,0)}][<-](1,0) -- (0.6,0.4) [very thick];
\draw [shift={+(2,0)}][shift={+(0,-0.2)}](0,0) node {$k_i$};
\draw [shift={+(2,0)}][shift={+(0,-0.2)}](1,0) node {$k_{i+1}$};

\draw [shift={+(4,0)}][<-](0,0) -- (1,1) [very thick];
\draw [shift={+(4,0)}][->](0.4,0.6) -- (0,1) [very thick];
\draw [shift={+(4,0)}][-](1,0) -- (0.6,0.4) [very thick];
\draw [shift={+(4,0)}][shift={+(0,-0.2)}](0,0) node {$k_i$};
\draw [shift={+(4,0)}][shift={+(0,-0.2)}](1,0) node {$k_{i+1}$};

\draw [shift={+(6,0)}][<-](0,0) -- (1,1) [very thick];
\draw [shift={+(6,0)}][-](0.4,0.6) -- (0,1) [very thick];
\draw [shift={+(6,0)}][<-](1,0) -- (0.6,0.4) [very thick];
\draw [shift={+(6,0)}][shift={+(0,-0.2)}](0,0) node {$k_i$};
\draw [shift={+(6,0)}][shift={+(0,-0.2)}](1,0) node {$k_{i+1}$};

\draw [shift={+(8,.5)}](0,0) -- (0,-.5) [very thick];
\draw [shift={+(8,.5)}](1,0) -- (1,-.5)[-] [very thick];
\draw [shift={+(8,.5)}](0,0) arc (180:0:.5) [very thick];
\draw [shift={+(8,0)}][shift={+(0,-0.2)}](0,0) node {$k_i$};
\draw [shift={+(8,0)}][shift={+(0,-0.2)}](1,0) node {$m-k_i$};

\draw [shift={+(10,.2)}](0,0) -- (0,.5) [-][very thick];
\draw [shift={+(10,.2)}](1,0) -- (1,.5) [very thick];
\draw [shift={+(10,.2)}](0,0) arc (180:360:.5) [very thick];
\draw [shift={+(10,0)}][shift={+(0,.9)}](0,0) node {$k_i$};
\draw [shift={+(10,0)}][shift={+(0,.9)}](1,0) node {$m-k_i$};
\end{tikzpicture}
\caption{The cap and cup can have either orientation.}\label{fig:1}
\end{figure}
Composing these morphisms together gives us $\Psi(L)$. Note that, by Lemma \ref{lem:homs}, this is a complex with terms direct sums of $\1_{(\u0,\um)}$. To obtain the link invariant $H_{\sl_m}^{*,*}(L)$ we apply the functor $\Hom_{\dot{\Ucat}_Q^m}(\1_{(\u0,\um)}, \bullet)$. By Lemma \ref{lem:homs} this has the effect of replacing each summand $\1_{(\u0,\um)}$  with $\Bbbk$.  For more details see \cite[Section 7]{Cautis}.

\begin{remark}
If we apply the functor $\Hom^*_{\dot{\Ucat}_Q^m}(\1_{(\u0,\um)}, \bullet)$ (i.e. the direct sum over all degrees) then by Lemma \ref{lem:homs} this has the effect of replacing each $\1_{(\u0,\um)}$    with $\Bbbk[e_1, \dots, e_m]$ with $e_j$ corresponding to the fake bubble of degree $2j$, see Lemma~\ref{lem:homs}. The resulting homology is a deformation of $H_{\sl_m}^{*,*}(L)$ over $\Bbbk[e_1, \dots, e_m]$~\cite{Wu2,RoseW}.
\end{remark}

\subsection{Deformations}\label{sec:defs1}
To obtain the desired deformation of the construction above we will label crossings with curved complexes. More precisely, given $\uy = (y_1,\dots,y_{2N})$\footnote{Here our parameters $y_i$ are deformation parameters rather than formal variables as in \cite{GHy}} we will denote by $b_\uy \in {\mf{h}}^*_\Bbbk \cong \End^2_{\dot{\Ucat}_Q}(\1_\uk)$ the $2$-morphism determined by the conditions $( b_\uy, \alpha_i ) = -y_i+y_{i+1}$. In this notation, we have $( b_\uy, \uk ) = \uy \cdot \uk$.

We now work in the quotient $\dot{\Ucat}_Q^m$. We denote by $B_\uy$ the image of $b_\uy$ in the 2-representation $\dot{\Ucat}_Q^m$. Given a link $L$ we decompose it as before into cups, caps and crossings and associate to each the corresponding curved complexes. More precisely, to the first crossing in Figure \ref{fig:1} we associate $\tau_{i,-y_i+y_{i+1}} \1_{\uk}$ which is a $(B_{s_i(\uy)}, B_\uy)$-factorization (and similarly for the other crossings). Whenever we see a cap/cup we associate the curved complexes as in (\ref{eq:cap}) and (\ref{eq:cup}) {\it{and}} we impose the condition that $\uy \cdot \alpha_i = 0$ (namely $y_i=y_{i+1}$, cf. Lemma \ref{deformedEFlemma}).

\begin{theorem}\label{thm:deflink}
Suppose $L = L_1 \cup \dots \cup L_c$ is an oriented, coloured link with $c$ components. Then the construction above defines a family of link homologies $\widetilde{H}_{\sl_m}^{*,*}(L)$ parameterized by $\uz \in \Bbbk^c$. This homology has the usual splitting properties from Batson-Seed \cite{BS}. In particular, for distinct values of $\uz$ the deformed homology of $L$ is isomorphic to the homology of the disjoint union of its components, and moreover, there exists a spectral sequence starting with $H_{\sl_m}^{*,*}(L)$ and converging to $\widetilde{H}_{\sl_m}^{*,*}(L)$.
\end{theorem}
\begin{proof}
Consider a presentation $[L]$ of the link $L$ and a choice of $\uy$. The construction above imposes a set of linear relations that the $\uy$ should satisfy (one condition for each cap and cup). We say that $\uy$ is compatible with $[L]$ if it satisfies all these linear conditions.

In this way, for any compatible $\uy$, we obtain a $1$-morphism
\begin{equation} \label{eq:def-deformed}
  \Psi_\uy(L): \Hom_{\tKom(\dot{\Ucat}_Q^m)[u]}(\1_{(\u0,\um)}, \1_{(\u0,\um)}).
\end{equation}
Since, by Lemma \ref{lem:homs}, $\End^2( \1_{(\u0,\um)})$ is spanned by $e_1$ this curved complex is a $(ce_1,c'e_1)$-factorization for some $c,c' \in \Bbbk$. The fact that the description of $L$ begins with a cup and ends with a cap means that we impose the condition that $(c e_1, e_1) = 0 = (c' e_1, e_1)$. Thus $c=c'=0$ and $\Psi_\uy(L)$ is an actual (non-curved) complex.

Next we need to understand why the set of $\uy$ compatible with $[L]$ is indexed by $\Bbbk^c$. This is more easily realized if we present $L$ as a composition of simple cups, a braid $\beta$ and then simple caps. In the construction above this means a composition of the form
$$\1_{(0,m,0,m,\dots,0,m)} (\sE_1^{(\ell_1)} \sE_3^{(\ell_2)} \dots \sE_{2N-1}^{(\ell_N)}) \Psi(\beta) (\sF_1^{(k_1)} \sF_3^{(k_2)} \dots \sF_{2N-1}^{(k_{N})}) \1_{(0,m,0,m,\dots,0,m)}$$
where $k_1, \dots, k_N$ and $\ell_1, \dots, \ell_N$ are the colours of the strands of $\beta$ at the bottom and top respectively while $\Psi(\beta)$ is the curved complex associated to $\beta$ (a composition of $\tau_{i,c}$'s).

We start with a general deformation parameter $\uy=(y_1, \dots, y_{2N})$. The cups impose linear relations
\[
\uy \cdot \alpha_1 = 0, \;  \uy \cdot \alpha_3 = 0,\;  \dots, \; \uy \cdot \alpha_{2\ell-1} = 0
\]
which are equivalent to $y_1=y_2$, $y_3=y_4, \dots, y_{2\ell-1}=y_{2\ell}$. Now $\Psi(\beta)$ is a $(B_{\beta(\uy)}, B_\uy)$-factorization. This means that the caps impose the linear relations
\[
\beta(\uy) \cdot \alpha_1 = 0, \ \ \beta(\uy) \cdot \alpha_3 = 0, \ \ \dots, \ \ \beta(\uy) \cdot \alpha_{2N-1} = 0.
\]
It is easy to see from this description that $\Psi_\uy(L)$ is a $(B_{\beta(\uy)}, B_{\uy})$-factorization and that the linear relations on $\uy$ cut out exactly a family indexed by the components of $L$.

Finally, the splitting properties of $\widetilde{H}_{\sl_m}^{*,*}(L)$ as in Batson-Seed \cite{BS} are a consequence of Lemma \ref{lem:square}.
\end{proof}

\subsection{Clasps} \label{sec:clasps1}

The construction of $\widetilde{H}_{\sl_m}^{*,*}(L)$ above is for links $L$ coloured by fundamental representations $V_{\Lambda_k}$ of $\sl_m$. One way to generalize this to arbitrary representations $V_{\sum_i \Lambda_{k_i}}$ is by cabling and using clasps (projectors) $\cP$ corresponding to the composition
$$V_{\Lambda_{k_1}} \otimes \dots \otimes V_{\Lambda_{k_s}} \to V_{\sum_i \Lambda_{k_i}} \to V_{\Lambda_{k_1}} \otimes \dots \otimes V_{\Lambda_{k_s}}$$
where the first map is the natural projection while the second is the natural inclusion.

This approach was followed in \cite[Theorems 2.2, 2.3]{Cautis} where $\cP$ was defined as a limit $\lim_{\ell \to \infty} (\tau'_{\omega})^{2 \ell}$ where
$$\tau'_{\omega} = (\tau'_1 \dots \tau'_{s-1}) (\tau'_1 \dots \tau'_{s-2}) \dots (\tau'_1 \tau'_2) (\tau'_1)$$
is a full twist. Here, for notational simplicity, we assume that the clasp $\cP$ consists of cabling the first $s$ strands (otherwise the indices $1,\dots,s-1$ would need to be shifted). Also note that our notation is such that $\tau_i'$ in the current paper corresponds to ${\sf T}_i$ in \cite{Cautis}.

\begin{proposition}\label{prop:clasp}
For any $B_\uy \in \End_{\dot{\Ucat}_Q^m}^2(\1_\uk)$ the clasp $\cP \1_\uk \in \Kom^-(\dot{\Ucat}_Q^m)$ deforms to a $(B_\uy,B_\uy)$-factorization $\tcP \in {\tKomm}(\dot{\Ucat}_Q^m)[u]$.
\end{proposition}
\begin{proof}
We will prove that the corresponding limit $\lim_{\ell \to \infty} (\tau'_{\omega})^{2 \ell} \1_\uk$ in ${\tKomm}(\dot{\Ucat}_Q^m)[u]$ converges. The argument from \cite{Cautis} that this limit exists in $\Kom^-(\dot{\Ucat}_Q^m)$ relies on a connecting map $\eta: \1_\uk \to (\tau'_\omega)^2 \1_\uk$ (c.f. \cite[Section 5.2]{Cautis}) such that
$$(\tau'_{\omega})^{2 \ell} Cone(\eta) \1_\uk$$
is isomorphic to a complex supported in cohomological degrees $< -d(\ell)$ where $d(\ell) \to \infty$ as $\ell \to \infty$. So we just need to deform this story.

The map $\eta$ is constructed from simpler $\1_\uk \to (\tau'_i)^2 \1_\uk$ or, equivalently, from maps $\eta_i: \tau_i \1_\uk \to \tau'_i \1_\uk$. It turns out that this rather simple map is just the map $\teta_i$ from Lemma \ref{lem:square} where we take $u=0$. Thus we already have a natural deformation for $\eta_i$ which leads to a natural deformation $\teta$ of $\eta$.

Now consider the truncation of $(\tau'_{\omega})^{2 \ell} Cone(\teta) \1_\uk$ to degrees $\ge -d(\ell)$. When restricted to $u=0$ this is homotopic to zero. It follows by Proposition \ref{contractlem} that the same is true of the whole truncation $(\tau'_{\omega})^{2 \ell} Cone(\teta) \1_\uk$. This implies that $\lim_{\ell \to \infty} (\tau'_{\omega})^{2 \ell}$ converges in ${\tKomm}(\dot{\Ucat}_Q^m)[u]$.
\end{proof}

Once we have this deformation $\tcP$ the construction from \cite{Cautis} can be repeated to yield a deformed homology for links $L$ labeled by arbitrary representations of $\sl_m$.

\begin{corollary} \label{cor:slm-higher-clasp}
Suppose $L = L_1 \cup \dots \cup L_c$ is an oriented link whose components are labeled by $V_{\mu_1}, \dots, V_{\mu_c}$ where $\mu_i = \sum_j \Lambda_{k_{i,j}}$. Then the construction above yields a family of link homologies $\widetilde{H}_{\sl_m}^{*,*}(L)$ parametrized by $\uz = (z_{i,j})$. This homology satisfies the usual/expected splitting properties. In particular, if $c=1$ and $\mu = \Lambda_{k_{1,1}} + \Lambda_{k_{1,2}}$ then restricting to $\uz$ where $z_{1,1} \ne z_{1,2}$ recovers the homology of $L$ labeled by the (reducible) representation $V_{\Lambda_{k_{1,1}}} \otimes V_{\Lambda_{k_{1,2}}}$.
\end{corollary}

\subsection{Comparison with \cite{CK4}} \label{sec:comparison}

Recall that in \cite{CaKa-sl2} and subsequent papers an algebro-geometric definition of $\sl_m$ link homology was developed. This construction was ``deformed'' in \cite{CK4} to yield a deformed theory with the same properties as the link homologies $\widetilde{H}_{\sl_m}^{*,*}(L)$ from Theorem \ref{thm:deflink}. Without going into too many details we would like to compare the construction of these two deformations.

In \cite{CK4} the role of $\dot{\Ucat}_Q^m$ is played by a 2-representation $\cK_{{\rm Gr},m}$. The objects in $\cK_{{\rm Gr},m}$ are the derived categories of coherent sheaves on certain convolution varieties $Y(\uk)$ associated to the affine Grassmannian of type $PGL_m$. These categories are $\Z$-graded with the grading shift denoted $\{1\}$. The 2-functor $\Phi: \dot{\Ucat}_Q \to \cK_{{\rm Gr},m}$ takes $\la 1 \ra \mapsto [1] \{-1\}$.

The source of the deformation in \cite{CK4} are certain deformations $\bY(\uk) \to \bA^{2N}$ of the varieties $Y(\uk)$ (these deformations are very natural from the perspective of the Beilinson-Drinfeld Grassmannian). Here points of $\bA^{2N}$ can be identified with sequences $\uy$. In general, a geometric deformation gives us degree two classes in the Hochschild cohomology of the variety. More precisely, this class is the product of the Atiyah and Kodaira-Spencer classes (c.f. \cite{HT}). In our case this yields a linear map
\begin{align}
\label{eq:phi} \C^{2N} & \to \End^2_{\cK_{{\rm Gr},m}}(\1_{\uk}) \\
\nonumber \uy & \mapsto B_\uy.
\end{align}

Now consider the image of a complex $\tau_i \1_\uk \in \dot{\Ucat}_Q$ under the 2-functor $\Phi$. This be identified with a kernel $\Phi(\tau_i \1_\uk)$ (a sheaf) living on $Y(s_i(\uk)) \times Y(\uk)$. The main (technical) result of \cite{CK4} is that $\Phi(\tau_i \1_\uk)$ deforms along
$$\{ (s_i(\uy),\uy): \uy \in \bA^{2N} \} = \bA^{2N} \subset \bA^{2N} \times \bA^{2N}.$$
The proof requires a more detailed geometric understanding of $\Phi(\tau_i \1_\uk)$. One should compare this result with Proposition \ref{prop:curved} which states that $\tau_i \1_\uk$ deforms along
$$\{ (s_i(b),b): b \in \End^2_{\dot{\Ucat}_Q^m}(\1_\uk) \} \subset \End^2_{\dot{\Ucat}_Q^m}(\1_{s_i(\uk)}) \times \End^2_{\dot{\Ucat}_Q^m}(\1_{\uk}).$$

The geometric deformation of $\Phi(\tau_i \1_\uk)$ has the property that its fiber over any $(s_i(\uy),\uy)$ with $y_i \ne y_{i+1}$ is the graph of an isomorphism
$$\bY(\uk)|_{\uy} \xrightarrow{\sim} \bY(s_i(\uk))|_{s_i(\uy)}.$$
This isomorphism is actually an involution (both these facts are easy to see). Hence, under the hypothesis that $y_i \ne y_{i+1}$, this deformation is isomorphic to its inverse. This result should be compared to Lemma \ref{lem:square}.

Another remark involves the formal indeterminate $u$. Recall that $u$ in $\dot{\Ucat}_Q$ has bi-degree $[2]\la -2 \ra$. Since the 2-functor $\Phi$ sends $\la 1 \ra \mapsto [1]\{-1\}$ and $[1] \mapsto [1]$ it follows that the image of $u$ has bi-degree $\{2\}$ in $\cK_{{\rm Gr},m}$. This agrees with the fact that the parameter space $\bA^{2N}$ of the deformation $\bY(\uk)$ is equipped with a dilating $\C^\times$ action of weight $\{2\}$.

In the undeformed case a procedure for comparing link homologies constructed via skew Howe duality was introduced in \cite[Section 7.5]{Cautis}. The rough idea is that any two link homologies constructed from a categorical action on a $2$-category $\cal{K}$ whose (nonzero) weight spaces can be identified with those of $\bigwedge_q^{m\infty}(\C^m \otimes \C^{2\infty})$ are automatically isomorphic. This can be used to show that the (underformed) link homologies defined using the affine Grassmannian~\cite{Cautis} must agree with (for example) the link homology $H_{\sl_m}^{*,*}(L)$ defined in Section \ref{sec:background1} using the 2-category $\dot{\Ucat}_Q^m$.

We expect the same is true in the deformed setting -- namely, that the deformed homology from \cite{CK4} is isomorphic to $\widetilde{H}_{\sl_m}^{*,*}(L)$. To prove this given the results mentioned above it remains to identify the deformation in \cite{CK4} and, more precisely, the induced map (\ref{eq:phi}) with the map
\begin{align*}
\C^{2N} & \to \End^2_{\dot{\Ucat}_Q^m}(\1_{\uk}) \\
\uy & \mapsto B_\uy
\end{align*}
from Section \ref{sec:defs1}. Although this is not terribly difficult, it is a bit technical and beyond the scope of the current paper.

\subsection{Example}

Consider the Hopf link $L$ where both components are labeled by the standard representation $V_{\Lambda_1}$ of $\mf{sl}_m$. We also suppose the components are labeled by deformation parameters $y_1$ and $y_2$.

\begin{equation} \label{hopfpic}
L \;\; = \;\;
\hackcenter{ \begin{tikzpicture} [scale=1.2]
\draw[thick] (.1,.9) .. controls ++(0,-.05) and ++(0,-.01) .. (0,1) .. controls ++(0,.3) and ++(0,.3) .. (-.5,1) to
    (-.5,0) .. controls ++(0,-.3) and ++(0,-.3) ..  (0,0)  ..controls ++(0,.2) and ++(0,-.2) .. (.5,.5)
        .. controls ++(0,.05) and ++(0,-.05) .. (.4,.7);
\draw[thick] (.1, .4) to (0,.5) .. controls ++(0,.2) and ++(0,-.2) .. (.5,1) .. controls ++(0,.3) and ++(0,.3) ..(1,1) to
        (1,0) .. controls ++(0,-.3) and ++(0,-.3) .. (.5,0) .. controls ++(0,.05) and ++(0,-.01) .. (.4,.1);
 \node at (.1,-.2) {$\scs 1$};
  \node at (.4,-.2) {$\scs 1$};
\end{tikzpicture}}
\end{equation}
Following the calculation in \cite[Section 10.3]{Cautis} we get that the complex for the Hopf link is homotopic to
\begin{equation*}
\oplus_{[m]} \oplus_{[m-1]} \Bbbk \langle -3 \rangle \underset{\underset{cu \Id}{\longleftarrow}}{\overset{0}{\longrightarrow}}
\oplus_{[m]} \oplus_{[m-1]} \Bbbk \langle -1 \rangle  \overset{f}{\longrightarrow} \oplus_{[m]} \oplus_{[m]} \Bbbk
\end{equation*}
where $f$ is an injective map and $c = y_1-y_2$.

Thus, if $c=0$, we get the following for the $\mf{sl}_m$ homology:
\[
\oplus_{j} \widetilde{H}_{\sl_m}^{i,j}(L) =
\left\{
  \begin{array}{ll}
    \oplus_{[m]} \Bbbk \langle m-1 \rangle & \hbox{ if $i=0$} \\
    \oplus_{[m][m-1]}  \Bbbk \langle -3 \rangle & \hbox{ if $i=-2$} \\
      0 & \hbox{ otherwise. }
  \end{array}
\right.
\]


If $ c \neq 0$, then the homology has rank $m^2$.  This is the rank of the homology of the unlink with two components which is what we expect from Lemma \ref{lem:square}. Compare these results with the analogous HOMFLYPT homology computation in \cite[Example 3.7]{GHy}.


\section{Application: deformed coloured HOMFLYPT homology} \label{sec-homfly}

In this section we continue to work with the choice of scalars $Q$ from Remark~\ref{weylactionstandard}.
\subsection{Background}\label{sec:back2}
We now explain an analogue of Section \ref{sec-slm} for coloured HOMFLYPT homology. Let us suppose that we have a link $L$ which is presented as the closure of a braid $\beta$ on $N$ strands. We colour the components of this link with arbitrary non-negative integers so that the bottom and top of $\beta$ is coloured $\uk = (k_1, \dots, k_N)$. We denote $d := \sum_i k_i$.

The starting point of our construction is the $2$-category $\Ucat_Q$ where the Cartan data is that of $\sl_{N}$ and the choice of scalars $Q$ has been fixed as in Remark~\ref{weylactionstandard}. As in Section \ref{sec-slm}, we identify the weights of $\Ucat_Q$ with sequences $\uk$ of integers such that $\sum_i k_i = d$. We will now define an integrable $2$-representation $\cK_N^d$ of $\Ucat_Q$.

First we fix some notation following \cite[Section 3.1]{Cautis-Rem}. For $k \ge 0$ denote
$$A_k := \Bbbk[x_1, \dots, x_k]^{S_k} = \Bbbk[\varepsilon_1, \dots, \varepsilon_k]$$
where the symmetric group $S_k$ acts naturally and where the $\varepsilon_i$ are elementary symmetric functions. These algebras are naturally $\Z$-graded with $x_i$ having degree $2$. The convention we use is that multiplication by $x_i$ induces a map $x_i: A_k \to A_k \{2\}$ where $\{\cdot\}$ denotes a shift in grading.

For a sequence $\uk$ we define
$$A_\uk := A_{k_1} \otimes_{\Bbbk} \dots \otimes_{\Bbbk} A_{k_N}.$$
We will denote the element $1 \otimes \dots \otimes \varepsilon_i \otimes \dots \otimes 1$ by $\varepsilon_i^{(j)}$ where the $\varepsilon_i$ occurs in the $j$th factor of $A_\uk$.

The (nonzero) objects in $\cK_N^d$ are indexed by sequences $\uk$ where all $k_i \ge 0$ and $\sum_i k_i = d$. The $1$-morphisms consist of $(A_{\uk}, A_{\uk'})$-bimodules which are flat as $A_{\uk}$ and also $A_{\uk'}$-modules. Composition of these $1$-morphisms is tensor product. The $2$-morphisms are then morphisms of bimodules.

Notice that we have natural inclusion maps
$$A_{(\dots,k_i+k_{i+1},\dots)} \to A_{(\dots,k_i,k_{i+1},\dots)}.$$
Subsequently, the algebra $A_{(\dots,k_i-1,1,k_{i+1},\dots)}$ is both an $A_\uk$-module as well as an $A_{\uk+\alpha_i}$-module.

\begin{proposition}\label{prop:gamma}
There exists a $2$-functor $\Gamma_N^d: \Ucat_Q \to \cK_N^d$ which sends $\la 1 \ra$ to $\{1\}$ and
\begin{align*}
\sE_i \1_{\uk} &\mapsto A_{(\dots,k_i-1,1,k_{i+1},\dots)} \{k_i-1\} \\
\1_{\uk} \sF_i &\mapsto A_{(\dots,k_i-1,1,k_{i+1},\dots)} \{k_{i+1}\} \\
\End_{\Ucat_Q}^2(\1_\uk) \ni b_j &\mapsto - \varepsilon_1^{(j)} + \varepsilon_1^{(j+1)} \in \End_{\cK_N^d}^2(\1_\uk).
\end{align*}
\end{proposition}
\begin{proof}
This $2$-functor is (up to a grading shift) the equivariant flag $2$-representation \cite[Section 6.3.3]{KL3}, see also \cite{Lau2} and \cite[Section 4.3]{MSV}. The image of bubbles in this $2$-representation are given in \cite[Equation 6.47]{KL3} or \cite[Equation 4.12]{MSV}.
\end{proof}

Next let us define the following shifted complexes
$$\hat{\tau}_i \1_\uk := \begin{cases}
\tau_i \1_\uk [-k_{i+1}] \{k_{i+1}+k_ik_{i+1}\} & \text{ if } \la \uk, \alpha_i \ra \le 0 \\
\tau_i \1_\uk [-k_i] \{k_i+k_ik_{i+1}\} & \text{ if } \la \uk, \alpha_i \ra \ge 0. \end{cases}$$
Using these shifted complexes is not crucial but it does simplify the grading/cohomological shifts in the long run and it does agree with the notation/definition from \cite{Cautis-Rem}.

Given a braid $\beta$, we denote by $\hat{\tau}_\beta$ the corresponding composition of $\hat{\tau}_i$. Then the $2$-functor $\Gamma_N^d$ from Proposition \ref{prop:gamma} gives us a complex of $(A_\uk,A_\uk)$-bimodules $\Gamma_N^d(\hat{\tau}_{\beta})$. Taking Hochschild cohomology of this complex of bimodules defines a triply graded homology $HH(L)$ (see \cite[Theorem 4.1]{Cautis-Rem}).

\subsection{Deformations}\label{sec:defs2}
To obtain the desired deformation of the construction above, we will label crossings with curved complexes. Notice that we have
$$\End^*_{\cK_N^d}(\1_{\uk}) = \Hom^*_{(A_\uk,A_\uk)}(A_\uk,A_\uk) = A_\uk.$$
Given $\uy = (y_1, \dots, y_N)$ we denote
$$B_\uy := \sum_i y_i \varepsilon_1^{(i)} \in \End^*_{\cK_N^d}(\1_{\uk}).$$

\begin{lemma}\label{lem:curved}
The complex $\Gamma_N^d(\tau_{i,-y_i+y_{i+1}} \1_{\uk})$ is a $(B_{s_i(\uy)}, B_\uy)$ factorization inside $\tKom(\cK_N^d)[u]$.
\end{lemma}
\begin{proof}
As in Section \ref{sec:defs1}, we define $b_\uy \in \End^2_{\dot{\Ucat}_Q}(\1_\uk)$ as the linear combination of bubbles determined by the relations $(b_\uy, \alpha_i) = -y_i+y_{i+1}$. Then we know that $\tau_{i,-y_i+y_{i+1}} \1_{\uk}$ is a $(b_{s_i(\uy)}, b_\uy)$ factorization.

On the other hand, the third relation in Proposition \ref{prop:gamma} implies that $\Gamma_N^d(b_\uy) = B_\uy$ if $\sum_i y_i = 0$. It follows that $\Gamma_N^d(\tau_{i,-y_i+y_{i+1}} \1_{\uk})$ is a $(B_{s_i(\uy)}, B_\uy)$ factorization if $\sum_i y_i = 0$.

More generally, we can write any $\uy$ as $\uy' + (c,\dots,c)$ where $\sum_i y_i' = 0$. Then $B_\uy = B_{\uy'}$ and so, by the above, it remains to show that
$$\Id_{\1_{s_i(\uk)}} \Id_{\Gamma_N^d(\tau_{i,-y_i+y_{i+1}})} B_{(c,\dots,c)} = B_{(c,\dots,c)} \Id_{\Gamma_N^d(\tau_{i,-y_i+y_{i+1}})} \Id_{\1_{\uk}} \in \End_{\cK_N^d}^2(\1_{s_i(\uk)} \Gamma_N^d(\tau_{i,-y_i+y_{i+1}}) \1_{\uk}).$$
But this is clear since it is not difficult to check that $B_{(c,\dots,c)}$ commutes with all $\sE_i$ and $\sF_i$ in $\cK_N^d$.
\end{proof}

Repeating the construction from Section \ref{sec:back2} we start with a braid $\beta$ and apply Lemma \ref{lem:curved} repeatedly to obtain a $(B_{\beta(\uy)}, B_\uy)$-factorization $\Gamma_N^d(\tau'_{\beta})$. To finish, we apply the functor $\Hom^*_{(A_\uk,A_\uk)}(A_\uk, \cdot)$. Since elements of $\End^2_{(A_\uk,A_\uk)}(A_\uk)$ (such as $B_\uy$) commute with any $1$-morphism of $(A_\uk,A_\uk)$-bimodules, it follows that the resulting curved complex has curvature $B_{\beta(\uy)} - B_\uy$. Thus, if we choose $\uy$ so that $\uy = \beta(\uy)$, then we get an actual complex (with zero curvature). We denote the resulting triply graded vector space ${\widetilde{HH}}(L)$.

\begin{theorem}\label{thm:defHOMFLY}
Suppose $L = L_1 \cup \dots \cup L_c$ is an oriented, coloured link with $c$ components. Then the construction above defines a family of link homologies $\widetilde{HH}(L)$ parametrized by $\uz \in \Bbbk^c$. This homology has the same splitting properties as $\widetilde{H}_{\sl_m}^{*,*}(L)$ from Theorem \ref{thm:deflink}. In particular, for distinct values of $\uz$, the deformed homology of $L$ is isomorphic to the homology of the disjoint union of its components, and moreover, there exists a spectral sequence starting with $HH(L)$ and converging to $\widetilde{HH}(L)$.
\end{theorem}
\begin{proof}
We just need to verify that $\widetilde{HH}(L)$ is indeed a link invariant (the rest follows from Lemma \ref{lem:square} as it does in the proof of Theorem \ref{thm:deflink}). The fact that it is a link invariant can be verified as in the proof of \cite[Theorem 4.1]{Cautis} where the analogous result is proved for the undeformed complex $HH(L)$.
\end{proof}

\subsection{Clasps}\label{sec:clasps2}

In \cite{Cautis-Rem} clasps were constructed (just like in \cite{Cautis}) as a limit of twists. This allowed one to define a HOMFLYPT homology for links coloured by arbitrary partitions. The discussion from Section \ref{sec:clasps1} now repeats word for word to give us a deformation of these clasps in the context of HOMFLYPT homology. This leads to the following analogue of Corollary \ref{cor:slm-higher-clasp}.

\begin{corollary} \label{cor:HOMFLY-higher-clasp}
Suppose $L = L_1 \cup \dots \cup L_c$ is an oriented link whose components are coloured by partitions $(\mu_1, \dots, \mu_c)$ where $\mu_i = (\mu_i^{(1)} \ge \dots \ge \mu_i^{(s_i)})$ is the decomposition into parts. Then the constructions above yield a family of link homologies $\widetilde{HH}(L)$ parametrized by $\uz = (z_{i,1}, \dots, z_{i,s_i})$. This homology satisfies the usual/expected splitting properties. In particular, if $c=1$ and $\mu_1 = (\mu_1^{(1)} \ge  \mu_1^{(2)})$ then restricting to $\uz$ where $z_{1,1} \ne z_{1,2}$ recovers the homology of two unlinked copies of $L$ coloured by the partitions $\mu_1^{(1)}$ and $\mu_2^{(2)}$.
\end{corollary}


\providecommand{\bysame}{\leavevmode\hbox to3em{\hrulefill}\thinspace}
\providecommand{\MR}{\relax\ifhmode\unskip\space\fi MR }
\providecommand{\MRhref}[2]{%
  \href{http://www.ams.org/mathscinet-getitem?mr=#1}{#2}
}
\providecommand{\href}[2]{#2}

%
%

\end{document}